\documentclass{amsart}
\usepackage{url}
\usepackage{graphicx}
\usepackage{amsmath}
\usepackage{amsthm}
\usepackage{amssymb}
\usepackage{verbatim}
\usepackage{stmaryrd}
\usepackage{mathtools}
\usepackage[cmtip,all]{xy}
\usepackage{xcolor}
\usepackage{enumitem}
\usepackage{mathtools}
\usepackage{bm}
\usepackage{tikz-cd}
\usetikzlibrary{matrix,arrows}
\usetikzlibrary{cd}

\usepackage{hyperref}

\newtheorem{theorem}{Theorem}[section]

\newtheorem{lemma}[theorem]{Lemma}
\newtheorem{proposition}[theorem]{Proposition}
\newtheorem{corollary}[theorem]{Corollary}

\theoremstyle{definition}
\newtheorem{definition}[theorem]{Definition}

\theoremstyle{remark}
\newtheorem{remark}[theorem]{Remark}

\numberwithin{equation}{section}

\newcommand{\bp}{\begin{pmatrix}}
\newcommand{\ep}{\end{pmatrix}}
\newcommand{\mf}{\mathfrak}
\newcommand{\mc}{\mathcal}
\newcommand{\zg}{\mc{Z}(\mf{g})}
\newcommand{\zl}{\mc{Z}(\ell_\Theta)}

\newcommand{\ug}{\mc{U}(\mf{g})}
\newcommand{\un}{\mc{U}(\mf{n})}
\newcommand{\h}{\mf{h}^{\Theta}}
\newcommand{\hs}{\mf{h}^{\Theta *}}
\newcommand{\ul}{\mc{U}(\ell_\Theta)}
\newcommand{\uu}{\mc{U}(\bar{\mf{u}}_\Theta)}
\newcommand{\otl}{\Omega_{\Theta, \lambda}}
\newcommand{\uh}{\mc{U}(\mf{h})}
\newcommand{\HH}{\mathcal{H}_\Theta}
\newcommand{\C}{\mathbb{C}}

\newcommand{\Z}{\mathbb{Z}}

\newcommand{\N}{\mathcal{N}}
\newcommand{\M}{\mathcal{M}}
\newcommand{\D}{\mathcal{D}}
\newcommand{\DD}{\mathbb{D}}

\newcommand{\I}{\mathcal{I}}

\newcommand{\Nte}{\mathcal{N}_{\theta, \eta}}
\newcommand{\W}{W_\Theta \backslash W}
\newcommand{\Zq}{\mathbb{Z}[q,q^{-1}]}
\newcommand{\MDNe}{\mathcal{M}_{coh}(\mathcal{D}_X,N,\eta)}

\DeclareMathOperator{\Hom}{Hom}

\DeclareMathOperator{\ch}{ch}
\DeclareMathOperator{\Int}{Int}
\DeclareMathOperator{\im}{im}
\DeclareMathOperator{\End}{End}

\begin{document}

\title{A Kazhdan-Lusztig algorithm for Whittaker modules} 

\author{Anna Romanov}
\maketitle

\begin{abstract}
We study a category of Whittaker modules over a complex semisimple Lie algebra by realizing it as a category of twisted $\mc{D}$-modules on the associated flag variety using Beilinson--Bernstein localization. The main result of this paper is the development of a geometric algorithm for computing the composition multiplicities of standard Whittaker modules. This algorithm establishes that these multiplicities are determined by a collection of polynomials we refer to as Whittaker Kazhdan--Lusztig polynomials. In the case of trivial nilpotent character, this algorithm specializes to the usual algorithm for computing multiplicities of composition factors of Verma modules using Kazhdan--Lusztig polynomials. 
\end{abstract}
\tableofcontents

\section{Introduction}

A fundamental goal in representation theory is to understand all representations of complex semisimple Lie algebras. However, the category of all modules for a given Lie algebra is so large that a full classification has only been obtained for the simplest example, the Lie algebra $\mf{sl}(2,\C)$ \cite{Block}. In light of this, one way to approach this goal is to study well-behaved categories of representations subject to certain restrictions, then relax the restrictions to expand the categories and observe what aspects of the structure carry over into the larger category. A classic example of such a well-behaved category is Bernstein-Gelfand--Gelfand's category $\mc{O}$, which has been studied extensively in the past 40 years and found to display deep connections across representation theory. The category $\mc{N}$ of Whittaker modules introduced by Mili\v{c}i\'{c}--Soergel in \cite{CompositionSeries} is a generalization of category $\mc{O}$ which also contains a collection of nondegenerate Whittaker modules introduced by Kostant \cite{Kostant}. In category $\mc{O}$, characters of simple modules are determined by Kazhdan--Lusztig polynomials. In this paper, we show that the same is true in the category of Whittaker modules, and we develop an algorithm for computing these characters. The main result of this paper is the following theorem. 

\begin{theorem}
\label{theorem1} (Theorem \ref{KLalgorithm}, Corollary \ref{multiplicity integral character}, equation (\ref{KLinversionWhittakergVerma}))
For any irreducible Whittaker module $L$ and standard Whittaker module $M$ with the same regular integral infinitesimal character, there exists a polynomial $Q_{ML} \in q\Z[q]\cup \{1 \} $ such that the multiplicity of $L$ in the composition series of $M$ is given by $Q_{ML}(-1)$. Moreover, the polynomials $Q_{ML}$ can be computed through a combinatorial recursive algorithm.
\end{theorem}

Our approach to studying Whittaker modules is to use the localization of Beilinson--Bernstein \cite{BB} to relate $\mc{N}$ to a certain category of holonomic $\mc{D}$-modules (so-called twisted Harish-Chandra sheaves) on the associated flag variety. This geometric approach gives us access to powerful tools such as the decomposition theorem for arbitrary holonomic $\mc{D}$-modules \cite{Mochizuki} which are essential in the development of the algorithm for computing the polynomials of Theorem \ref{theorem1}. 

The four main contributions of this paper to the existing literature on Whittaker modules are the following. First, we develop a theory of formal characters for Whittaker modules which generalizes the theory of formal characters of highest weight modules and distinguishes isomorphism classes of objects in the Grothendieck group of the category (Section \ref{Character theory}). Second, we give a detailed description of the structure of the category of twisted Harish-Chandra sheaves (Section \ref{A category of twisted sheaves}). Irreducible objects in this category were classified in \cite{TwistedSheaves}, but this paper includes a collection of new results describing the action of intertwining functors on certain costandard sheaves, which were originally introduced by Mili\v{c}i\'{c}--Soergel in \cite{TwistedSheaves}. The third and most significant contribution of the current paper is the development of an algorithm for computing the composition multiplicities of standard Whittaker modules, which establishes that the formal characters of simple Whittaker modules are given by a collection of polynomials that we refer to as Whittaker Kazhdan--Lusztig polynomials (Section \ref{A Kazhdan-Lusztig algorithm}). Finally, we give a comparison of the Whittaker Kazhdan--Lusztig polynomials which arise in our algorithm to other types of Kazhdan--Lusztig polynomials in the existing literature (Section \ref{Whittaker Kazhdan-Lusztig polynomials}). This places Theorem \ref{theorem1} in the context of the Kazhdan--Lusztig combinatorics of the Hecke algebra and establishes a connection between Whittaker modules and other representation theoretic objects such as generalized Verma modules.


We will spend the rest of the introduction describing the main results of this paper in more detail. Let $\ug$ be the universal enveloping algebra of a semisimple Lie algebra $\mf{g}$ over $\C$ and $\zg$ the center of $\ug$. Let $\mf{b}$ be a fixed Borel subalgebra of $\mf{g}$ with nilpotent radical $\mf{n}=[\mf{b},\mf{b}]$ and $\mf{h}\subset \mf{b}$ a Cartan subalgebra. The category $\mc{N}$ of Whittaker modules consists of all $\ug$-modules which are finitely generated, $\zg$-finite, and $\un$-finite. For a choice of $\lambda \in \mf{h}^*$ and a Lie algebra morphism $\eta:\mf{n}\rightarrow \C$, McDowell \cite{McDowell} constructed a standard Whittaker module $M(\lambda, \eta)$ (Definition \ref{standard Whittaker module}), which has a unique irreducible quotient $L(\lambda, \eta)$, and showed that all irreducible Whittaker modules appear as such quotients. When $\eta=0$, the $M(\lambda, 0)$ are Verma modules and the $L(\lambda, 0)$ are simple highest weight modules. When $\eta$ acts non-trivially on all root subspaces of $\mf{g}$ corresponding to simple roots (we say such $\eta$ are {\em nondegenerate}), the $M(\lambda, \eta)$ are the irreducible modules studied by Kostant in \cite{Kostant}. 

Unlike highest weight modules, Whittaker modules don't decompose into generalized $\mf{h}$-weight spaces. However, in blocks of $\mc{N}$ where the nilpotent radical $\mf{n}$ acts by a specific character $\eta$, Whittaker modules do decompose into generalized weight spaces for a certain subalgebra $\mf{h}^\Theta \subset \mf{h}$, which is the center of a Levi subalgebra of $\mf{g}$ determined by the character $\eta$ (Section \ref{Standard and simple modules}). In contrast to the generalized $\mf{h}$-weight spaces of category $\mc{O}$, the generalized $\mf{h}^\Theta$-weight spaces in this decomposition are not finite-dimensional, but they are of finite length in the category of modules over the specified Levi subalgebra. We can capture the structure of this $\mf{h}^\Theta$-weight space decomposition by defining the {\em formal character} (Definition \ref{formal character}) of a Whittaker module in a way that generalizes the formal character of highest weight modules. Then a natural problem in understanding the structure of the category of Whittaker modules is to compute the formal characters of irreducible modules in $\mc{N}$, which reduces to computing the multiplicities of the irreducible constituents of a standard Whittaker module.

These multiplicities were first determined for integral $\lambda$ by Mili\v{c}i\'{c} and Soergel in \cite{CompositionSeries} and for arbitrary $\lambda$ by Backelin in \cite{Backelin} by relating subcategories of Whittaker modules to certain blocks of category $\mc{O}$ and using the classical Kazhdan--Lusztig algorithm for Verma modules.  The current paper provides a more efficient procedure for calculating these multiplicities by using a geometric realization of Whittaker modules as twisted sheaves of $\mc{D}$-modules on the flag variety. This geometric perspective allows us to relate the multiplicities to combinatorial data extracted from the associated Hecke algebra, providing a direct link between Whittaker modules and Kazhdan--Lusztig polynomials. 

The first step in studying Whittaker modules geometrically is to realize $\mc{N}$ as a category of twisted Harish-Chandra modules. Let $N$ be the unipotent subgroup of $\Int \mf{g}$ such that Lie$N=\mf{n}$. For a Lie algebra morphism $\eta:\mf{n}\rightarrow \C$, the category of $\eta$-twisted Harish-Chandra modules consists of $\mf{g}$-modules which admit an algebraic action of $N$ whose differential differs from the restricted $\mf{g}$-action by $\eta$.  We denote the category of such modules with infinitesimal character corresponding to a Weyl-group orbit $\theta \subset \mf{h}^*$ (via the Harish-Chandra homomorphism) by $\mc{M}_{fg}(\mc{U}_\theta, N, \eta)$. In \cite{TwistedSheaves}, Mili\v{c}i\'{c} and Soergel established a categorical equivalence between certain blocks of $\mc{N}$ and the categories $\mc{M}_{fg}(\mc{U}_\theta, N, \eta)$. 

This description allows us to use the localization theory of Beilinson--Bernstein to study Whittaker modules. For each $\lambda \in \mf{h}^*$, Beilinson and Bernstein \cite{BB} constructed a sheaf of twisted differential operators $\mc{D}_\lambda$ on the flag variety $X$ of $\mf{g}$ whose global sections $\Gamma(X, \mc{D}_\lambda)$ are equal to $\mc{U}_\theta$, where $\theta$ is the Weyl group orbit of $\lambda$ in $\mf{h}^*$ and $\mc{U}_\theta$ is the quotient of $\ug$ by the corresponding ideal in $\zg$. Applying the localization functor $\Delta_\lambda=\mc{D}_\lambda \otimes_{\mc{U}_\theta}-$ to the category $\mc{M}_{fg}(\mc{U}_\theta, N, \eta)$, we obtain a geometric category $\mc{M}_{coh}(\mc{D}_\lambda, N, \eta)$ of {\em $\eta$-twisted Harish-Chandra sheaves} (Section \ref{A category of twisted sheaves}), which are $N$-equivariant $\mc{D}_\lambda$-modules satisfying a compatibility condition determined by $\eta$. This category consists of holonomic $\mc{D}_\lambda$-modules, so its objects have finite length and there is a well-defined duality in the category. The morphism $\eta$ determines a parabolic subgroup $W_\Theta$ of the Weyl group $W$ of $\mf{g}$, and from the parameters $\eta:\mf{n}\rightarrow \C$, $C \in W_\Theta \backslash W$, and $\lambda \in \mf{h}^*$, we construct a standard sheaf $\mc{I}(w^C, \lambda, \eta)$, costandard sheaf $\mc{M}(w^C, \lambda, \eta)$, and irreducible sheaf $\mc{L}(w^C, \lambda, \eta)$ (Section \ref{A category of twisted sheaves}). Here $w^C$ is the longest element in the coset $C$. The precise relationship between the algebraic category $\mc{N}$ and the geometric category $\M_{coh}(\D_\lambda, N, \eta)$ is given by the following theorem, which we prove in Section \ref{Geometric description of Whittaker modules}. 
\begin{theorem}
\label{translation} (Theorem \ref{global sections of costandards}, Theorem \ref{global sections of irreducibles})  Let $\lambda \in \mf{h}^*$, $\eta:\mf{n}\rightarrow \C$ a Lie algebra morphism, and $C \in W_\Theta \backslash W$. Let $\mc{M}(w^C, \lambda, \eta)$ be the corresponding costandard $\eta$-twisted Harish-Chandra sheaf and $M(w^C\lambda, \eta)$ the corresponding standard Whittaker module. Then 
\begin{enumerate}[label=(\roman*)] 
\item if $\lambda$ is antidominant,
\[ 
\Gamma(X, \mc{M}(w^C, \lambda, \eta)) = M(w^C \lambda, \eta), \text{ and }
\] 
\item if $\lambda$  is also regular, then 
\[
\Gamma(X, \mc{L}(w^C, \lambda, \eta)) = L(w^C\lambda, \eta). 
\]
\end{enumerate}
\end{theorem}
Hence to compute the composition multiplicities of standard Whittaker modules $M(\lambda, \eta)$, it suffices to compute the composition multiplicities of the costandard $\eta$-twisted Harish-Chandra sheaves $\mc{M}(w^C, \lambda, \eta)$ in the category $\mc{M}_{coh}(\mc{D}_\lambda, N, \eta)$.  In the case of regular integral $\lambda \in \mf{h}^*$, the structure of this category is completely determined by the parameter $\eta$, so we may further restrict our attention to the case $\lambda = - \rho$, where $\rho$ is the half-sum of positive roots. In this setting, $\mc{D}_\lambda = \mc{D}_X$ is the sheaf of differential operators on $X$ (without a twist). One way to better understand the structure of the irreducible $\mc{D}_X$-modules $\mc{L}(w^C, -\rho, \eta)$ (or indeed any $\mc{D}_X$-module in this category) is to utilize the stratification of the flag variety and to restrict them to Bruhat cells contained in their support. The resulting restricted $\mc{D}$-modules are easy to understand: the $N$-equivariance guarantees that they decompose into a direct sum of copies of the structure sheaf on the corresponding Bruhat cell. By keeping track of how many copies appear in the direct sum corresponding to each Bruhat cell (we refer to this integer as the ``$\mc{O}$-dimension,'' denoted $\dim_\mc{O}$), we can construct a combinatorial object which captures all important structural information of each irreducible $\mc{D}_X$-module in the category  $\mc{M}_{coh}(\mc{D}_X, N, \eta)$. For each coset $D \in \W$, let $\delta_D$ be a formal variable parameterized by $D$, and let $\mathcal{H}_\Theta$ be the free $\Z[q,q^{-1}]$-module with basis $\{\delta_D$, $D \in W_\Theta \backslash W\}$. Let $i_{w^D}:C(w^D) \rightarrow X$ be the inclusion of the corresponding Bruhat cell into the flag variety. We define a map $\nu:\mc{M}_{coh}(\mc{D}_X,N,\eta) \rightarrow \mc{H}_\Theta$ by 
\[
\nu(\mc{F})=\sum_{D \in \W} \sum_{m \in \Z} \dim_\mc{O} (R^m i^!_{w^D}(\mc{F}))q^m \delta_D .
\]
Here, $R^mi_{w^D}^!$ are the right derived functors of the $\mc{D}_X$-module extraordinary inverse image functor (Section \ref{Modules over twisted sheaves of differential operators}).

We use $\nu$ to develop our desired Kazhdan--Lusztig algorithm for Whittaker modules. Let $\Sigma$ be the root system of $\mf{g}$ and $\Pi \subset \Sigma$ the set of simple roots determined by our fixed $\mf{b}$. Let $\Theta \subset \Pi$ be the subset of simple roots picked out by $\eta\in \ch{\mf{n}}$, and let $W_\Theta \subset W$ be the corresponding parabolic subgroup of the Weyl group. For any $\alpha \in \Pi$, we define a certain $\Z[q,q^{-1}]$-module endomorphism $T_\alpha: \mc{H}_\Theta \rightarrow \mc{H}_\Theta$ (Section \ref{A Kazhdan-Lusztig algorithm}). The main result of this paper is the following theorem. 
\begin{theorem} 
\label{KLalgorithmintro}
(Theorem \ref{KLalgorithm}, Proposition \ref{existence}) The function $\varphi:W_\Theta \backslash W \rightarrow \mathcal{H}_\Theta$ given by $\varphi(C) = \nu(\mc{L}(w^C, - \rho, \eta))$ is the unique function satisfying the following properties. 
\begin{enumerate}[label=(\roman*)]
\item{For $C \in W_\Theta \backslash W$, 
\[
\varphi(C) = \delta_C + \sum_{D<C}P_{CD} \delta_D, 
\]
where $P_{CD} \in q\Z[q]$.}
\item{For $\alpha \in \Pi$ and $C \in \W$ such that $Cs_\alpha < C$, there exist $c_D \in \Z$ such that 
\[
T_\alpha(\varphi(Cs_\alpha))= \sum_{D\leq C} c_D\varphi(D).
\]}
\end{enumerate}
\end{theorem}
The existence and uniqueness of a function satisfying equivalent conditions to (i) and (ii) was shown combinatorially by Soergel in \cite{Soergel97}\footnote{The formulation in \cite{Soergel97} is in terms of the antispherical module of the Hecke algebra. We prove in Section \ref{Recursion is equivalent to duality} that this formulation is equivalent to conditions (i) and (ii) in Theorem \ref{KLalgorithmintro}.}. By realizing the function $\varphi$ explicitly in terms of the category $\mc{M}_{coh}(\mc{D}_\lambda, N, \eta)$, Theorem \ref{KLalgorithmintro} relates the Hecke algebra combinatorics established in \cite{Soergel97} to the category of Whittaker modules, which is the main accomplishment of this paper.  Theorem \ref{KLalgorithmintro} determines a family $\{P_{CD}\}$ of polynomials in $q\Z[q]$ parameterized by pairs of cosets $C,D \in \W$. We refer to these as {\em Whittaker Kazhdan--Lusztig polynomials}. In Section \ref{Whittaker Kazhdan-Lusztig polynomials} we describe their relationship to other types of Kazhdan--Lusztig polynomials appearing in the literature. These polynomials determine the composition multiplicities of standard Whittaker modules. More precisely, if $(\mu_{CD})_{C,D \in \W}$ is the inverse of the lower-triangular matrix $(P_{CD}(-1))_{C,D \in \W}$, then we have the following corollary to Theorem \ref{KLalgorithmintro}. 

\begin{corollary}
(Corollary \ref{multiplicity of irreducible in standard}, Corollary \ref{multiplicity integral character})
Let $\lambda \in \mf{h}^*$ be regular, integral, and antidominant. Then the multiplicity of the irreducible Whittaker module $L(w^D(\lambda - \rho), \eta)$ in the standard Whittaker module $M(w^C(\lambda - \rho), \eta)$ is $\mu_{CD}$. 
\end{corollary}

This paper is organized in the following way. We start by describing the structure of the algebraic category of Whittaker modules in Section \ref{A category of Whittaker modules}, following \cite{McDowell}. In this section we recall McDowell's construction of standard and simple Whittaker modules and develop a new theory of formal characters for Whittaker modules. In Section \ref{A category of twisted sheaves}, we describe the category of twisted Harish-Chandra sheaves, following \cite{TwistedSheaves}. We recall Mili\v{c}i\'{c}--Soergel's construction of standard and simple objects in this category, then introduce a class of costandard objects. These costandard objects were mentioned in \cite{TwistedSheaves} but not explicitly defined or studied. We prove some results about the action of intertwining functors on these costandard objects which are necessary for our arguments in Section \ref{Geometric description of Whittaker modules}. We dedicate Section \ref{Geometric description of Whittaker modules} to explicitly relating the category $\mc{N}$ of Whittaker modules and the category $\mc{M}_{coh}(\mc{D}_\lambda, N, \eta)$ of twisted Harish-Chandra sheaves by proving that the global sections of costandard twisted Harish-Chandra sheaves are standard Whittaker modules. This result sets us up to work completely in the geometric category. Section \ref{A Kazhdan-Lusztig algorithm} contains the proof of Theorem \ref{KLalgorithmintro}, which is the main result of this paper. In Section \ref{Whittaker Kazhdan-Lusztig polynomials} we determine the relationship between Whittaker Kazhdan--Lusztig polynomials and Kazhdan--Lusztig polynomials, and we describe a combinatorial duality between the Kazhdan--Lusztig algorithm for generalized Verma modules found in \cite{localization} and the Kazhdan--Lusztig algorithm for Whittaker modules established in this paper. In Appendix \ref{Geometric preliminaries}, we record our geometric conventions and include some fundamental facts about modules over twisted sheaves of differential operators.

\subsection{Acknowledgements} This project grew from the work of my PhD advisor Dragan Mili\v{c}i\'{c} and his collaborator Wolfgang Soergel, and it owes its existence to the detailed mathematical foundation established in their joint work. I am grateful for the mentorship and guidance provided by Dragan Mili\v{c}i\'{c} that led to this paper. I also thank Peter Trapa for suggesting that I introduce the Hecke algebra into this story, and I thank Geordie Williamson for directing me to the combinatorial existence argument in \cite{Soergel97} and pointing out the connection with the antispherical category. I thank Emily Cliff and Christopher Leonard for providing very useful feedback on preliminary versions of this paper.


\section{A category of Whittaker modules}
\label{A category of Whittaker modules}
In this section, we introduce the category of representations which is the main focus of this paper and describe some key aspects of its structure. Let $\mf{g}$ be a complex semisimple Lie algebra, $\ug$ its universal enveloping algebra, and $\zg$ the center of $\ug$. Let $\mf{b}$ be a Borel subalgebra with nilpotent radical $\mf{n}=[\mf{b},\mf{b}]$ and $\mf{h}$ the (abstract) Cartan subalgebra of $\mf{g}$ \cite[\S2]{D-modules}. Let $\Pi \subset \Sigma^+ \subset \Sigma\subset \mf{h}^*$ be the corresponding set of simple roots and positive roots, respectively, inside the root system of $\mf{g}$. Let $W$ be the Weyl group of $\mf{g}$, and denote by $\rho \in \mf{h}^*$ the half-sum of positive roots. 

We begin by recalling some standard terminology. For a $W$-orbit $\theta \subset \mf{h}^*$, there is a unique maximal ideal $J_\theta \subset \zg$, which can be obtained as the kernel of the Lie algebra morphism $\chi_\lambda:\zg \rightarrow \C$ defined by $z \mapsto (\lambda - \rho) (\gamma(z))$, where $\gamma:\zg \rightarrow \uh$ is the untwisted Harish-Chandra homomorphism and $\lambda$ is an element of the $W$-orbit $\theta$ \cite[Ch. 1 \S9]{BGGcatO}. All $\lambda \in \theta$ result in the same homomorphism $\chi_\lambda$. We call such a Lie algebra morphism $\chi_\lambda$ an {\em infinitesimal character}.  We say a $\mf{g}$-module $V$ {\em has infinitesimal character} if it has the property that there exists an infinitesimal character $\chi_\lambda$ such that for any $z \in \zg$ and $v \in V$, $zv = \chi_\lambda(z)v$, or, equivalently, if it is annihilated by the ideal $J_\theta$. We say a $\mf{g}$-module {\em has generalized infinitesimal character} if there exists an infinitesimal character $\chi_\lambda$ and $k \in \mathbb{N}$ such that for all $v \in V$ and $z \in \zg$, $(z - \chi_\lambda(z))^kv=0$, or, equivalently, if it is annihilated by a power of the ideal $J_\theta$. 

We are interested in the following category of $\mf{g}$-modules, which was originally introduced by Mili\v{c}i\'{c} and Soergel in \cite{CompositionSeries}.

\begin{definition} 
\label{Whittaker module}
Let $\mc{N}$ be the category of $\mf{g}$-modules which are 
\begin{enumerate}[label=(\roman*)]
\item finitely generated as $\ug$-modules,
\item $\zg$-finite, and 
\item $\un$-finite. 
\end{enumerate} 
We refer to objects in this category as \emph{Whittaker modules}.
\end{definition}

\begin{remark}
In Kostant's paper \cite{Kostant} the term {\em Whittaker module} is used to describe any $\mf{g}$-module that is cyclically generated by a Whittaker vector. (These are vectors where $\mf{n}$ acts by a nondegenerate Lie algebra morphism $\eta: \mf{n}\rightarrow \C$.) We note that Definition \ref{Whittaker module} differs from Kostant's original terminology, though all irreducible Whittaker modules (in the sense of Kostant) are contained in $\mc{N}$. 
\end{remark}

McDowell showed that all objects in $\mc{N}$ have finite length \cite{McDowell} (a fact which follows immediately from their description as holonomic $\mc{D}$-modules in \cite{TwistedSheaves}). This category is a natural generalization of Bernstein--Gelfand--Gelfand's category $\mc{O}$. Indeed, if condition (ii) is replaced by the stronger condition that $\mf{h}$ acts semisimply on the module, the resulting category is exactly category $\mc{O}$ \cite{BGGcatO}, so $\mc{O}$ is a full subcategory of $\mc{N}$. A key difference between $\mc{N}$ and $\mc{O}$ is that when the $\mf{h}$-semisimplicity condition is relaxed to $\zg$-finiteness, the existence of weight space decompositions is lost. However, the finiteness conditions (ii) and (iii) provide us with other useful decompositions of $\mc{N}$ which lead to structural results reminiscent of those in category $\mc{O}$.  In particular, we have two categorical decompositions \cite[\S2 Lem. 2.1, Lem. 2.2]{TwistedSheaves}: 
\[
\mc{N} = \bigoplus_{\theta \in W \backslash \mf{h}^*} \mc{N}_{\hat{\theta}} \text{ and } \mc{N} = \bigoplus_{\eta \in \mf{n}^*} \mc{N}_\eta. 
\]
Here $\mc{N}_{\hat{\theta}}$ is the full subcategory of $\mc{N}$ consisting of modules with generalized infinitesimal character $\chi_\lambda$ for $\lambda \in \theta$, and $\mc{N}_\eta$ is the full subcategory of $\mc{N}$ consisting of modules where for any $X \in \mf{n}$, $X-\eta(X)$ acts locally nilpotently on $V$. The only elements $\eta \in \mf{n}^*$ for which $\mc{N}_\eta\neq 0$ are Lie algebra morphisms \cite[Ch. VII \S1.3 Prop. 9(iii)]{Bourbaki}. We call such a Lie algebra morphism $\eta:\mf{n}\rightarrow \C$ an {\em $\mf{n}$-character} and say that modules in $\mc{N}_\eta$ have {\em generalized $\mf{n}$-character $\eta$}. We denote by $\ch{\mf{n}}\subset \mf{n}^*$ the set of $\mf{n}$-characters. 

Let $\mc{N}_\theta$ be the full subcategory of $\mc{N}$ consisting of modules with infinitesimal character $\chi_\lambda$ for $\lambda \in \theta$, and let $\mc{N}_{\theta, \eta}$ be the intersection $\mc{N}_\theta \cap \mc{N}_\eta$. Any irreducible Whittaker module lies in $\mc{N}_{\theta, \eta}$ for some Weyl group orbit $\theta$ and some $\eta \in \ch{\mf{n}}$, so we will often restrict our attention to this full subcategory $\Nte$. 

The category $\mc{N}_{\theta, \eta}$ is equivalent to a certain category of $\eta$-twisted Harish-Chandra modules, which is easier to relate to the geometric categories which appear later in this paper. We describe this equivalence now. Let $N$ be the unipotent subgroup of $\Int \mf{g}$ such that Lie$N=\mf{n}$. Because $N$ acts on the flag variety $X$ of $\mf{g}$ with finitely many orbits, the pair $(\mf{g}, N)$ is a Harish-Chandra pair in the sense of \cite[\S1]{TwistedSheaves}. For a fixed $\mf{n}$-character $\eta \in \ch{\mf{n}}$, denote by $\mc{M}_{fg}(\mf{g},N,\eta)$ the category of triples $(\pi, \nu, V)$ such that:
\begin{enumerate}[label=(\roman*)]  
\item{$(\pi, V)$ is a finitely generated $\ug$-module,}
\item{$(\nu, V)$ is an algebraic representation of $N$, and}
\item{the differential of the $N$-action on $V$ induces a $\mc{U}(\mf{n})$-module structure on $V$ such that for any $\xi \in \mf{n}$,
\[
\pi (\xi)=d\nu(\xi)+\eta(\xi).
\]}
\end{enumerate}
This is the category of {\em $\eta$-twisted Harish-Chandra modules} for the Harish-Chandra pair $(\mf{g},N)$. Let $\mc{M}_{fg}(\mc{U}_\theta,N,\eta)$ be the full subcategory of $\mc{M}_{fg}(\mf{g},N,\eta)$ consisting of modules which are also $\mc{U}_\theta =\ug / (\ug J_\theta)$-modules; that is, modules $V \in \mc{M}_{fg}(\mf{g},N,\eta)$ which are annihilated by $J_\theta$. In \cite[\S2 Lem. 2.3]{TwistedSheaves}, Mili\v{c}i\'{c} and Soergel show that the categories $\Nte$ and $\mc{M}_{fg}(\mc{U}_\theta,N,\eta)$ are equivalent. This association lets us use the localization functor of Beilinson and Bernstein (Section \ref{Beilinson-Bernstein localization}) to study the category of Whittaker modules geometrically. In particular, by localizing objects in $\mc{M}_{fg}(\mc{U}_\Theta, N, \eta)$ one obtains a category of $\eta$-twisted holonomic $\mc{D}$-modules which are equivariant for the action of $N$. We will discuss the details of this construction in Section \ref{A category of twisted sheaves}.


\subsection{Standard and simple modules}
\label{Standard and simple modules}

In this section we briefly review McDowell's construction of standard Whittaker modules, which are a class of induced modules in $\mc{N}$ that generalize the Verma modules in category $\mc{O}$. For a choice of $\lambda \in \mf{h}^*$ and $\eta \in \ch{\mf{n}}$, we construct a standard Whittaker module $M(\lambda, \eta)$. When $\eta =0$, these modules are Verma modules, and when $\eta$ is nondegenerate, these modules are the irreducible modules studied by Kostant in \cite{Kostant}. For partially degenerate $\eta$, these modules share some structural properties with Verma modules and some structural properties with Kostant's nondegenerate modules. In particular, McDowell showed that the $M(\lambda, \eta)$ decompose into $\h$-weight spaces for the action of a certain subalgebra $\h\subset \mf{h}$ depending on $\eta$. When $\eta=0$, this subalgebra is equal to $\mf{h}$ and McDowell's decomposition is the decomposition of a Verma module into finite-dimensional weight spaces. When $\eta$ is nondegenerate, this subalgebra is trivial, so the entire module is a single infinite-dimensional weight space. After reviewing the construction of $M(\lambda, \eta)$, we generalize McDowell's result and show that all modules in $\mc{N}_\eta$ admit generalized $\h$-weight space decompositions. We also show that these $\h$-weight spaces are themselves Whittaker modules for a Levi subalgebra determined by $\eta$. This extra structure enables us to develop a new theory of formal characters for $\mc{N}$ in Section \ref{Character theory} which generalizes the theory of formal characters of highest weight modules (as described in \cite[\S1.15]{BGGcatO}). 

For the remainder of this subsection, fix an $\mf{n}$-character $\eta \in \ch{\mf{n}}$. For $\alpha \in \Sigma$, let $\mf{g}_\alpha$ be the root space corresponding to $\alpha$. Then $\eta$ determines a subset $\Theta \subset \Pi$ of the simple roots in the following way:
\[
\Theta=\{\alpha \in \Pi : \eta|_{\mf{g}_\alpha}\neq 0\} .
\]
If $\Theta = \Pi$, we say that $\eta$ is \emph{nondegenerate}. We call a Whittaker module $V \in \N_\eta$ for $\eta$ nondegenerate a \emph{nondegenerate Whittaker module}. The cyclically generated Whittaker modules studied by Kostant in \cite{Kostant} are examples of nondegenerate Whittaker modules in our terminology.

Let $\Sigma_\Theta \subset \Sigma$ be the root subsystem generated by $\Theta$, and $\Sigma_\Theta ^+ = \Sigma^+ \cap \Sigma_\Theta$ the corresponding set of positive roots. Let $W_\Theta$ be the Weyl group of $\Sigma_\Theta$, and $\rho_\Theta = \frac{1}{2} \sum_{\alpha \in \Sigma_\Theta ^+} \alpha$. Let 
\[
\mf{n}_\Theta = \bigoplus_{\alpha \in \Sigma_\Theta^+} \mf{g}_\alpha, \mf{u}_\Theta = \bigoplus_{\alpha \in \Sigma^+ - \Sigma_\Theta ^+} \mf{g}_\alpha, \bar{\mf{n}}_\Theta = \bigoplus_{\alpha \in -\Sigma_\Theta ^+} \mf{g}_\alpha, \text{ and }  \bar{\mf{u}}_\Theta = \bigoplus_{\alpha \in -\Sigma^+ -( -\Sigma_\Theta ^+)} \mf{g}_\alpha.
\]
In this way, the character $\eta$ determines a reductive subalgebra $\mf{l}_\Theta = \bar{\mf{n}}_\Theta \oplus \mf{h} \oplus \mf{n}_\Theta$ of $\mf{g}$ and a parabolic subalgebra $\mf{p}_\Theta = \mf{l}_\Theta \oplus \mf{u}_\Theta$. The reductive Lie subalgebra $\mf{l}_\Theta$ decomposes into the direct sum of a semisimple subalgebra $\mf{s}_\Theta$ and its center $\mf{z}_\Theta$. The semisimple subalgebra $\mf{s}_\Theta$ in this decomposition is the derived subalgebra $[\mf{l}_\Theta, \mf{l}_\Theta]$, and it is easy to check that the center $\mf{z}_\Theta$ is the subalgebra $\h=\{H \in \mf{h} \mid \alpha(H) = 0, \alpha \in \Theta\}\subset \mf{h}$.  

Let $\gamma_\Theta:\zl \rightarrow \uh$ be the untwisted Harish-Chandra homomorphism of $\zl$ \cite[Ch. 1 \S 7]{BGGcatO}. Fix $\lambda \in \mf{h}^*$, and define $\varphi_{\Theta, \lambda}:\uh \longrightarrow \mathbb{C}$ to be the homomorphism sending $H\in \mf{h}$ to $(\lambda - \rho_\Theta)(H)\in \C$. The homomorphism  
\begin{equation}
\label{infinitesimal character}
\otl = \varphi_{\Theta, \lambda} \circ \gamma_\Theta: \zl \longrightarrow \mathbb{C}
\end{equation}
is an infinitesimal character of $\zl$. This gives us a map associating elements of $\mf{h}^*$ to maximal ideals in $\zl$:
\begin{align*}
\xi_\Theta:\mf{h}^* &\longrightarrow \text{Max}\zl \\
\lambda &\mapsto \text{ker}(\otl).
\end{align*}

From the data $(\lambda, \eta) \in \mf{h}^* \times \ch{\mf{n}}$, we construct an $\mf{l}_\Theta$-module
\[
Y(\lambda, \eta)=\left(\ul / \xi_\Theta(\lambda)\ul \right) \otimes_{\mathcal{U}(\mf{n}_\Theta)} \mathbb{C}_\eta.
\]
Here $\C_\eta$ is the one-dimensional $\mc{U}(\mf{n}_\Theta)$-module where $\mf{n}_\Theta$ acts by $\eta$. This induced module $Y(\lambda, \eta)$ is an irreducible $\mf{l}_\Theta$-module\cite[\S2 Prop. 2.3]{McDowell}. 

\begin{definition} 
\label{standard Whittaker module}
The $\textit{standard Whittaker module}$ in $\mc{N}$ associated to $\lambda \in \mf{h}^*$ and the character $\eta \in \ch{\mf{n}}$ is the $\mf{g}$-module
\[
M(\lambda, \eta) = \ug \otimes_{\mathcal{U}(\mf{p}_\Theta)}Y(\lambda - \rho + \rho_\Theta, \eta).
\]
Here $Y(\lambda - \rho + \rho_\Theta, \eta)$ is viewed as a $\mathcal{U}(\mf{p}_\Theta)$-module by letting $\mf{u}_\Theta$ act trivially and $M(\lambda, \eta)$ is a $\mf{g}$-module by left multiplication on the first factor.  
\end{definition}

To get a sense for this construction, it is useful to examine particular values of $\eta$. If $\eta = 0$, then $\Theta$ is empty, and $M(\lambda, 0)=\ug \otimes _{\mathcal{U}(\mf{b})} Y(\lambda - \rho, 0)$ is a Verma module of highest weight $\lambda-\rho$. If $\eta$ is nondegenerate, then $M(\lambda, \eta) = Y(\lambda, \eta)$ is an irreducible Whittaker module, as in \cite{Kostant}. 

Two such modules $M(\lambda, \eta)$ and $M(\mu, \eta)$ are isomorphic if and only if $\lambda$ and $\mu$ are in the same $W_\Theta$-orbit in $\mf{h}^*$. McDowell showed that each standard Whittaker module $M(\lambda, \eta)$ has a unique irreducible quotient $L(\lambda, \eta)$, and all irreducible Whittaker modules appear as such quotients \cite[\S2 Thm. 2.9]{McDowell}. Clearly both $M(\lambda, \eta)$ and $L(\lambda, \eta)$ have infinitesimal character $\chi_\lambda$ and generalized $\mf{n}$-character $\eta$, so they both lie in $\Nte$. 




McDowell showed that the center $\h$ of $\mf{l}_\Theta$ acts semisimply on $M(\lambda, \eta)$ \cite[\S2 Prop. 2.4(e)]{McDowell}. This decomposition will be necessary in the theory of formal characters established in the following section, so we briefly review it here. For any $\nu \in \mf{h}^*$, we use bold to denote the restriction of $\nu$ to $\hs$; that is, $\bm{\nu}=\nu|_{\h} \in \hs$. There is a natural partial order on $\hs$ \cite[\S 1 Prop. 1.8(a)]{McDowell}. Let $\Pi-\Theta = \{ \alpha_1, \alpha_2, \cdots, \alpha_p\}$. Then $\{\bm{\alpha_1}, \cdots , \bm{\alpha_p}\}$ is a basis for $\hs$. For $\bm{\alpha}, \bm{\beta} \in \hs$, say that $\bm{\alpha} \leq \bm{\beta}$ if 
\[
\bm{\beta} - \bm{\alpha} = c_1\bm{\alpha_1}+c_2 \bm{\alpha_2} + \cdots + c_p\bm{\alpha_p}
\]
for $c_i \in \mathbb{Z}_{\geq0}$. For a module $V$ in $\mc{N}_\eta$ and linear functional $\bm{\mu} \in \hs$, let $V_{\bm{\mu}}=\{v \in V | Xv=\bm{\mu}(X)v \text{ for all }X \in \h\}$ be the corresponding $\h$-weight space, and $V^{\bm{\mu}}=\{v \in V | \text{ for all } X \in \h, (X - \bm{\mu}(X))^kv=0 \text{ for some } k \in \mathbb{N}\}$ the corresponding generalized $\h$-weight space. If $V^{\bm{\mu}} \neq 0$, we say $\bm{\mu}$ is a \emph{$\h$-weight} of $V$. Then we have the following decomposition:  
\[
M(\lambda, \eta)=\bigoplus_{\bm{\nu} \leq \bm{\lambda - \rho}} M(\lambda, \eta)_{\bm{\nu}}.
\]
Furthermore, $M(\lambda, \eta)_{\bm{\lambda - \rho}}= Y(\lambda - \rho + \rho_\Theta)$, and $M(\lambda, \eta)_{\bm{\nu}}= \uu_{\bm{\mu}}\otimes_\mathbb{C}Y(\lambda - \rho + \rho_\Theta, \eta)$ for $\bm{\mu}\leq0$ in  $\hs$. (Here, we are using the fact that $\mf{h}^\Theta$ acts semisimply on $\uu$ \cite[\S2 Lem. 2.2(a)]{McDowell}.) 
 
The $\h$-weight spaces of $M(\lambda, \eta)$ have a richer structure than just that of $\h$-modules, as the following proposition shows. Given an $\mf{l}_\Theta$-module $V$, we denote by $\overline{V}$ the $\mf{s}_\Theta$-module induced by the inclusion of $\mf{s}_\Theta \subset \mf{l}_\Theta$. Since $\mf{s}_\Theta$ is semisimple, standard semisimplicity results apply to $\overline{V}$. Let $\mc{N}(\mf{s}_\Theta)$ be the category of $\mf{s}_\Theta$-Whittaker modules.

\begin{proposition}
\label{generalized weight spaces of standards}
Let $M(\lambda, \eta)=\bigoplus_{\bm{\nu} \leq \bm{\lambda - \rho}} M(\lambda, \eta)_{\bm{\nu}}$ be the decomposition of a standard Whittaker module in $\N_\eta$ into $\h$-weight spaces. For each $\bm{\nu} \in \hs$, 
\begin{enumerate}[label=(\roman*)]
\item $M(\lambda, \eta)_{\bm{\nu}}$ is a finite length $\mf{l}_\Theta$-module, and
\item $\overline{M(\lambda, \eta)_{\bm{\nu}}}$ is an object in $\mc{N}(\mf{s}_\Theta)$.
\end{enumerate}
\end{proposition}

\begin{proof}
If $\eta=0$, then $\h=\mf{h}$ and $\mf{s}_\Theta = 0$. In this setting, the assertion is trivially true, so we assume $\eta \neq 0$. The action of $\mf{l}_\Theta$ commutes with the action of $\h$, so the $\mf{h}^\Theta$-weight spaces of $M(\lambda, \eta)$ are $\mf{l}_\Theta$-stable. This proves that $M(\lambda, \eta)_{\bm{\nu}}$ are $\mf{l}_\Theta$-modules. The vector space $\uu_{\bm{\mu}}$ is finite dimensional because there are only finitely many ways that we can express a given $\bm{\mu}\leq 0$ in $\hs$ as a negative sum of roots in $\Pi - \Theta$. This implies that $M(\lambda, \eta)_{\bm{\nu}}$ is the tensor product of a finite dimensional $\mf{l}_\Theta$-module with an irreducible Whittaker module. Such modules are of finite length and have composition factors which are irreducible Whittaker modules (for $\eta|_{\mf{n}_\Theta}$) by \cite[\S 4 Thm. 4.6]{Kostant}.  Because categories of Whittaker modules are closed under extensions \cite[\S 1]{CompositionSeries}, this in turn implies that $\overline{M(\lambda, \eta)_{\bm{\nu}}}$ is an object in $\mc{N}(\mf{s}_\Theta)$. 
\end{proof}

The $\h$-weight space structure of $M(\lambda, \eta)$ described in proposition \ref{generalized weight spaces of standards} is also inherited by its unique irreducible quotient $L(\lambda, \eta)$. Moreover, because the unique maximal submodule $N\subset M(\lambda, \eta)$ has $\h$-weights which are strictly less than $\bm{\lambda - \rho}$, $L(\lambda, \eta)$ has a unique maximal $\h$-weight, $\bm{\lambda - \rho}$, with respect to the partial order on $\hs$, and all other weights of $L(\lambda, \eta)$ lie in a cone below this ``highest'' weight. The highest $\h$-weight space of a standard module in $\mc{N}$ and the highest $\h$-weight space of its unique irreducible quotient are both isomorphic to an irreducible $\mf{l}_\Theta$-Whittaker module: $M(\lambda, \eta)_{\bm{\lambda - \rho}} = L(\lambda, \eta)_{\bm{\lambda - \rho}} = Y(\lambda - \rho + \rho_\Theta, \eta)$. 

We finish this section by showing that all modules in $\mc{N}_\eta$ decompose into generalized $\h$-weight spaces, and these weight spaces are modules in $\mc{N}(\mf{s}_\Theta)$. 


\begin{theorem} 
\label{Whittaker weight space decomposition}
Any object $V$ in $\mathcal{N}_\eta$ admits a decomposition 
\[
V=\bigoplus_{\bm{\mu} \in \hs} V^{\bm{\mu}}
\]
where the generalized $\h$-weight spaces $V^{\bm{\mu}}$ are finite length $\mf{l}_\Theta$-modules. Moreover, if we restrict the $\mf{l}_\Theta$-action to the semisimple subalgebra $\mf{s}_\Theta \subset \mf{l}_\Theta$ and denote the resulting $\mf{s}_\Theta$-module by $\overline{V^{\bm{\mu}}}$, the generalized $\h$-weight spaces $\overline{V^{\bm{\mu}}}$ of $V$ are objects in $\mc{N}(\mf{s}_\Theta)$. 
\end{theorem}

\begin{proof} 
It is enough to consider $V \in \Nte$. By \cite[\S 1]{CompositionSeries}, these categories are stable under subquotients and extensions.
The $\h$-semisimplicity of irreducible modules in $\Nte$ implies that all modules in $\Nte$ are $\mc{U}(\h)$-finite. Because objects in $\mc{N}$ are finite length and exact sequences of $\mf{g}$-modules in $\Nte$ descend to exact sequences of $\h$-weight spaces, the assertion follows from induction in the length of $V$.
\end{proof}


\subsection{Character theory}
\label{Character theory}

In this section, we use the decomposition of a module in $\N_\eta$ into generalized $\h$-weight spaces to develop a theory of formal characters in the category of Whittaker modules which generalizes the theory of formal characters of highest weight modules \cite[Ch. 1 \S13]{BGGcatO}. This character theory is new to the literature, though an alternate version of a character theory for Whittaker modules appeared in unpublished work \cite{Lukic}. The main result of this section is that the formal character of a module $V$ in $\N_\eta$ completely determines its class in the Grothendieck group $K\N_\eta$.

Fix an  $\mf{n}$-character $\eta \in \ch{\mf{n}}$, and let $K\N(\mf{s}_\Theta)$ be the Grothendieck group of the category $\N(\mf{s}_\Theta)$. For an object $V \in \N(\mf{s}_\Theta)$, we refer to the corresponding isomorphism class in $K\N(\mf{s}_\Theta)$ by $[V]$.

\begin{definition}
\label{formal character}
 Let $V$ be an object in $\N_\eta$. For $\eta \neq 0$, the \textit{formal character} of $V$ is 
\[
\ch V=\sum_{\bm{\mu} \in \hs} [\overline{V^{\bm{\mu}}}] e^{\bm{\mu}}
\]
where $\overline{V^{\bm{\mu}}}$ is the restriction of the $\mf{l}_\Theta$-module $V^{\bm{\mu}}$ to the semisimple subalgebra $\mf{s}_\Theta\subset \mf{l}_\Theta$, $[\overline{V^{\bm{\mu}}}]$ is the class of $\overline{V^{\bm{\mu}}}$ in the Grothendieck group $K\mc{N}(\mf{s}_\Theta)$, and $e^{\bm{\mu}}$ is a formal variable parameterized by $\bm{\mu} \in \hs$. For $\eta = 0$ and $V \in \N_0$ we define $\ch V = [V] \in K\N.$ 
\end{definition}

A standard Whittaker module is completely determined by its formal character. 

\begin{proposition} The following are equivalent. 
\begin{enumerate}[label=(\roman*)]
\item{$\ch M(\lambda, \eta) = \ch M(\nu, \eta)$.}
\item{$M(\lambda, \eta)=M(\nu, \eta)$.}
\end{enumerate}
\end{proposition}

\begin{proof} It is clear that (ii) implies (i). Assume that $\ch M(\lambda, \eta)=\ch M(\nu, \eta)$. Then $M(\lambda, \eta)$ and $M(\nu, \eta)$ have the same $\h$-weights, and $[\overline{M(\lambda, \eta)}_{\bm{\mu}}]=[\overline{M(\nu, \eta)}_{\bm{\mu}}]$ for any such $\h$-weight $\bm{\mu}$. This implies that $\bm{\lambda - \rho}$ is an $\h$-weight of $M(\nu, \eta)$, so $\bm{\lambda - \rho} \leq \bm{\nu - \rho}$. But also, $\bm{\nu - \rho}$ is an $\h$-weight of $M(\lambda, \eta)$, so $\bm{\nu - \rho}\leq \bm{\lambda - \rho}$ and thus $\bm{\lambda - \rho}=\bm{\nu - \rho}$. Because $M(\lambda, \eta)_{\bm{\lambda - \rho}}=Y(\lambda - \rho + \rho_\Theta, \eta)$ and $M(\nu, \eta)_{\bm{\nu - \rho}}=Y(\nu- \rho + \rho_\Theta, \eta)$, we have 
\[
[\overline{Y(\lambda - \rho + \rho_\Theta, \eta)}]=[\overline{Y(\nu - \rho + \rho_\Theta, \eta)}] \in K\N(\mf{s}_\Theta).
\]
Because the $\mf{s}_\Theta$-modules $\overline{Y(\lambda - \rho +\rho_\Theta)}$ and $\overline{Y(\nu - \rho + \rho_\Theta)}$ are irreducible objects in $\N(\mf{s}_\Theta)$, the equality $[\overline{Y(\lambda - \rho + \rho_\Theta, \eta)}] = [\overline{Y(\nu - \rho + \rho_\Theta, \eta)}]$ of isomorphism classes in the Grothendieck group implies that $\overline{Y(\lambda - \rho + \rho_\Theta, \eta)} = \overline{Y(\nu - \rho + \rho_\Theta, \eta)}$ as $\mf{s}_\Theta$-modules. Irreducible nondegenerate Whittaker modules are completely determined by their infinitesimal character \cite[\S 3 Thm. 3.6.1]{Kostant}, so both modules have infinitesimal character $\Omega_{\Theta, \lambda - \rho + \rho_\Theta}$. This is only possible if $W_\Theta \cdot \lambda = W_\Theta \cdot \nu$, which implies that $M(\lambda, \eta) = M(\nu, \eta)$. 
\end{proof}

Because any module $V$ in $\Nte$ has infinitesimal character $\chi_\lambda$ for $\lambda \in \theta$, there are only finitely many irreducible modules in the category $\Nte$. Let $\{ L(\lambda_1, \eta), \ldots, L(\lambda_m, \eta)\}$ be the distinct irreducible modules in $\Nte$, and let $S_0 = \{ \bm{\lambda_1 - \rho} , \ldots, \bm{\lambda_m - \rho} \} \subset \hs$ be the collection of their highest $\mf{h}^\Theta$-weights. Any module $V$ in $\Nte$ must have composition factors on this list, so by Theorem \ref{Whittaker weight space decomposition}, the $\h$-weights $\bm{\mu}$ of $V$ that show up in the character must be of the form $\bm{\mu}=\bm{\lambda_i - \rho} - \sum_{j=1}^p m_j \bm{\alpha_j}$ for $1 \leq i \leq m$ and $m_j \in \Z_{\geq 0}$.

Let $K \Nte$ be the Grothendieck group of the category $\Nte$. If $V$ and $W$ are isomorphic objects in $\Nte$, then $\ch V=\ch W$, and since character is additive on short exact sequences, we have a well-defined homomorphism
\[
\ch : K\Nte \longrightarrow \prod_{\bm{\mu} \leq S_0} K \N(\mf{s}_\Theta) e^{\bm{\mu}}
\]
given by $\ch [V]=\ch V$. Here $\bm{\mu}\leq S_0$ means that $\bm{\mu} \leq \bm{\lambda_i - \rho}$ for some $\bm{\lambda_i -\rho} \in S_0$. Our main result of this section is the following. 

\begin{theorem}
\label{ch is injective} $\ch : K\Nte \longrightarrow  \prod_{\bm{\mu} \leq S_0} K \N(\mf{s}_\Theta) e^{\bm{\mu}}$ is an injective homomorphism.
\end{theorem}

\begin{proof} To show that $\ch $ is injective, it is enough to show that the set of characters $\{\ch [L(\lambda_1, \eta)], \ldots, \ch [L(\lambda_m, \eta)] \}$ is linearly independent. Consider a non-trivial linear combination
\[
b_1 \ch [L(\lambda_1, \eta)] + \cdots + b_m \ch [L(\lambda_m, \eta)]=0.
\]
As before let $S_0 = \{ \bm{\lambda_1 - \rho}, \ldots, \bm{\lambda_m - \rho} \} \subset \hs$ be the collection of the highest $\h$-weights of the irreducible objects in $\Nte$. Note that the elements $\{\lambda_i\}_{i=1}^m \subset \mf{h}^*$ are distinct, but it is possible that when restricted to $\h$, $\bm{\lambda_i}=\bm{\lambda_j}$ for some $i\neq j$, so $S_0$ might have repeated elements. Choose a maximal element of this set, $\bm{\lambda_j - \rho}$. Then $\bm{\lambda_j - \rho}$ can only appear as a highest weight of modules in $\{ L(\lambda_1, \eta), \ldots, L(\lambda_m, \eta) \}$. 

Because the linear combination of irreducible characters vanishes, the coefficient of $e^{\bm{\lambda_j - \rho}}$ must vanish as well. That coefficient is 
\[
b_{i_1} [\overline{L(\lambda_{i_1}, \eta)}_{\bm{\lambda_j - \rho}}] + \cdots + b_{i_n} [\overline{L(\lambda_{i_n}, \eta)}_{\bm{\lambda_j - \rho}}],
\]
where $\{ \lambda_{i_1}, \ldots, \lambda_{i_n} \} \subset \{ \lambda_1, \ldots, \lambda_m \}$ are the elements of $\mf{h}^*$ so that $\bm{\lambda_{i_1} - \rho} = \cdots = \bm{\lambda_{i_n} - \rho} = \bm{\lambda_j - \rho}$. Because the highest $\h$-weight space of an irreducible module in $\mc{N}$ is an irreducible Whittaker module for $\mf{s}_\Theta$, we have a vanishing linear combination of isomorphism classes of irreducible objects in $K\N(\mf{s}_\Theta)$:
\[
b_{i_1}[\overline{Y(\lambda_{i_1} - \rho + \rho_\Theta, \eta)}] + \cdots + b_{i_n}[\overline{Y(\lambda_{i_n} - \rho + \rho_\Theta, \eta)}]=0
\]
Each of the classes in the above sum must be distinct because the corresponding irreducible modules are non-isomorphic, so we conclude that $b_{i_1}=\cdots = b_{i_n}=0$, and $\ch $ must be injective. 
\end{proof}

This immediately implies the following corollary.

\begin{corollary} 
\label{character determines composition series}
Let $V$ and $W$ be objects in $\Nte$. Then the following are equivalent:
\begin{enumerate}[label=(\roman*)]
\item{$\ch V=\ch W$.}
\item{$V$ and $W$ have the same composition factors.}
\end{enumerate}
\end{corollary}

We complete this section with an explicit calculation of the formal character of a standard Whittaker module, which we will use in Section \ref{Geometric description of Whittaker modules}. Let $M(\lambda, \eta)$ be the standard Whittaker module determined by $\lambda \in \mf{h}^*$ and $\eta \in \ch{\mf{n}}$. Note that as an $\mf{l}_\Theta$-module, $M(\lambda, \eta) = \mc{U}(\bar{\mf{u}}_\Theta) \otimes_\C Y(\lambda - \rho + \rho_\Theta, \eta)$. The Cartan subalgebra $\mf{h}$ acts semisimply on $\mc{U}(\bar{\mf{u}}_\Theta)$, and the collection of $\mf{h}$-weights of $\mc{U}(\bar{\mf{u}}_\Theta)$ are 
\[
Q=\left\{ - \sum_{\alpha \in \Sigma^+ \backslash \Sigma_\Theta ^+} m_\alpha \alpha  : m_\alpha \in \Z_{\geq 0} \right\}.
\]
As described in Section \ref{Standard and simple modules}, $M(\lambda, \eta)$ decomposes into $\h$-weight spaces of the form 
\[
M(\lambda, \eta)_{\bm{\nu}}=\mc{U}(\bar{\mf{u}}_\Theta)_{\bm{\mu}}\otimes_\C Y(\lambda - \rho + \rho_\Theta, \eta)
\]
for $\bm{\mu} \leq 0$ in $\hs$. The $\h$-weight space of $\mc{U}(\bar{\mf{u}}_\Theta)$ corresponding to a $\h$-weight $\bm{\mu}\leq 0$ is the sum of the $\mf{h}$-weight spaces of $\mc{U}(\bar{\mf{u}}_\Theta)$ corresponding to $\mf{h}$-weights that restrict to $\bm{\mu}$ on $\h$; that is, for $\mu \in \h$,
\[
\mc{U}(\bar{\mf{u}}_\Theta)_{\bm{\mu}}=\sum_{\kappa \in Q, \kappa|_{\h}=\bm{\mu}} \mc{U}(\bar{\mf{u}}_\Theta)_{\kappa}.
\]
We define a function $p:Q \rightarrow \mathbb{N}$ by $p(\kappa)=\dim\mc{U}(\bar{\mf{u}}_\Theta)_\kappa$. This function can be interpreted combinatorially as counting the number of distinct ways that $\nu \in \mf{h}^*$ can be expressed as a sum of roots in $\Sigma^+\backslash \Sigma_\Theta ^+$. When $\Theta = \emptyset$, this is Kostant's partition function.

By \cite[\S 2 Lem. 2.2(b)]{McDowell}, each $\mc{U}(\bar{\mf{u}}_\Theta)_{\bm{\mu}}$ is a finite-dimensional $\mf{l}_\Theta$-module, so the $\mf{s}_\Theta$-module $\overline{M(\lambda, \eta)_{\bm{\nu}}}$ is the direct sum of a finite-dimensional $\mf{s}_\Theta$-module and an irreducible $\mf{s}_\Theta$-module. This allow us to apply \cite[\S 4 Thm. 4.6]{Kostant} and conclude that $\mf{n}_\Theta$ acts on $\overline{M(\lambda, \eta)_{\bm{\nu}}}$ by the nondegenerate character $\eta|_{\mf{n}_\Theta}$, and that $\overline{M(\lambda, \eta)_{\bm{\nu}}}$ has composition series length equal to $\displaystyle{\dim \mc{U}(\bar{\mf{u}}_\Theta)_{\bm{\mu}}=\sum_{\kappa \in Q, \kappa|_{\h}=\bm{\mu}} p(\kappa)}$. Furthermore, \cite[\S 4 Thm. 4.6]{Kostant} implies that the composition factors of $\overline{M(\lambda, \eta)_{\bm{\nu}}}$ are 
\[
\{Y(\lambda - \rho + \rho_\Theta + \kappa, \eta) \mid \kappa \in Q \text{ and } \bm{\kappa}=\bm{\mu}\}.
\]
This implies that in the Grothendieck group $K\N(\mf{s}_\Theta)$, 
\[
[\overline{M(\lambda, \eta)_{\bm{\nu}}}]=\sum_{\kappa \in Q, \kappa|_{\h}=\bm{\mu}}p(\kappa)[\overline{Y(\lambda - \rho + \rho_\Theta + \kappa, \eta)}].
\]
Therefore, 
\begin{equation}
\label{character of standard}
\ch M(\lambda, \eta) = \sum_{\bm{\nu} \in \hs}[\overline{M(\lambda, \eta)_{\bm{\nu}})}]e^{\bm{\nu}} = \sum_{\kappa \in Q} p(\kappa) [\overline{Y(\lambda - \rho + \rho_\Theta + \nu, \eta)}]e^{\bm{\lambda - \rho + \kappa}}.
\end{equation}


\section{A category of twisted sheaves}
\label{A category of twisted sheaves}

In this section, we introduce the geometric objects that correspond to Whittaker modules under Beilinson--Bernstein localization. Let $X$ be the flag variety of $\mf{g}$, and for $\lambda \in \mf{h}^*$, let $\mc{D}_\lambda$ be the corresponding twisted sheaf of differential operators on $X$. (See Appendix \ref{Beilinson-Bernstein localization} for more details on this construction.) The geometric category that emerges as an analogue to the category $\Nte$ is a certain subcategory of the category $\mc{M}_{qc}(\mc{D}_\lambda)$ of quasi-coherent $\mc{D}_\lambda$-modules which is equivariant under the action of the Lie group $N=\Int \mf{n}$. We start by describing this category of twisted Harish-Chandra sheaves\footnote{When $\eta=0$, the twist disappears and this category is exactly the category of Harish-Chandra sheaves in \cite[Ch. 4, \S3]{localization}.} for a general Harish-Chandra pair $(\mf{g},K)$ to establish a parameterization of simple objects and to define standard and costandard objects. Then we specialize to the Harish-Chandra pair $(\mf{g},N)$ which describes our setting of Whittaker modules. The classification of simple $\eta$-twisted Harish-Chandra sheaves for an arbitrary Harish-Chandra pair $(\mf{g},K)$ appeared in \cite{TwistedSheaves}, as did the idea of using holonomic duality to define costandard $\eta$-twisted Harish-Chandra sheaves. The results on costandard $\eta$-twisted Harish-Chandra sheaves in this section are new to the literature. 

\subsection{Twisted Harish-Chandra sheaves}
\label{Twisted Harish-Chandra sheaves}

In this section we describe the category of twisted Harish-Chandra sheaves, following \cite{TwistedSheaves}. For details on our choice of notation and geometric conventions, see Appendix \ref{Geometric preliminaries}. Fix a Harish-Chandra pair $(\mf{g}, K)$ and linear form $\lambda \in \mf{h}^*$. Let $\mf{k}$ be the Lie algebra of $K$, and let $\eta:\mf{k} \rightarrow \C$ be a Lie algebra morphism.  We say that $\mc{V}$ is a $(\mc{D}_\lambda, K, \eta)$-module if  
\begin{enumerate}[label=(\roman*)]
\item $\mc{V}$ is a coherent $\D_\lambda$-module,
\item $\mc{V}$ is a $K$-equivariant $\mc{O}_X$-module, and 
\item in $\End{\mc{V}}$, $\pi(\xi) = \mu(\xi)+ \eta(\xi)$  for all $\xi \in \mf{k}$, and the morphism 
\[
\D_\lambda \otimes \mc{V} \rightarrow \mc{V}
\]
is $K$-equivariant. Here $\pi$ is induced by the $\D_\lambda$-action and $\mu$ is the differential of the $K$-action. 
\end{enumerate}
We denote by $\M_{coh}(\D_\lambda, K, \eta)$ the category of $(\D_\lambda, K, \eta)$-modules, and we refer to the objects in this category as \textit{$\eta$-twisted Harish-Chandra sheaves}. This category of twisted Harish-Chandra sheaves carries much of the same structure as the non-twisted category described in \cite[Ch. 4]{localization}. In particular, any $\eta$-twisted Harish-Chandra sheaf is holonomic \cite[Lem. 1.1]{TwistedSheaves} so all $\eta$-twisted Harish-Chandra sheaves have finite length \cite[Cor. 1.2]{TwistedSheaves}.

Irreducible $\eta$-twisted Harish-Chandra sheaves were classified in \cite[\S3]{TwistedSheaves}. An irreducible sheaf in $\M_{coh}(\D_\lambda, K, \eta)$ is uniquely determined by a pair $(Q,\tau)$ of a $K$-orbit $Q\subset X$ and an irreducible $\eta$-twisted connection $\tau$ on $Q$. All irreducible $\eta$-twisted Harish-Chandra sheaves $\mc{L}(Q, \tau)$ occur as unique irreducible subsheaves of standard $\eta$-twisted Harish-Chandra sheaves, which are defined as follows. Fix $x \in Q$, and let $\mf{b}_x$ be the corresponding Borel subalgebra of $\mf{g}$. Let $S_x$ denote the stabilizer in $K$ of $x$. Then the Lie algebra of $S_x$ is $\mf{k} \cap \mf{b}_x$. Let $\mf{c}$ be a Cartan subalgebra in $\mf{g}$ contained in $\mf{b}_x$, and $s:\mf{h}^* \rightarrow \mf{c}^*$ the specialization at $x$ \cite[\S2]{D-modules}.  Let $\mu$ denote the restriction of the specialization of $\lambda + \rho$ to $\mf{k} \cap \mf{b}_x$ and $i:Q \rightarrow X$ the inclusion of $Q$ into $X$. Then in the notation of Appendix \ref{Geometric preliminaries}, $(\mc{D}_\lambda)^i = \mc{D}_{Q, \mu}$ \cite[App. A]{HMSWI}.   

\begin{definition}
\label{standard sheaf}
Let $Q$ be a $K$-orbit in $X$, $i:Q \rightarrow X$ be the natural inclusion, and $\tau$ an irreducible $\mc{M}(\mc{D}_{Q, \mu}, K, \eta)$-module. Then $\mc{I}(Q,\tau) = i_+(\tau)$ is a holonomic $(\mc{D}_\lambda, K, \eta)$-module. We call $\mc{I}(Q,\tau)$ the \textit{standard $\eta$-twisted Harish-Chandra sheaf} attached to $(Q, \tau)$.
\end{definition}

Let us now see how holonomic duality can be used to define costandard objects in the category $\M_{coh}(\D_\lambda, K, \eta)$. For our fixed $\lambda \in \mf{h}^*$ let $\theta\subset \mf{h}^*$ be the Weyl group orbit of $\lambda$. Let $D^b_{coh}(\mc{M}(\mc{D}_\lambda))$ be the derived category of bounded complexes of coherent $\mc{D}_\lambda$-modules. We have a duality functor 
\[
\DD:D^b_{coh}(\mc{M}(\mc{D}_\lambda)) \rightarrow D^b_{coh}(\mc{M}(\D_{-\lambda}))^{op}
\]
given by the formula 
\[
\DD(\mc{V}^\cdot) = RHom_{\D_\lambda}(\mc{V}^\cdot, \D_\lambda)[\dim X],
\]
for $\mc{V}^\cdot\in D^b_{coh}(\mc{M}(\mc{D}_\lambda))$. 

In the case of holonomic $\D_\lambda$-modules, we can use this duality on derived categories to define a notion of duality on modules. Let $\M_{hol}(\D_\lambda)$ be the thick subcategory of $\M_{coh}(\D_\lambda)$ consisting of holonomic $\D_\lambda$-modules. If $\mc{V}$ is an object in $\M_{hol}(\D_\lambda)$, then $\DD(\mc{V})$ is a complex in $D^b_{coh}(\mc{M}(\mc{D}_{-\lambda}))$ with holonomic cohomology and $H^p(\DD(\mc{V}))=0$ for $p \neq 0$. Therefore, we can define a functor 
\[
^*: \M_{hol}(\D_\lambda) \rightarrow \M_{hol}(\D_{-\lambda})^{op}
\]
by 
\[
\mc{V}^*=H^0(\DD(\mc{V})). 
\]
This is the \textit{holonomic duality functor}. We have the following result. 
\begin{theorem}
\label{holonomic duality}
\begin{enumerate}[label=(\roman*)]
\item The functor $\mc{V} \mapsto \mc{V}^*$ from $\M_{hol}(\D_\lambda)$ to $ \M_{hol}(\D_{-\lambda})^{op}$ is an antiequivalence of categories. 
\item The functor $\mc{V} \mapsto (\mc{V}^*)^*$ is isomorphic to the identity functor on $\M_{hol}(\D_\lambda)$.
\end{enumerate}
\end{theorem}

We use the holonomic duality functor to construct costandard objects in the category $\mc{M}_{coh}(\mc{D}_\lambda, K, \eta)$ as follows. Let $Q$ be a $K$-orbit in $X$ and $\tau$ an irreducible $\mc{M}(\mc{D}_{Q, \mu}, K, \eta)$-module.  Let $\mc{L}(Q, \tau)$ be the corresponding irreducible $\eta$-twisted Harish-Chandra sheaf, and $\mc{I}(Q, \tau)$ the corresponding standard $\eta$-twisted Harish-Chandra sheaf. Then $\mc{L}(Q, \tau)$ is an irreducible holonomic $\D_\lambda$-module supported on the closure of the orbit $Q$. Therefore, by Theorem \ref{holonomic duality}, $\mc {L}(Q, \tau)^*$ is an irreducible holonomic $\mc{D}_{-\lambda}$-module whose support is contained in the closure of $Q$. 

\begin{lemma}
\label{duality preserves irreducibility}
\[
\mc{L}(Q, \tau^*)^* = \mc{L}(Q, \tau).
\]
\end{lemma}

\begin{proof} Let $\partial Q = \overline{Q}-Q$ and $X'=X-\partial Q$. Then $j:Q \rightarrow X'$ is a closed immersion, and $k:X' \rightarrow X$ is an open immersion. We have an exact sequence of $\eta$-twisted Harish-Chandra sheaves
\[
0 \rightarrow \mc{L}(Q, \tau) \rightarrow \mc{I}(Q, \tau) \rightarrow \mc{Q} \rightarrow 0,
\]
where $\mc{Q}=\mc{I}(Q, \tau)/\mc{L}(Q, \tau)$. One can show that $\mc{Q}$ is supported on $\partial Q$ \cite[\S3]{TwistedSheaves}. Because $k$ is an open immersion, $k^+$ is exact, and for any $\mc{D}_\lambda$-module $\mc{V}$, $k^+(\mc{V}) = \mc{V}|_{X'}$. Therefore, by restricting to $X'$ we see that $\mc{L}(Q, \tau)|_{X'} = \mc{I}(Q, \tau)|_{X'}$. Because duality is local, we have 
\[
\mc{L}(Q, \tau)^*|_{X'} = (\mc{L}(Q, \tau)|_{X'})^*=(\mc{I}(Q, \tau)|_{X'})^* = j_+(\tau)^*.
\]
Moreover, by Kashiwara's equivalence of categories (Theorem \ref{kashiwara}), $j_+$ commutes with duality, so we have 
\[
\mc{L}(Q, \tau)^*|_{X'} = j_+(\tau^*). 
\]
On the other hand, $\tau^*$ is an irreducible $\eta$-twisted $K$-equivariant connection on $Q$ compatible with $(-\lambda+\rho, \eta)$. Hence, 
\[
\mc{L}(Q, \tau)^*|_{X'} = j_+(\tau^*) = \mc{L}(Q, \tau^*)|_{X'},
\]
and we see that 
\[
\mc{L}(Q, \tau)^* = \mc{L}(Q, \tau^*).
\]
Dualizing, we obtain the desired result.  
\end{proof}

This leads us to our definition of costandard objects in the category $\mc{M}_{coh}(\mc{D}_\lambda, K, \eta)$.

\begin{definition}  Let $Q$ be a $K$-orbit in $X$, $i:Q \rightarrow X$ be the natural inclusion, and $\tau$ an irreducible $\mc{M}(\mc{D}_{Q, \mu}, K, \eta)$-module. The $\eta$-twisted Harish-Chandra sheaf $\M(Q,\tau)=\mc{I}(Q, \tau^*)^*$ is the \textit{costandard $\eta$-twisted Harish-Chandra sheaf} attached to the geometric data $(Q, \tau)$.  
\end{definition}

There is a natural inclusion $\mc{L}(Q, \tau^*) \rightarrow \mc{I}(Q, \tau^*)$. By dualizing, we get a natural epimorphism $\M(Q, \tau) \rightarrow \mc{L}(Q, \tau)$, so $\mc{L}(Q, \tau)$ is a quotient of $\M(Q, \tau)$. The main properties of costandard $\eta$-twisted Harish-Chandra sheaves are the following. 
\begin{proposition}
\label{costandardHCsheaves}
\begin{enumerate}[label=(\roman*)]
\item The length of $\M(Q, \tau)$ is equal to the length of $\mc{I}(Q, \tau)$. 
\item The irreducible $\eta$-twisted Harish-Chandra sheaf $\mc{L}(Q, \tau)$ is the unique irreducible quotient of $\M(Q, \tau)$. The kernel of this projection is supported on $\partial Q$. 
\end{enumerate}
\end{proposition} 
\begin{proof}
Duality preserves irreducibility and $\mc{L}(Q', \tau'^*)^* = \mc{L}(Q', \tau')$ for any irreducible $\eta$-twisted Harish-Chandra sheaf $\mc{L}(Q', \tau')$, so by Lemma \ref{duality preserves irreducibility}, the composition factors of $\M(Q, \tau)$ must be equal to those of $\I(Q, \tau)$. This proves (i). Furthermore, we have a short exact sequence of $\D_{-\lambda}$-modules 
\[
0 \rightarrow \mc{L}(Q, \tau^*) \rightarrow \I(Q, \tau^*) \rightarrow \mc{Q} \rightarrow 0,
\]
where $\mc{Q}$ is a holonomic $\D_{-\lambda}$-module supported in $\partial Q$. Applying holonomic duality to this, we get a short exact sequence of $\D_\lambda$-modules 
\[
0 \rightarrow \mc{Q}^* \rightarrow \M(Q, \tau) \rightarrow \mc{L}(Q, \tau) \rightarrow 0. 
\]
Because $\mc{L}(Q, \tau^*)$ is the unique irreducible submodule of $\I(Q, \tau^*)$ and duality preserves support, this implies that the kernel $\mc{Q}^*$ of the projection map $\M(Q, \tau)\rightarrow \mc{L}(Q, \tau)$ is the unique maximal submodule of $\M(Q, \tau)$ and is supported in $\partial Q$. This proves the proposition.
\end{proof}

We complete this section with a proposition (Proposition \ref{classification}) which will be of use in computing global sections of $\eta$-twisted Harish-Chandra sheaves in Section \ref{Geometric description of Whittaker modules}. The proof of the proposition uses the following three lemmas.  

\begin{lemma}
\label{standard morphism}
If $\mc{V}$ is a object in $\mc{M}_{coh}(\mc{D}_\lambda, K, \eta)$ such that $[\mc{V}]=[\mc{I}(Q,\tau)]$ in the Grothendieck group $K\mc{M}_{coh}(\mc{D}_\lambda, K, \eta)$, then there exists a nontrivial morphism from $\mc{V}$ into $\mc{I}(Q,\tau)$. 
\end{lemma}

\begin{proof}
Let $i:Q \rightarrow X$ be the natural inclusion. As in the proof of Lemma \ref{duality preserves irreducibility}, we can write $i$ as the composition of a closed immersion $j:Q \rightarrow X':=X - \partial Q$ and an open immersion $k:X' \rightarrow X$. Because the quotient $\mc{Q}:=\mc{I}(Q,\tau)/\mc{L}(Q, \tau)$ is supported on $\partial Q$ and the restriction functor $k^+=|_{X'}$ is exact, we have 
\[
\mc{I}(Q,\tau)|_{X'} = \mc{L}(Q,\tau)|_{X'}.
\] 
In $K\mc{M}_{coh}(\mc{D}_\lambda, K, \eta)$, $[\mc{V}]-[\mc{L}(Q,\tau)]=[\mc{Q}]$, so all other composition factors of $\mc{V}$ must be supported in $\partial Q$. Hence
\[
\mc{V}|_{X'}= \mc{L}(Q,\tau)|_{X'}
\]
as well. Since $k_+$ is right adjoint to $|_{X'}$, we have 
\[
\Hom(\mc{V},\mc{I}(Q,\tau)) = \Hom (\mc{V}|_{X'}, j_+(\tau)) = \Hom(\mc{L}(Q,\tau)|_{X'}, \mc{L}(Q,\tau)|_{X'})\neq 0.
\]
This proves the lemma. 
\end{proof}

\begin{lemma}
\label{costandard morphism}
If $\mc{V}$ is an object in $\mc{M}_{coh}(\mc{D}_\lambda, K, \eta)$ such that $[\mc{V}]=[\mc{M}(Q,\tau)]$ in the Grothendieck group $K\mc{M}_{coh}(\mc{D}_\lambda, K, \eta)$, then there exists a nontrivial morphism from $\mc{M}(Q, \tau)$ into $\mc{V}$. 
\end{lemma}

\begin{proof}
By dualizing the morphism in Lemma \ref{standard morphism}, we know that if $[\mc{V}^*]=[\mc{M}(Q, \tau^*)]$ in $K\mc{M}_{coh}(\mc{D}_{-\lambda}, K, \eta)$, then there exists a nontrivial morphism from $M(Q, \tau^*)$ into $\mc{V}^*$. Applying this fact to $\mc{V}^*$ proves the lemma. 
\end{proof} 

\begin{lemma}
\label{character gives isomorphism geometry}
If $\mc{V}$ is an object in $\mc{M}_{coh}(\mc{D}_\lambda, K, \eta)$ such that $[\mc{V}]=[\mc{M}(Q,\tau)]$ and $\mc{V}$ has $\mc{L}(Q, \tau)$ as a unique irreducible quotient, then $\mc{V} \simeq \mc{M}(Q, \tau)$. 
\end{lemma}
\begin{proof}
By Lemma \ref{costandard morphism}, there is a nontrivial morphism $f: \mc{M}(Q, \tau) \rightarrow \mc{V}$. Because $\mc{L}(Q, \tau)$ is the unique irreducible quotient of $\mc{M}(Q, \tau)$ (Proposition \ref{costandardHCsheaves}), the image of $f$ has $\mc{L}(Q, \tau)$ as a composition factor. If the image of $f$ is not all of $\mc{V}$, then it is contained in the unique maximal submodule of $\mc{V}$. But then the image of $f$ cannot have $\mc{L}(Q,\tau)$ as a composition factor. Hence $f$ must be surjective. The objects $\mc{V}$ and $\mc{M}(Q, \tau)$ have the same length, so the kernel of $f$ is zero. We conclude that $f$ is an isomorphism. 
\end{proof} 

We can use the preceding lemmas to relate global sections of $\eta$-twisted Harish-Chandra sheaves to $\eta$-twisted Harish-Chandra modules. For a regular $W$-orbit $\theta \subset \mf{h}^*$ and Lie algebra morphism $\eta:\mf{k} \rightarrow \C$, let $\mc{M}_{fg}(\mc{U}_\theta, K, \eta)$ be the category of $\eta$-twisted Harish-Chandra modules, as in \cite{TwistedSheaves}[\S1] \footnote{The definition in Section \ref{A category of Whittaker modules} is a special case of this category for $K=N$.}.

\begin{proposition}
\label{classification}
Let $\lambda \in \theta\subset \mf{h}^*$ be antidominant and regular, and $\{\mc{M}(Q,\tau)\} \subset \mc{M}_{coh}(\mc{D}_\lambda, K, \eta)$ the set of costandard $\eta$-twisted Harish-Chandra sheaves. Let $\{M(Q, \tau)\}$ be a family of modules in $\mc{M}_{fg}(\mc{U}_\theta,K, \eta)$ parameterized by the pairs $(Q, \tau)$ such that 
\begin{enumerate}[label=(\roman*)]
\item each $M(Q, \tau)$ has a unique irreducible quotient $L(Q, \tau)$, and
\item in $K\mc{M}_{fg}(\mc{U}_\theta, K, \eta)$, $[\Gamma(X, \mc{M}(Q, \tau))] = [M(Q, \tau)]$. 
\end{enumerate}
Then $\Gamma(X, \mc{L}(Q, \tau))=L(Q, \tau)$ and $\Gamma(X, \mc{M}(Q, \tau))=M(Q, \tau)$. 
\end{proposition}
\begin{proof}
We prove the proposition by induction on the dimension of $Q$. Assume that $Q$ is of minimal dimension. Then $\mc{M}(Q, \tau)$ is irreducible. Because $\lambda$ is antidominant and regular, $\Gamma(X, \mc{M}(Q, \tau))$ must be irreducible. The modules $\Gamma(X, \mc{M}(Q, \tau))$ and $M(Q,\tau)$ have the same composition factors because they have the same class in the Grothendieck group, so $\Gamma(X, \mc{M}(Q, \tau))=M(Q, \tau)$. Because $\mc{M}(Q, \tau)=\mc{L}(Q, \tau)$, this proves the proposition in the base case. 

Let $Q$ be of dimension $n$, and assume that $(i)$ and $(ii)$ hold for all $Q'$ of dimension less than or equal to $n$. Because $\mc{M}(Q, \tau)$ has $\mc{L}(Q, \tau)$ as its unique irreducible quotient, all other composition factors of $\mc{M}(Q, \tau)$ are of the form $\mc{L}(Q', \tau')$ for orbits $Q'$ which are contained in $\partial Q$. By the induction assumption, the composition factors of $\Gamma(X, \mc{M}(Q, \tau))$  are $\Gamma(X, \mc{L}(Q', \tau'))=L(Q', \tau')$ and $\Gamma(X, \mc{L}(Q, \tau))$. But $\mc{L}(Q, \tau) \neq \mc{L}(Q', \tau')$ for $Q \neq Q'$, so $\Gamma(X, \mc{L}(Q, \tau)) \neq L(Q', \tau')$. Since $M(Q, \tau)$ has $L(Q, \tau)$ as a unique irreducible quotient and $[M(Q, \tau)]=[\Gamma(X, \mc{M}(Q, \tau))]$ in the Grothendieck group, we must have that $\Gamma(X, \mc{L}(Q, \tau))=L(Q, \tau)$. This proves the first statement. 

It follows that $\Delta_\lambda(M(Q, \tau))$ has unique irreducible quotient $\Delta_\lambda(L(Q, \tau))=\mc{L}(Q, \tau)$. Therefore, by Lemma \ref{character gives isomorphism geometry}, $\Delta_\lambda(M(Q, \tau)) \simeq \mc{M}(Q, \tau)$. This completes the proof.  
\end{proof}


\subsection{The Harish-Chandra pair $(\mf{g}, N)$}
\label{The Harish-Chandra pair (g,N)}

Now we specialize to the setting of Whittaker modules. Let $K=N=\Int \mf{n}$. Let $\mf{b}$ be the unique Borel subalgebra of $\mf{g}$ containing $\mf{n}=\text{Lie}N$. The pair $(\mf{g}, N)$ is a Harish-Chandra pair.  By the discussion in Section \ref{Twisted Harish-Chandra sheaves}, standard objects in $\mc{M}_{coh}(\mc{D}_\lambda, N, \eta)$ are parameterized by pairs $(Q, \tau)$, where $Q$ is an $N$-orbit and $\tau$ is an irreducible $N$-equivariant connection in $\mc{M}_{coh}(\mc{D}_{Q, \mu}, N, \eta)$. In the setting of the Harish-Chandra pair $(\mf{g}, N)$, we can describe these pairs more explicitly. 

The $N$-orbits on $X$ are Bruhat cells $C(w)$, $w \in W$. Our fixed character $\eta \in \ch{\mf{n}}$ determines a parabolic subgroup $P_\Theta \subset G$ such that Lie$P_\Theta = \mf{p}_\Theta$ as in Section \ref{Standard and simple modules}. The $P_\Theta$-orbits on $X$ are unions of Bruhat cells \cite[Ch. 6 \S1 Lem. 1.9]{localization}, and for each $P_\Theta$-orbit, there is a unique Bruhat cell which is open in that orbit. There is a bijection between the $P_\Theta$-orbits in $X$ and the cosets $\W$, and the partial order on orbits determined by closure corresponds to the partial order on $\W$ inherited from the Bruhat order on longest coset representatives \cite[Ch. 6 \S1 Prop. 1.10, Prop 1.11]{localization}. Furthermore, the Weyl group element $w$ parameterizing the unique open Bruhat cell in a $P_\Theta$-orbit is the unique longest coset representative $w^C $ in the corresponding coset $C$. In \cite[\S4]{TwistedSheaves}, Mili\v{c}i\'{c} and Soergel established that the only $N$-orbits admitting compatible connections\footnote{That is, the only orbits on which there exist nontrivial irreducible $(\mc{D}_{Q, \mu}, N, \eta)$-modules} are Bruhat cells $C(w)$ that are open in some $P_\Theta$-orbit. They also established that the only irreducible $\eta$-twisted $N$-equivariant $\mc{O}_{C(w)}$-modules on such Bruhat cells are $\mc{O}_{C(w)}$. Therefore, our standard, simple, and costandard objects in the category $\mc{M}_{coh}(\mc{D}_\lambda, N, \eta)$ are the following. 

\begin{definition}
For the parameters $C \in \W$, $\lambda \in \mf{h}^*$ and $\eta \in \ch{\mf{n}}$, we define $\mc{I}(w^C, \lambda, \eta)$ to be the standard $\eta$-twisted Harish-Chandra sheaf corresponding to the $N$-orbit $C(w^C)$ and the compatible connection $\mc{O}_{C(w^C)}$ on $C(w^C)$. (Here $w^C$ is the unique longest coset representative of $C$.) We refer to the corresponding irreducible $\eta$-twisted Harish-Chandra sheaf by $\mc{L}(w^C, \lambda, \eta)$ and the corresponding costandard $\eta$-twisted Harish-Chandra sheaf by $\mc{M}(w^C, \lambda, \eta)$. 
\end{definition}

\begin{remark}
The parameter $\lambda \in \mf{h}^*$ in this definition emerges in the direct image functor, $i_+:\mc{M}(\mc{D}_{Q,\mu})\rightarrow \mc{M}(\mc{D}_\lambda)$, whose construction depends on $\lambda$. (See Appendix \ref{Modules over twisted sheaves of differential operators} for more details.)  
\end{remark}

It is clear that the global sections of irreducible $\eta$-twisted Harish-Chandra sheaves for the Harish-Chandra pair $(\mf{g},N)$ are $\eta$-twisted Harish-Chandra modules for the same Harish-Chandra pair. Under the equivalence of the categories $\mc{M}_{fg}(\mc{U}_\theta, N, \eta)$ and $\Nte$ \cite[\S2 Lem. 2.3]{TwistedSheaves}, these irreducible $\eta$-twisted Harish-Chandra modules correspond to irreducible Whittaker modules. Recall that the goal of this paper is to develop an algorithm for computing composition multiplicities of standard Whittaker modules. From the arguments above, we see that converting this multiplicity question to the geometric setting of twisted Harish-Chandra sheaves amounts to showing that the global sections of either costandard or standard $\eta$-twisted Harish-Chandra sheaves are standard Whittaker modules. We will do this in Section \ref{Geometric description of Whittaker modules}, but first we establish some useful results on the action of intertwining functors on costandard $\eta$-twisted Harish-Chandra sheaves. 


\subsection{Intertwining functors and U-functors}
\label{Intertwining functors and U-functors}

For $\lambda \in \mf{h}^*$ and $w \in W$, one can construct an ``intertwining functor'' which sends $\mc{D}_\lambda$-modules to $\mc{D}_{w\lambda}$-modules.  These functors play a crucial role in our geometric arguments in Section \ref{A Kazhdan-Lusztig algorithm}, so we use this section to record some of their key properties. Detailed development of these properties can be found in \cite[Ch. 3 \S3]{localization}. 

The orbits of the diagonal action of $G=\Int(\mf{g})$ on $X \times X$ are smooth subvarieties, and can be parameterized in the following way. Given $x, y$ in $X$ and corresponding Borel subalgebras $\mf{b}_x, \mf{b}_y$, we can choose a Cartan subalgebra $\mf{c}$ contained in $\mf{b}_x \cap \mf{b}_y$. Let $\mf{n}_x = [\mf{b}_x, \mf{b}_x]$ and $\mf{n}_y = [\mf{b}_y, \mf{b}_y]$. Then $\mf{b}_x$ and $\mf{b}_y$ determine specializations \cite[\S 2]{D-modules} of $(\mf{h}^*, \Sigma, \Sigma^+)$ into $(\mf{c}^*, R, R_x^+)$, and $(\mf{c}^*, R, R_y^+)$, respectively, where $R$ is the root system of $(\mf{g}, \mf{c})$, $R_x^+ \subset R$ is the collection of positive roots determined by $\mf{n}_x$, and $R_y^+\subset R$ is the collection of positive roots determined by $\mf{n}_y$. The positive root systems $R_x^+$ and $R_y^+$ are related by $w(R_x^+)=R_y^+$  for some Weyl group element $w \in W$, and this $w$ does not depend on choice of Cartan subalgebra in $\mf{b}_x \cap \mf{b}_y$. We say that $\mf{b}_y$ is in \textit{relative position} $w$ with respect to $\mf{b}_x$. It is clear that $\mf{b}_x$ is in relative position $w^{-1}$ with respect to $\mf{b}_y$. For $w \in W$, let 
\begin{equation}
\label{Zw}
Z_w = \{ (x, y) \in X \times X | \mf{b}_y \text{ is in relative position }w \text{ with respect to }\mf{b}_x\}.\
\end{equation}
This gives us a parameterization of $G$-orbits in $X \times X$. 
\begin{lemma}\cite[Ch. 3 \S 3 Lem. 3.1]{localization}
\begin{enumerate}[label=(\roman*)]
\item Sets $Z_w$ for $w \in W$ are smooth subvarieties of $X \times X$.
\item The map $w \mapsto Z_w$ is a bijection of $W$ onto the set of $G$-orbits in $X \times X$. 
\end{enumerate}
\end{lemma}
Denote by $p_1$ and $p_2$ the projections of $Z_w$ onto the first and second factors of $X \times X$, respectively. Then $p_i$ for $i=1, 2$ are locally trivial fibrations with fibers isomorphic to affine spaces of dimension $\ell(w)$. Moreover, they are affine morphisms \cite[Ch. 3 \S 3 Lem. 3.2]{localization}. Let $\omega_{Z_w|X}$ be the invertible $\mc{O}_{Z_w}$-module of top degree relative differential forms for the projection $p_1: Z_w \rightarrow X$ and let $\mc{T}_w$ be its inverse sheaf. Then $\mc{T}_w = p_1^*(\mc{O}(\rho - w\rho))$, and there is a natural isomorphism \cite[Ch. 3 \S 3 Lem. 3.3]{localization}
\[
(\mc{D}_{w\lambda})^{p_1} = (\mc{D}_\lambda^{p_2})^{\mc{T}_w}. 
\]
The morphism $p_2:Z_w \rightarrow X$ is a surjective submersion, so the inverse image functor 
\[
p_2^+:\mc{M}(\mc{D}_\lambda) \rightarrow \mc{M}(\mc{D}_\lambda^{p_2})
\]
is exact. Because twisting by an invertible sheaf is also an exact functor, we can define a functor 
\[
LI_w:D^b(\mc{M}(\mc{D}_\lambda)) \rightarrow D^b(\mc{M}(\mc{D}_{w \lambda}))
\]
by the formula 
\[
LI_w(\mc{V}^\cdot) = p_{1+}(\mc{T}_w \otimes_{\mc{O}_{Z_w}} p_2^+(\mc{V}^\cdot))
\]
for $\mc{V}^\cdot \in D^b(\mc{M}(\mc{D}_\lambda))$. This is the left derived functor of the functor 
\[
I_w: \mc{M}(\mc{D}_\lambda) \rightarrow \mc{M}(\mc{D}_{w \lambda}),
\]
where for $\mc{V} \in \mc{M}(\mc{D}_\lambda)$, 
\[
I_w(\mc{V}) = H^0p_{1+}(\mc{T}_w \otimes_{\mc{O}_{Z_w}}p_2^+(\mc{V})).
\]
We call the right exact functor $I_w$ the \textit{intertwining functor} attached to $w \in W$. 

In the case where $w$ is a simple root, we can define a related collection of {\em U-functors}, which have desirable semisimplicity properties. Let $\alpha \in \Pi$ be a simple root, and denote by $X_\alpha$ the variety of parabolic subalgebras of type $\alpha$. Let $p_\alpha$ be the natural projection of $X$ onto $X_\alpha$, and let $Y_\alpha = X \times_{X_\alpha} X$ be the fiber product of $X$ with $X$ relative to the morphism $p_\alpha$. Denote by $q_1$ and $q_2$ the projections of $Y_\alpha$ onto the first and second factors, respectively. Then we have the following commutative diagram:
\begin{center}
\begin{tikzcd}
Y_\alpha \arrow[rightarrow]{r}{q_2}\arrow[rightarrow,swap]{d}{q_1} 
  & X\arrow[rightarrow]{d}{p_\alpha} \\
X \arrow[rightarrow]{r}{p_\alpha}  
  &X_\alpha .
\end{tikzcd}
\end{center}
There is a natural embedding of $Y_\alpha$ into $X \times X$ that identifies $Y_\alpha$ with the closed subvariety $Z_1\cup Z_{s_\alpha}$ of $X \times X$. Under this identification, $Z_1$ is a closed subvariety of $Y_\alpha$, and $Z_{s_\alpha}$ is an open, dense, affinely embedded subvariety of $Y_\alpha$ \cite[Ch. 3 \S 8 Lem. 8.1]{localization}.  

Let $\lambda \in \mf{h}^*$ be such that $p=-\alpha^\vee(\lambda)$ is an integer. Let $\mc{L}$ be the invertible $O_{Y_\alpha}$-module on $Y_\alpha$ given by 
\[
\mc{L}=q_1^*(\mc{O}((-p+1)s_\alpha \rho + \alpha)) \otimes_{\mc{O}_{Y_\alpha}}q_2^*(\mc{O}((-p+1)\rho))^{-1}.
\]
This allows us to define functors
\[
U^j: \mc{M}_{qc}(\mc{D}_\lambda) \rightarrow \mc{M}_{qc}(\mc{D}_{s_\alpha \lambda})
\]
by the formula 
\[
U^j(\mc{V}) = H^jq_{1+}(q_2^+(\mc{V}) \otimes_{\mc{O}_{Y_\alpha}}\mc{L})
\]
for $\mc{V} \in \mc{M}_{qc}(\mc{D}_\lambda)$ \cite[Ch. 3 \S 8, Lem. 8.2]{localization}. These functors first appeared in \cite{localization} as geometric analogues to the $U_\alpha$ functors in \cite{Vogan}, and they play a critical role in the algorithm of Section \ref{A Kazhdan-Lusztig algorithm} for their semisimplicity properties. Because the fibers of $q_1$ are one-dimensional, $U^j=0$ for $j\neq -1, 0, 1$. If $\mc{V}$ is irreducible, the relationship between $U^j(\mc{V})$ and $I_{s_\alpha}(\mc{V})$ is captured in the following theorem.  
\begin{theorem}\cite[Ch. 3 \S 8 Thm. 8.4]{localization}
\label{UfunctorsIfunctors}
Let $\lambda \in \mf{h}^*$ be such that $p=-\alpha^\vee(\lambda)$ is an integer, and $\mc{V} \in \mc{M}_{qc}(\mc{D}_\lambda)$ an irreducible $\mc{D}_\lambda$-module. Then either 
\begin{enumerate}[label=(\roman*)]
\item $U^{-1}(\mc{V}) = U^1(\mc{V}) = \mc{V}(p\alpha)$ and $U^0(\mc{V}) = 0$, and in this case $I_{s_\alpha}(\mc{V}) = 0$ and $L^{-1}I_{s_\alpha}(\mc{V}) = \mc{V}(p\alpha)$; or 
\item $U^{-1}(\mc{V})=U^1(\mc{V})=0$, and in this case $L^{-1}I_{s_\alpha}(\mc{V})=0$ and there exists a natural exact sequence
\[
0 \rightarrow U^0(\mc{V}) \rightarrow I_{s_\alpha}(\mc{V}) \rightarrow \mc{V}(p\alpha) \rightarrow 0.
\]
The module $U^0(\mc{V})$ is the largest proper quasicoherent $\mc{D}_{s_\alpha \lambda}$-submodule of $I_{s_\alpha}(\mc{V})$. 
\end{enumerate}
\end{theorem}


\subsection{Intertwining functors on standard and costandard sheaves}
\label{Intertwining functors on standard and costandard sheaves}

In this section we examine the action of intertwining functors on standard and costandard $\eta$-twisted Harish-Chandra sheaves in the category $\M_{coh}(\D_\lambda, N, \eta)$. These results will be critical in establishing the relationship between $\Nte$ and $\mc{M}_{coh}(\mc{D}_\lambda, N, \eta)$, and are new to the literature. Let $\alpha \in \Pi$, $w \in W$, and $p_i$ for $i=1,2$ the projections of $Z_{s_\alpha}$ (equation \ref{Zw}) onto the first and second coordinates, respectively. As in Section \ref{The Harish-Chandra pair (g,N)}, let $\mf{b}$ be the unique Borel subalgebra of $\mf{g}$ containing $\mf{n}=\text{Lie}N$. We start with a useful lemma. 
\begin{lemma}
\label{projectionlemma}
The projection $p_1:Z_{s_\alpha} \rightarrow X$ induces an immersion of $p_2^{-1}(C(w))$ into $X$, and its image is equal to $C(ws_\alpha)$. 
\end{lemma}
\begin{proof} If $y \in C(w)$, then $\mf{b}_x$ is in relative position $s_\alpha$ with respect to $\mf{b}_y$ if and only if $x \in C(ws_\alpha)$. Therefore, $p_2^{-1}(C(w))=C(ws_\alpha) \times C(w)$, which implies the result.
\end{proof}
Our first result is the following proposition. 
\begin{proposition}
\label{Intertwining Functors on Standard HC Sheaves}
Let $C \in \W$ and $\alpha \in \Pi$ be such that $Cs_\alpha>C$, and let $\lambda \in \mf{h}^*$ be arbitrary. Then 
\[
LI_{s_\alpha}(\mc{I}(w^C, \lambda, \eta)) = \mc{I}(w^Cs_\alpha, s_\alpha \lambda, \eta).
\]
\end{proposition}
\begin{proof}
The diagram 
\begin{center}
\begin{tikzcd}
p_2^{-1}(C(w^C)) \arrow[rightarrow]{r}{j}\arrow[rightarrow,swap]{d}{pr_2} 
  & Z_{s_\alpha}\arrow[rightarrow]{d}{p_2} \\
C(w^C) \arrow[rightarrow]{r}{i_{w^C}}  
  &X
\end{tikzcd}
\end{center}
commutes. Furthermore, $p_2$ and $pr_2=p_2|_{p_2^{-1}(C(w^C))}$ are surjective submersions and $j$ and $i_{w^C}$ are affine immersions, so $p_2^+, pr_2^+, i_{w^C+},$ and $j_+$ are all exact. Thus, 
\begin{align}
p_2^+(\mc{I}(w^C, \lambda, \eta)) &= p_2^+(i_{w^C+}(\mc{O}_{C(w^C)})) \label{eq:a}\\
&= j_+(pr_2^+(\mc{O}_{C(w^C)})) \label{eq:b}\\
&=j_+(\mc{O}_{p_2^{-1}(C(w^C))}).\label{eq:c}
\end{align}
Here (\ref{eq:a}) is the definition of $\mc{I}(w^C, \lambda, \eta)$, (\ref{eq:b}) is base change, and (\ref{eq:c}) follows from the fact that dim$Z_{s_\alpha} - \text{dim}X = \text{dim}p_2^{-1}(C(w^C))-\text{dim}C(w^C)$. 

Applying the projection formula of Proposition \ref{Projection Formula} to the morphism $p_1$, the line bundle $\mc{L}=\mc{O}(\rho - s_\alpha \rho)$, and the twisted sheaf of differential operators $\mc{D}_\lambda$ on $X$, we obtain the following commutative diagram: 
\begin{center}
\begin{tikzcd}
\mc{M}(\mc{D}_\lambda^{p_1}) \arrow[rightarrow]{r}{p_{1+}}\arrow[rightarrow,swap]{d}{p_1^*(\mc{L}) \otimes_{\mc{O}_{Z_\alpha}} -} 
  & \mc{M}(\mc{D}_\lambda)\arrow[rightarrow]{d}{\mc{L}\otimes_{\mc{O}_X}-} \\
\mc{M}((\mc{D}_\lambda^\mc{L})^{p_1}) \arrow[rightarrow]{r}{p_{1+}}  
  &\mc{M}(\mc{D}_\lambda^\mc{L}).
\end{tikzcd}
\end{center}
We compute
\begin{align}
LI_{s_\alpha}(\mc{I}(w^C, \lambda, \eta))&=p_{1+}(\mc{T}_{s_\alpha} \otimes_{\mc{O}_{Z_{s_\alpha}}}p_2^+(\mc{I}(w^C, \lambda, \eta))) \label{eq:1}\\
&= p_{1+}(\mc{T}_{s_\alpha} \otimes_{\mc{O}_{Z_{s_\alpha}}} j_+(\mc{O}_{p_2^{-1}(C(w^C))})) \label{eq:2}\\
&= p_{1+}(p_1^* (\mc{O}(\rho - s_\alpha \rho)) \otimes_{\mc{O}_{Z_{s_\alpha}}} j_+(\mc{O}_{p_2^{-1}(C(w^C))})) \label{eq:3}\\
&=\mc{O}(\rho - s_\alpha \rho) \otimes _{\mc{O}_X} p_{1+}(j_+(\mc{O}_{p_2^{-1}(C(w^C))})). \label{eq:4}
\end{align}
Here (\ref{eq:1}) follows from the definition of intertwining functors, (\ref{eq:2}) from the equations (\ref{eq:a})-(\ref{eq:c}) above, (\ref{eq:3}) from the fact that $\mc{T}_{s_\alpha}=p_1^* (\mc{O}(\rho - s_\alpha \rho))$, and (\ref{eq:4}) from the projection formula diagram.

By Lemma \ref{projectionlemma}, we have a commutative diagram 
\begin{center}
\begin{tikzcd}
p_2^{-1}(C(w^C)) \arrow[rightarrow]{r}{j}\arrow[rightarrow,swap]{d}{pr_1} 
  & Z_{s_\alpha}\arrow[rightarrow]{d}{p_1} \\
C(w^Cs_\alpha) \arrow[rightarrow]{r}{i_{w^Cs_\alpha}}  
  &X.
\end{tikzcd}
\end{center}
where $pr_1=p_1|_{C(w)}$.

Picking up our previous computation, this lets us further conclude that 
\begin{align}
(\ref{eq:4}) &= \mc{O}(\rho - s_\alpha \rho) \otimes _{\mc{O}_X} i_{w^Cs_\alpha +}(pr_{1+}(\mc{O}_{p_2^{-1}(C(w^C))})) \label{eq:5}\\
&=  \mc{O}(\rho - s_\alpha \rho) \otimes _{\mc{O}_X} i_{w^Cs_\alpha +}(\mc{O}_{C(w^Cs_\alpha)}) \label{eq:6} \\
&=\mc{I}(w^Cs_\alpha, s_\alpha \lambda, \eta).\label{eq:7}
\end{align}
In this final computation, (\ref{eq:5}) follows from the commutative diagram immediately preceding it, (\ref{eq:6}) from Lemma \ref{projectionlemma}, and (\ref{eq:7}) from the definition of $\mc{I}(w^Cs_\alpha, s_\alpha \lambda, \eta)$ and \cite[Ch.2 \S2]{localization}.
\end{proof}

For $C \in \W$, let $\M(w^C, \lambda, \eta)$ be the corresponding costandard $\eta$-twisted Harish-Chandra sheaf in the category $\M_{coh}(\D_\lambda, N, \eta)$. Our second result is the following. 
\begin{proposition} 
\label{Intertwining Functors on Costandard HC Sheaves}
Let $C \in \W$ and $\alpha \in \Pi$ be such that $Cs_\alpha<C$, and let $\lambda \in \mf{h}^*$ be arbitrary. Then 
\[
I_{s_\alpha}(\mc{M}(w^C, \lambda, \eta)) = \M(w^Cs_\alpha, s_\alpha \lambda, \eta),
\]
and
\[
L^pI_{s_\alpha}(\M(w^C, \lambda, \eta)) = 0 \text{ for } p \neq 0.
\]
\end{proposition}
\begin{proof} By Proposition \ref{Intertwining Functors on Standard HC Sheaves} applied to the coset $Cs_\alpha$ and linear form $-\lambda \in \mf{h}^*$, we have 
\[
\I(w^C, -\lambda, \eta)=LI_{s_\alpha}(\I(w^Cs_\alpha, -s_\alpha \lambda, \eta)).
\]
Applying holonomic duality, we get 
\begin{align*}
\M(w^C, \lambda, \eta) &= \DD(LI_{s_\alpha}(\I(w^Cs_\alpha, -s_\alpha \lambda, \eta))) \\
&=(\DD\circ LI_{s_\alpha} \circ\DD )(\M(w^Cs_\alpha, s_\alpha \lambda, \eta)) \\
\end{align*}
By \cite[Ch. 3 \S4 Thm. 4.4]{localization}, $\DD \circ LI_{s_\alpha} \circ \DD$ is the quasi-inverse of the intertwining functor $LI_{s_\alpha}$, so applying $LI_{s_\alpha}$ to both sides of the above equation proves the proposition. 
\end{proof}
Combined with \cite[Ch. 3 \S3 Cor. 3.22]{localization}, this implies the following result. 
\begin{theorem}
\label{Cohomology of Costandard HC Sheaves}
If $\lambda \in \mf{h}^*$ is $\alpha$-antidominant, and $C\in \W$ is such that $Cs_\alpha<C$, we have 
\[
H^p(X, \M(w^C, \lambda, \eta)) = H^p(X, \M(w^Cs_\alpha, s_\alpha \lambda, \eta))
\] 
for any $p \in \Z_+$. 
\end{theorem}
The final result of this section is a technical lemma which uses Proposition \ref{Intertwining Functors on Costandard HC Sheaves} to relate costandard $\eta$-twisted Harish-Chandra sheaves supported on arbitrary $P_\Theta$-orbits to costandard $\eta$-twisted Harish-Chandra sheaves supported on the unique closed $P_\Theta$-orbit. This lemma will be critical in the arguments of Section \ref{Geometric description of Whittaker modules}. Recall that every coset $C \in \W$ has a unique longest coset representative $w^C$ and unique shortest coset representative $w_C$ \cite[Ch. 6 \S1 Thm. 1.4]{localization}. If $w_\Theta \in W_\Theta$ is the longest element, then by \cite[Ch. 6 \S1 Thm. 1.2 Thm. 1.4]{localization}, we have $w_\Theta w_C = w^C$, and $\ell(w_\Theta w_C) = \ell(w_\Theta) + \ell(w_C) = \ell(w^C)$.

\begin{lemma}
\label{Intertwining with Shortest Elements}
Let $\lambda \in \mf{h}^*$ be arbitrary. For any $C \in \W$, 
\[
I_{w_C}(\mc{M}(w^C, \lambda, \eta)) = \mc{M}(w_\Theta, w_C \lambda, \eta),
\]
and
\[
L^pI_{w_C}(\mc{M}(w^C, \lambda, \eta)) = 0 \text{ for } p\neq0.
\]
\end{lemma}
\begin{proof}
We proceed by induction in $\ell(w_C)$. If $\ell(w_C)=0$, then $C=W_\Theta$, and the assertion is trivially true.  If $\ell(w_C) = 1$, then $w_C$ is a simple reflection $s_\alpha$ for $\alpha \in \Pi - \Theta$. Then $\ell(w_\Theta s_\alpha) = \ell(w_\Theta) + 1$ and $W_\Theta s_\alpha > W_\Theta$. By Proposition \ref{Intertwining Functors on Costandard HC Sheaves},
\[
I_{s_\alpha}(\mc{M}(w_\Theta s_\alpha, \lambda, \eta)) = \mc{M}(w_\Theta, s_\alpha \lambda, \eta),
\]
and 
\[
L^pI_{s_\alpha}(\mc{M}(w_\Theta s_\alpha, \lambda, \eta))=0 \text{ for }p\neq0.
\]

Now let $C \in \W$ be arbitrary and assume that 
\[
I_{w_C}(\mc{M}(w^C, \lambda, \eta)) = \mc{M}(w_\Theta, w_C \lambda, \eta) \text{ and } L^pI_{w_C}(\mc{M}(w^C, \lambda, \eta)) = 0 \text{ for } p\neq0.
\]
Let $\alpha \in \Pi$ be such that $Cs_\alpha>C$. By \cite[Ch. 6 \S1 Prop. 1.6]{localization}, the shortest element $w_{Cs_\alpha}$ in $Cs_\alpha$ is $w_Cs_\alpha$. Thus, 
\begin{align*}
I_{w_Cs_\alpha}(\mc{M}(w^Cs_\alpha, \lambda, \eta)) &= I_{w_C}(I_{s_\alpha}(\mc{M}(w^Cs_\alpha, \lambda, \eta))) \\
&= I_{w_C}(\mc{M}(w^C, s_\alpha \lambda, \eta)) \\
&= \mc{M}(w_\Theta, w_Cs_\alpha \lambda, \eta).
\end{align*}
Here the first equality follows from the ``product formula'' for intertwining functors \cite[Ch. 3 \S3 Cor. 3.8]{localization} and the second equality from Proposition \ref{Intertwining Functors on Costandard HC Sheaves}. This completes the proof of the lemma by induction.
\end{proof}


\section{Geometric description of Whittaker modules}
\label{Geometric description of Whittaker modules}

In this section we establish the connection between the category of Whittaker modules and the category of twisted Harish-Chandra sheaves by proving that global sections of costandard twisted Harish-Chandra sheaves are standard Whittaker modules. The theorem is proven in three steps: first, we establish the result for costandard sheaves where the parameter $\eta \in \ch{\mf{n}}$ is nondegenerate; then, we prove that the formal characters align properly for costandard sheaves corresponding to the smallest $P_\Theta$-orbit (where the parameter $\eta$ is allowed to be arbitrary); finally, we extend the result to all costandard sheaves. This proof is new to the literature, though an alternate proof of this relationship was given in the unpublished work \cite{Lukic}. This allows us to use geometric arguments to draw conclusions about our algebraic category of Whittaker modules, which will be essential in the interpretation of the algorithm developed in Section \ref{A Kazhdan-Lusztig algorithm}. Our main tool in this section is the theory of formal characters developed in Section \ref{Character theory}. 

We begin by examining the nondegenerate case. Let $w_0$ be the longest element of the Weyl group $W$ of $\mf{g}$.  
\begin{proposition}
\label{nondegenerate}
Let $\eta \in \ch{\mf{n}}$ be nondegenerate and $\lambda \in \mf{h}^*$. Then 
\[
\Gamma(X, \mc{M}(w_0, \lambda, \eta)) = M(w_0 \lambda, \eta).
\]
\end{proposition}
\begin{proof}
If $\eta$ is nondegenerate, then $W=W_\Theta$, so by \cite[Thm. 5.1]{TwistedSheaves}, there exists a unique irreducible object $\mc{L}(w_0, \lambda, \eta)=\mc{I}(w_0, \lambda, \eta)=\mc{M}(w_0, \lambda, \eta) = \mc{D}_\lambda \otimes_{\mc{U}(\mf{n})} \C_\eta$ in $\mc{M}_{coh}(\mc{D}_\lambda, N, \eta)$. Assume $\lambda$ is antidominant, and let $\theta\subset \mf{h}^*$ be the $W$-orbit of $\lambda$. Then by \cite[Thm. 5.2]{TwistedSheaves}, 
\[
\Gamma(X, \mc{M}(w_0, \lambda, \eta)) = \mc{U}_\theta \otimes _{\mc{U}(\mf{n})} \C_\eta = M(w_0 \lambda, \eta). 
\]
Now, in order to deal with general $\lambda \in \mf{h}^*$, let $w \in W$ be arbitrary. By the preceding argument (first equality) and \cite[Ch. 3 \S3 Thm 3.23]{localization} (second equality), we have 
\[
M(w_0 \lambda, \eta)= R\Gamma(\mc{M}(w_0, \lambda, \eta))=R\Gamma(LI_w(\mc{M}(w_0, \lambda, \eta)))=R\Gamma(\mc{C}^\cdot),
\]
where $\mc{C}^\cdot$ is a complex in $D^b(\mc{D}_{w\lambda})$ such that for any $i \in \Z$, $\mc{C}^i$ is a finite sum of copies of the unique irreducible object $\mc{M}(w_0, w\lambda, \eta)$. (The last equality follows from \cite[\S5 Thm. 5.6]{TwistedSheaves}.) Because the image of $M(w_0 \lambda, \eta)$ in the derived category is a complex with a single irreducible object in degree zero and zeros elsewhere and $R\Gamma$ is an equivalence of derived categories, the equality above implies that 
\[
LI_w(\mc{M}(w_0, \lambda, \eta)) = \mc{M}(w_0, w\lambda, \eta). 
\]
Therefore, 
\[
\Gamma(X, \mc{M}(w_0, w\lambda, \eta)) = M(w_0 \lambda, \eta) = M(w_0w\lambda, \eta).
\]
This completes the proof of the proposition. 
\end{proof}

\begin{proposition}
\label{smallest orbit}
Let $\eta \in \ch{\mf{n}}$ be arbitrary, $\lambda \in \mf{h}^*$, and $\theta \subset \mf{h}^*$ the Weyl group orbit of $\lambda$. In the Grothendieck group $K\mc{M}_{fg}(\mc{U}_\theta, N, \eta)$,
\[
[\Gamma(X, \mc{M}(w_\Theta, \lambda, \eta))] = [M(w_\Theta \lambda, \eta)]. 
\]
\end{proposition}
Here $w_\Theta$ is the longest element in the Weyl group $W_\Theta$ determined by $\Theta$. 
We will prove the proposition in a series of steps. Our first step is to realize the standard sheaf corresponding to the smallest $P_\Theta$-orbit as the direct image of a twisted Harish-Chandra sheaf for the flag variety of $\mf{l}_\Theta$. Let $P(w_\Theta)$ be the $P_\Theta$-orbit with open Bruhat cell $C(w_\Theta)\subset P(w_\Theta)$. Because $w_\Theta$ is minimal in the set of longest coset representatives \cite[Ch. 6 \S 1 Lem. 1.7]{localization}, $P(w_\Theta)$ is a closed subvariety of $X$. Because $P(w_\Theta)$ is an orbit of an algebraic group action it is also a smooth subvariety of $X$. In fact, $P(w_\Theta)$ is isomorphic to the flag variety of $\mf{l}_\Theta$. In particular, by \cite[Ch. 6, \S1, Lem. 1.9]{localization}, we have the following orbit decomposition $P(w_\Theta) = \bigcup_{t \in W_\Theta} C(tw_\Theta) = \bigcup_{w \in W_\Theta} C(w)$. Let 
\[
i_{w_\Theta}:C(w_\Theta) \rightarrow P(w_\Theta),  j:P(w_\Theta) \rightarrow X, \text{ and } i:C(w_\Theta) \rightarrow X
\]  
be the natural inclusions, so $i = j \circ i_{w_\Theta}$ is the composition of an open immersion and a closed immersion. By definition, $\mc{I}(w_\Theta, \lambda, \eta) = j_+(\mc{F})$, where $\mc{F} = i_{w_\Theta +}(\mc{O}_{C(w_\Theta)})$, and $\mc{O}_{C(W_\Theta)}$ is the $N$-equivariant connection in $\mc{M}_{coh}(\mc{D}_\lambda^i, N, \eta)$ described in Section \ref{The Harish-Chandra pair (g,N)}. 
\begin{lemma}
\label{induced from standard}
The sheaf $\mc{F}$ is the standard object $\mc{I}(w_\Theta, \lambda+ \rho - \rho_\Theta, \eta|_{\mf{n}_\Theta})$ in the category $\mc{M}_{coh}(\mc{D}_{P(w_\Theta), \lambda + \rho}, N_\Theta, \eta|_{\mf{n}_\Theta})$ corresponding to the open Bruhat cell $C(w_\Theta) \subset P(w_\Theta)$. 
\end{lemma}
\begin{proof} As described above, we can view $P(w_\Theta)$ as the flag variety for $\mf{l}_\Theta$, and the character $\eta|_{\mf{n}_\Theta}$ is nondegenerate on $\mf{l}_\Theta$. The irreducible $N$-equivariant connection $\mc{O}_{C(w_\Theta)}$ is compatible with $(\lambda, \eta) \in \mf{h}^* \times \ch{\mf{n}}$ by construction. We can restrict the $N$-action to $N_\Theta \subset N$, and consider $\mc{O}_{C(w_\Theta)}$ as an irreducible $N_\Theta$-equivariant connection compatible with $(\lambda, \eta|_{\mf{n}_\Theta}) \in \mf{h}^* \times \mf{n}_\Theta^*$. This allows us to interpret $\mc{F} = i_{w_\Theta +}(\mc{O}_{C(w_\Theta)})$ as the standard sheaf on the flag variety of $\mf{l}_\Theta$ induced from the irreducible $N_\Theta$-equivariant connection $\mc{O}_{C(w_\Theta)}$ on $C(w_\Theta)$ in $\mc{M}_{coh}((\mc{D}_\lambda^j)^i, N_\Theta, \eta|_{\mf{n}_\Theta})$. (Note that because $\eta|_{\mf{n}_\Theta}$ is nondegenerate, this is the only standard $\eta|_{\mf{n}_\Theta}$-twisted Harish-Chandra sheaf in the category $\mc{M}_{coh}(\mc{D}_\lambda^j, N_\Theta, \eta|_{\mf{n}_\Theta})$ by \cite[Thm. 5.1]{TwistedSheaves}.) Because 
\[
\mc{D}_\lambda^j = (\mc{D}_{X, \lambda + \rho})^j = \mc{D}_{P(w_\Theta), \lambda + \rho} = \mc{D}_{\lambda + \rho - \rho_\Theta},
\]
we have that 
\[
\mc{F} = \mc{I}(w_\Theta, \lambda + \rho - \rho_\Theta, \eta|_{\mf{n}_\Theta}).
\]
This completes the proof. 
\end{proof}
Our next step is to use the normal degree filtration (Appendix \ref{Modules over twisted sheaves of differential operators}) to analyze the global sections of the standard sheaf $\mc{I}(w_\Theta, \lambda, \eta)$. We will do so using the theory of formal characters established in Section \ref{Character theory}. By Lemma \ref{induced from standard}, we can express our standard sheaf $\mc{I}(w_\Theta, \lambda, \eta) = j_+(\mc{F})$, where $\mc{F} =  \mc{I}(w_\Theta, \lambda+ \rho - \rho_\Theta, \eta|_{\mf{n}_\Theta})$. Because $j:P(w_\Theta) \rightarrow X$ is a closed immersion, this implies that $\mc{I}(w_\Theta, \lambda, \eta)$ has a filtration by normal degree, $F_n\mc{I}(w_\Theta, \lambda, \eta)$. Let $Gr \mc{I}(w_\Theta, \lambda, \eta)$ be the associated graded sheaf. Let $ch:\Nte \longrightarrow \prod_{\bm{\mu} \leq S_0} K \N([\mf{l}_\Theta, \mf{l}_\Theta]) e^{\bm{\mu}}$ be the formal character function described in Section \ref{Character theory}. 
\begin{lemma}
\label{graded cohomology}
$ch \Gamma(X, Gr \mc{I}(w_\Theta, \lambda, \eta)) = ch\Gamma(X, \mc{I}(w_\Theta, \lambda, \eta)).$
\end{lemma}
\begin{proof}
By construction, we have 
\[
\Gamma(X, \mc{I}(w_\Theta, \lambda, \eta)) = \varinjlim \Gamma(X, F_n \mc{I}(w_\Theta, \lambda, \eta)). 
\]
For each $n \in \Z_+$, we have an exact sequence 
\[
0 \rightarrow F_{n-1} \mc{I}(w_\Theta, \lambda, \eta) \rightarrow F_n \mc{I}(w_\Theta, \lambda, \eta) \rightarrow Gr_n \mc{I}(w_\Theta, \lambda, \eta) \rightarrow 0 . 
\]
We claim that $H^p(X, Gr_n \mc{I}(w_\Theta, \lambda, \eta)) =0$ for $p> 0$. To see this, note that by construction, $Gr_n \mc{I}(w_\Theta, \lambda, \eta)$ is the sheaf-theoretic direct image of a sheaf on $P(w_\Theta)$ which has a finite filtration such that the graded pieces are standard $\eta|_{\mf{n}_\Theta}$-twisted Harish-Chandra sheaves on the flag variety $P(w_\Theta)$ of $\mf{l}_\Theta$. These have vanishing cohomologies by the proof of Proposition \ref{nondegenerate}, which implies the claim. The short exact sequence above gives rise to a long exact sequence 
\begin{align*}
0 &\rightarrow \Gamma(X, F_{n-1} \mc{I}(w_\Theta, \lambda, \eta)) \rightarrow \Gamma(X, F_n \mc{I}(w_\Theta, \lambda, \eta)) \rightarrow \Gamma(X, Gr_n \mc{I}(w_\Theta, \lambda, \eta)) \rightarrow \\
& \rightarrow H^1(X, F_{n-1} \mc{I}(w_\Theta, \lambda, \eta)) \rightarrow H^1(X, F_n \mc{I}(w_\Theta, \lambda, \eta)) \rightarrow 0 \rightarrow \cdots 
\end{align*}
Using induction on $n$ and the preceding paragraph, we see that $H^p(X, F_n \mc{I}(w_\Theta, \lambda, \eta)) = 0$ for $p>0$, and therefore $H^p(X, \mc{I}(w_\Theta, \lambda, \eta)) = 0$ for $p>0$. This implies that for each $n \in \Z_+$, we have a short exact sequence  
\[
0 \rightarrow \Gamma(X, F_{n-1} \mc{I}(w_\Theta, \lambda, \eta)) \rightarrow \Gamma(X, F_{n} \mc{I}(w_\Theta, \lambda, \eta)) \rightarrow \Gamma(X, Gr_n \mc{I}(w_\Theta, \lambda, \eta)) \rightarrow 0. 
\]
Note that if $\lambda \in \mf{h}^*$ is antidominant, the existence of this short exact sequence follows from the exactness of $\Gamma$, but this argument above holds for arbitrary $\lambda \in \mf{h}^*$. This gives us a filtration of $\Gamma(X, \mc{I}(w_\Theta, \lambda, \eta))$, with associated graded module  
\begin{align*}
\Gamma(X, Gr \mc{I}(w_\Theta, \lambda, \eta)) &= \bigoplus \Gamma(X, Gr_n\mc{I}(w_\Theta, \lambda, \eta)) \\
&= \bigoplus \Gamma(X, F_n \mc{I}(w_\Theta, \lambda, \eta))/ \Gamma(X, F_{n-1}\mc{I}(w_\Theta, \lambda, \eta)). 
\end{align*}
Because the formal character sums over short exact sequences, we have 
\[
ch \Gamma(X, Gr_n \mc{I}(w_\Theta, \lambda, \eta)) = ch \Gamma(X, F_n \mc{I}(w_\Theta, \lambda, \eta)) - ch \Gamma(X, F_{n-1} \mc{I}(w_\Theta, \lambda, \eta)).
\]
Now we compute the formal character, using the fact that it distributes through direct sums. 
\begin{align*}
ch \Gamma(X, Gr\mc{I}(w_\Theta, \lambda, \eta)) &= ch \bigoplus_{n \in \Z_+} \Gamma(X, Gr_n \mc{I}(w_\Theta, \lambda, \eta)) \\
&= \sum_{n \in \Z_+} (ch\Gamma (X, F_n \mc{I}(w_\Theta, \lambda, \eta)) - ch\Gamma(X, F_{n-1} \mc{I}(w_\Theta, \lambda, \eta))) \\
&= ch \Gamma(X, \mc{I}(w_\Theta, \lambda, \eta)). 
\end{align*}
This completes the proof.
\end{proof}
This reduces our calculation of the formal character of $\Gamma(X, \mc{I}(w_\Theta, \lambda, \eta))$ to the calculation of the formal character of $\Gamma(X, Gr\mc{I}(w_\Theta, \lambda, \eta))$. Before completing this calculation, we need a few more supporting lemmas. 

The adjoint action of the Borel $\mf{b}$ on $\overline{\mf{u}}_\Theta$ extends to an action of $\mf{b}$ on the universal enveloping  algebra $\mc{U}(\overline{\mf{u}}_\Theta)$.  The $\mf{h}$-weights of this action are 
\[
Q=\left.\left\{ - \sum_{\alpha \in \Sigma^+ \backslash \Sigma_\Theta ^+} m_\alpha \alpha  \right|  m_\alpha \in \Z_{\geq 0} \right\}.
\]
Let $\N_{X|P(w_\Theta)}=j^*(\mc{T}_X)/\mc{T}_{P(w_\Theta)}$ be the normal sheaf of $P(w_\Theta)$ in $X$ and $S(\N_{X|P(w_\Theta)})$ the corresponding sheaf of symmetric algebras.
\begin{lemma} 
\label{symmetric decomposition} As $\mc{O}_{P(w_\Theta)}$-modules,
\[
 S(\mc{N}_{X|P(w_\Theta)})=\bigoplus_{\mu \in Q} \mc{O}(\mu).
\]
\end{lemma}
\begin{proof}
For any $x \in P(w_\Theta)$, there is an equivalence of categories between the category $\mc{M}_{qc}(\mc{O}_{P(w_\Theta)},P_\Theta)$ of quasicoherent $P_\Theta$-equivariant $\mc{O}_{P(w_\Theta)}$-modules and the category of algebraic representations of $B_x=stab_{P_\Theta}\{x\}$ given by taking the geometric fiber of a sheaf $\mc{F}$ in $\mc{M}_{qc}(\mc{O}_{P(w_\Theta)},P_\Theta)$. Under this correspondence, the one-dimensional representation $\C_\mu$ of weight $\mu$ corresponds to the sheaf $\mc{O}_{P(w_\Theta)}(\mu)$. 

Let $x_0 \in X$ be the point corresponding to $B$. The $P_\Theta$-orbit of $x_0$ in $X$ is the unique closed $P_\Theta$-orbit, so it must be equal to $P(w_\Theta)$. In particular, $x_0 \in P(w_\Theta)$, so we have an equivalence of the category $\mc{M}_{qc}(\mc{O}_{P(w_\Theta)},P_\Theta)$ with the category of algebraic representations of $B$. Under this equivalence, the normal sheaf $\mc{N}_{X|P(w_\Theta)}$ corresponds to the Adjoint representation of $B$ on $\overline{\mf{u}}_\Theta$, or, equivalently, the adjoint representation of $\mf{b}$ on $\overline{\mf{u}}_\Theta$. 

Therefore to analyze the $\mc{O}_{P(w_\Theta)}$-module $S(\N_{X|P(w_\Theta)})$, we can examine the symmetric algebra $S(\overline{\mf{u}}_\Theta)$, viewed as a $\mf{b}$-module under the inherited action of the adjoint representation of $\mf{b}$ on $\overline{\mf{u}}_\Theta$. The universal enveloping algebra $\mc{U}(\overline{\mf{u}}_\Theta)$ has a PBW filtration such that the associated graded module $Gr\mc{U}(\overline{\mf{u}}_\Theta)$ is isomorphic to $S(\overline{\mf{u}}_\Theta)$. Under the adjoint action, $\mc{U}(\overline{\mf{u}}_\Theta)$ decomposes into $\mf{h}$-weight spaces corresponding to weights in $Q$. Therefore, the $\mf{b}$-module $S(\overline{\mf{u}}_\Theta)$ decomposes into $\mf{h}$-weight spaces corresponding to the same weights in $Q$. 

For $k \in \Z_{\geq 0}$, consider $V=S^k(\overline{\mf{u}}_\Theta)$. There is a $\mf{b}$-invariant filtration 
\[
0=F_0V \subset F_1V \subset \cdots \subset F_nV=V
\]
such that $F_iV/F_{i-1}V=\C_\mu$, where $\mu \in Q$ is an $\mf{h}$-weight of $S^k(\overline{\mf{u}}_\Theta)$. This induces a filtration of $\mc{V}=S^k(\N_{X|P(w_\Theta)})$
\[
0=F_0\mc{V} \subset F_1\mc{V} \subset \cdots \subset F_n\mc{V}=\mc{V}
\]
where each $F_i\mc{V}$ is a $P_\Theta$-equivariant subsheaf and $F_i\mc{V}/F_{i+1}\mc{V}=\mc{O}_{P(w_\Theta)}(\mu)$. This proves the result.
\end{proof}

\begin{lemma}
\label{twisting standards}
For $\lambda, \mu \in \mf{h}^*$,
\[
\mc{I}(w_\Theta, \lambda, \eta|_{\mf{n}_\Theta}) \otimes_{\mc{O}_{P(w_\Theta)}} \mc{O}(\mu) = \mc{I}(w_\Theta, \lambda + \mu, \eta|_{\mf{n}_\Theta}).
\]
\end{lemma}
\begin{proof} 
This follows immediately from the definition of $\mc{I}(w_\Theta, \lambda + \mu, \eta|_{\mf{n}_\Theta})$ (Definition \ref{standard sheaf}) and the projection formula (Proposition \ref{Projection Formula}). 
\end{proof}

\begin{lemma}
\label{graded decomposition}
As a left $\mc{D}_\lambda$-module, the graded sheaf 
\[
Gr\mc{I}(w_\Theta, \lambda, \eta) = j_\bullet(\mc{F} \otimes_{\mc{O}_{P(w_\Theta)}}S(\mc{N}_{X|P(w_\Theta)})\otimes_{\mc{O}_{P(w_\Theta)}}\mc{O}(2\rho_\Theta-2\rho)).
\]
\end{lemma}
\begin{proof}
Recall the left $\mc{D}_\lambda^j$-module $\mc{F}$ of Lemma \ref{induced from standard}. By an application of equation (\ref{grading by normal degree}) to the right $\mc{D}_\lambda^j$-module $\mc{F}\otimes_{\mc{O}_{P(w_\Theta)}}\omega_{P(w_\Theta)}$, we see that as a right $\mc{D}_\lambda$-module,
\[
Gr\mc{I}(w_\Theta, \lambda, \eta) = j_\bullet(\mc{F} \otimes_{\mc{O}_{P(w_\Theta)}}S(\mc{N}_{X|P(w_\Theta)})\otimes_{\mc{O}_{P(w_\Theta)}}\omega_{P(w_\Theta)}).
\]
Twisting by $\omega_X$ gives us the left $\mc{D}_\lambda$-module structure
\[
Gr\mc{I}(w_\Theta, \lambda, \eta) = j_\bullet(\mc{F} \otimes_{\mc{O}_{P(w_\Theta)}}S(\mc{N}_{X|P(w_\Theta)})\otimes_{\mc{O}_{P(w_\Theta)}}\omega_{P(w_\Theta)|X}),
\]
where $\omega_{P(w_\Theta)|X} = \omega_{P(w_\Theta)} \otimes_{\mc{O}_{P(w_\Theta)}}j^*(\omega_X^{-1})$ is the invertible $\mc{O}_{P(w_\Theta)}$-module of top degree relative differential forms for the morphism $j$. The result then follows from the fact that $\omega_{P(w_\Theta)|X}=\mc{O}(2\rho_\Theta - 2 \rho)$.
\end{proof}
Now we are ready to prove Proposition \ref{smallest orbit}.
\begin{proof}
{\em of Proposition \ref{smallest orbit}}. Using the preceding lemmas and the computation of the character of standard Whittaker modules from Section \ref{Character theory}, we can show that the formal character of $\Gamma(X, \mc{I}(w_\Theta, \lambda, \eta))$ is equal to the formal character of $M(w_\Theta \lambda, \eta)$. By Corollary \ref{character determines composition series}, this implies our result. Here $\lambda \in \mf{h}^*$ and $\eta \in \ch{\mf{n}}$ are arbitrary.  We compute:
\begin{align}
ch\Gamma(X, \mc{I}(&w_\Theta, \lambda, \eta)) = ch\Gamma(X,  Gr \mc{I}(w_\Theta, \lambda, \eta)) \label{eqn4.1}\\ 
&=ch \Gamma(X,  j_\bullet(\mc{F} \otimes_{\mc{O}_{P(w_\Theta)}}S(\mc{N}_{X|P(w_\Theta)})\otimes_{\mc{O}_{P(w_\Theta)}}\mc{O}(2\rho_\Theta-2\rho))) \label{eqn4.2}\\
&=ch\Gamma(P(w_\Theta), \mc{F} \otimes_{\mc{O}_{P(w_\Theta)}}S(\mc{N}_{X|P(w_\Theta)})\otimes_{\mc{O}_{P(w_\Theta)}}\mc{O}(2\rho_\Theta-2\rho)) \label{eqn4.3}\\
&=ch\Gamma(P(w_\Theta), \mc{F} \otimes_{\mc{O}_{P(w_\Theta)}}\bigoplus_{\mu \in Q} \mc{O}(\mu)\otimes_{\mc{O}_{P(w_\Theta)}}\mc{O}(2\rho_\Theta-2\rho)) \label{eqn4.4}\\
&=ch\Gamma(P(w_\Theta),\bigoplus_{\mu \in Q}\mc{I}(w_\Theta, \lambda + \rho - \rho_\Theta + \mu+2\rho_\Theta-2\rho, \eta|_{\mf{n}_\Theta})) \label{eqn4.5}\\
&= ch\bigoplus_{\mu \in Q} Y( \lambda - \rho + \rho_\Theta + \mu, \eta|_{\mf{n}_\Theta}) \label{eqn4.6}\\
&=\sum_{\mu \in Q}[\overline{Y(\lambda - \rho + \rho_\Theta + \mu, \eta)}]e^{\bm{\lambda - \rho + \mu}} \label{eqn4.7}\\
&=chM(\lambda, \eta) =chM(w_\Theta \lambda, \eta) \label{eqn4.8}.
\end{align}
Here, (\ref{eqn4.1}) follows from Lemma \ref{graded cohomology}, (\ref{eqn4.2}) from Lemma \ref{graded decomposition}, (\ref{eqn4.3}) from Kashiwara's theorem, (\ref{eqn4.4}) from Lemma \ref{symmetric decomposition}, (\ref{eqn4.5}) from Lemma \ref{induced from standard}, (\ref{eqn4.6}) from Proposition \ref{nondegenerate}, (\ref{eqn4.7}) from Definition \ref{formal character}, and (\ref{eqn4.8}) from equation (\ref{character of standard}) and the fact that two standard Whittaker modules are isomorphic if their $\mf{h}^*$ parameters are in the same $W_\Theta$-orbit. 

Because $\mc{I}(w_\Theta, \lambda, \eta)=\mc{M}(w_\Theta, \lambda, \eta)$, we conclude using Corollary \ref{character determines composition series} that in $K\mc{M}_{fg}(\mc{U}_\theta, N, \eta)$, 
\[
[\Gamma(X,\mc{M}(w_\Theta, \lambda, \eta))]=[M(w_\Theta \lambda, \eta)].
\]
This completes the proof of Proposition \ref{smallest orbit}.
\end{proof}
Before stating and proving the main result of this section, we record one final fact about tensor products of standard Whittaker modules with finite-dimensional $\mf{g}$-modules. This lemma will be used in the proof of Theorem \ref{global sections of costandards} to deal with the case of singular $\lambda \in \mf{h}^*$. 

Let $\lambda \in \mf{h}^*$ be antidominant and $\mu \in P(\Sigma)$ be antidominant and regular. Then $\lambda + \mu$ is antidominant and regular. Let $Q(\Sigma)$ be the root lattice. Let 
\[
W_\lambda = \{ w \in W \mid w\lambda - \lambda \in Q(\Sigma)\} \subset W
\]
be the integral Weyl group of $\lambda$, which is the Weyl group of the root subsystem
\[
\Sigma_\lambda = \{ \alpha \in \Sigma \mid \alpha^\vee(\lambda) \in \Z\} \subset \Sigma.
\]
For any $\mf{g}$-module $V$, denote by $V_{[\lambda]}$ the generalized $\mc{Z}(\mf{g})$-eigenspace of $V$ corresponding to the infinitesimal character $\chi_\lambda$. 
\begin{lemma}
\label{twisting}
Let $F$ be the finite-dimensional $\mf{g}$-module of highest weight $-\mu$. For $w \in W$, 
\[
(M(w(\lambda+ \mu), \eta) \otimes_\C F)_{[\lambda]} = M(w \lambda, \eta).
\]
\end{lemma}
\begin{proof}
By \cite[Lem. 5.12]{CompositionSeries}, $T:=M(w(\lambda + \mu), \eta) \otimes_\C F$ has a filtration by $\mf{g}$-submodules 
\[
0=T_0 \subset T_1 \subset \cdots \subset T_n=T
\]
such that the associated graded module $GrT$ is isomorphic to the direct sum 
\[
\bigoplus_{\nu \in P(F)} M(w(\lambda+ \mu)+\nu, \eta),
\]
where $P(F)$ is the set of weights of $F$, counted with multiplicity. We claim that there is exactly one standard Whittaker module appearing in this sum with infinitesimal character $\chi_\lambda$, and it is equal to $M(w\lambda, \eta)$. Indeed, assume that for some $v \in W$ and $\nu \in P(F)$, 
\[
w(\lambda + \mu) + \nu = v \lambda.
\] 
Then $\lambda + \mu + w^{-1}\nu = w^{-1} v \lambda$, so $w^{-1}v \lambda - \lambda = w^{-1} \nu - (-\mu) \in Q(\Sigma)$. On one hand, since $\lambda$ is antidominant, $w^{-1}v\lambda - \lambda$ must be a positive sum of positive roots in $\Sigma_\lambda$. On the other hand, since $-\mu$ is the highest weight of $F$ and $w^{-1}\nu \in P(F)$, $w^{-1}\nu - (- \mu)$ is a negative sum of positive roots in $\Sigma_\lambda$. Hence
\[
w^{-1} v \lambda - \lambda = \mu + w^{-1} \nu = 0. 
\]
This implies that $\nu = -w \mu$. The weight $\nu = -w \mu$ is an extremal weight of $F$, so it must occur with multiplicity $1$. Therefore, there is exactly one standard Whittaker module in the direct sum decomposition above with infinitesimal character $\chi_\lambda$, and it is equal to $M(w \lambda, \eta)$. 

The generalized $\mc{Z}(\mf{g})$-eigenspace corresponding to $\chi_\lambda$ is the submodule 
\[
T_{[\lambda]}=\{t \in T \mid (\ker \chi_\lambda)^k \cdot t = 0 \text{ for some }k \in \Z\} \subset T. 
\]
Since $M(w \lambda, \eta)$ appears exactly once in $Gr T$, there is some index $1 \leq i \leq n$ such that 
\[
T_i / T_{i-1} \simeq M(w \lambda , \eta),
\]
and the quotient $T/T_i$ is annihilated by a power of $\prod_{j=i+1}^n \ker \chi_{w(\lambda + \mu) + \nu_j}$ with $\chi_{w(\lambda + \mu) + \nu_j} \neq \chi_\lambda$. This implies that $T/T_i$ is a direct sum of submodules with generalized infinitesimal characters different from $\chi_\lambda$. It follows that $T_{[\lambda]} \subset T_i$. 

Since $T_i$ is annihilated by a power of $\prod_{j=1}^i \ker \chi_{w(\lambda + \mu) + \nu_j}$, $T_i$ splits into a direct sum of submodules with generalized infinitesimal characters $\chi_{w(\lambda + \mu) + \nu_j}$ for $1 \leq j \leq i$. Since $T_{i-1}$ is not annihilated by any power of $\ker \chi_\lambda$, it follows that $T_{[\lambda]}$ is a direct complement of $T_{i-1}$ in $T_i$. Hence $T_{[\lambda]} \simeq M(w \lambda, \eta)$. 
\end{proof}

Finally, we are ready to prove our desired result. 
\begin{theorem}
\label{global sections of costandards}
Let $\lambda \in \mf{h}^*$ be antidominant, $C \in \W$, and $\eta \in \ch{\mf{n}}$ be arbitrary. Then 
\[
\Gamma(X, \mc{M}(w^C, \lambda, \eta)) = M(w^C \lambda, \eta). 
\]
\end{theorem}
\begin{proof}

\noindent
Lemma \ref{Intertwining with Shortest Elements} implies that for $C \in \W$,
\[
LI_{w_C}(\mc{M}(w^C, \lambda, \eta)) = \mc{M}(w_\Theta, w_C \lambda, \eta)
\]
and 
\[
R\Gamma(LI_{w_C}(\mc{M}(w^C, \lambda, \eta)) = R\Gamma(\mc{M}(w_\Theta, w_C \lambda, \eta)).
\]
If $\lambda \in \mf{h}^*$ is antidominant, then by \cite[Ch. 3 \S3 Thm. 3.23]{localization},
\[
R\Gamma(\mc{M}(w^C, \lambda, \eta)) = R\Gamma(\mc{M}(w_\Theta, w_C \lambda, \eta)),
\]
and 
\[
H^p(X, \mc{M}(w^C, \lambda, \eta)) = 0 \text{ for } p>0.
\]
Therefore, by Proposition \ref{smallest orbit},
\[
[\Gamma(X, \mc{M}(w^C, \lambda, \eta))] = [\Gamma(X, \mc{M}(w_\Theta, w_C \lambda, \eta))] = [M(w^C \lambda, \eta)].
\]
Assume furthermore that $\lambda \in \mf{h}^*$ is regular. Because $M(w^C \lambda, \eta)$ has a unique irreducible quotient and $\lambda \in \mf{h}^*$ is antidominant and regular, Proposition \ref{classification} implies our result. 

Now assume that $\lambda \in \mf{h}^*$ is antidominant but not necessarily regular. We extend the result above to this setting using the Zuckerman translation functors of \cite[Ch. 2 \S2]{localization}. Let $\mu \in P(\Sigma)$ be antidominant and regular, so $\lambda + \mu$ is antidominant and regular. By definition, for any coset $C \in \W$, $\mc{I}(w^C, \lambda, \eta) = \mc{I}(w^C, \lambda+ \mu, \eta)(-\mu)$, and by dualizing, the analogous statement is also true for costandard $\eta$-twisted Harish-Chandra sheaves. Let $F$ be the finite-dimensional irreducible $\mf{g}$-module of highest weight $-\mu$. Let $\mc{F}=\mc{O}_X \otimes_\C F$. The sheaf $\mc{F}$ naturally has the structure of an $\mc{U}^\circ:=\mc{O}_X \otimes_\C \mc{U}(\mf{g})$-module. For any $\mc{U}^\circ$-module $\mc{V}$, we denote by $\mc{V}_{[\lambda]}$ the generalized $\mc{Z}(\mf{g})$-eigensheaf corresponding to $\lambda$. (For more details on this construction, see \cite[Ch. 2 \S2]{localization}.) Then, using the fact that $\lambda + \mu$ is antidominant and regular, we compute
\begin{align*}
\Gamma(X, \mc{M}(w^C, \lambda, \eta)) &= \Gamma(X, \mc{M}(w^C, \lambda+ \mu, \eta)(-\mu)) \\
&= \Gamma(X, (\mc{M}(w^C, \lambda + \mu, \eta) \otimes_{\mc{O}_X}\mc{F})_{[\lambda]}) \\
&= \Gamma(X, \mc{M}(w^C, \lambda + \mu, \eta) \otimes_{\mc{O}_X}\mc{F})_{[\lambda]} \\
&= (\Gamma(X, \mc{M}(w^C, \lambda + \mu, \eta)) \otimes_\C F)_{[\lambda]}\\
&= (M(w^C(\lambda+\mu), \eta) \otimes_\C F)_{[\lambda]} \\
&=M(w^C\lambda, \eta).
\end{align*}
Here the second equality follows from \cite[Ch. 2 \S2 Lem. 2.1]{localization} and the final equality follows from Lemma \ref{twisting}. This completes the proof of Theorem \ref{global sections of costandards}.
\end{proof}
It is now straightforward to calculate the global sections of irreducible modules. 
\begin{theorem}
\label{global sections of irreducibles}
Let $\lambda \in \mf{h}^*$ be regular antidominant. Then, for any $C \in \W$, we have 
\[
\Gamma(X, \mc{L}(w^C, \lambda, \eta)) = L(w^C\lambda, \eta). 
\]
\end{theorem}
\begin{proof}Because $\lambda$ is regular antidominant, the global sections functor $\Gamma(X, -)$ is an equivalence of categories. Therefore, by Theorem \ref{global sections of costandards}, the unique irreducible quotient $\mc{L}(w^C, \lambda, \eta)$ of $\mc{M}(w^C, \lambda, \eta)$ must be mapped to the unique irreducible quotient $L(w^C \lambda, \eta)$ of $M(w^C \lambda, \eta)$ by $\Gamma(X, -)$. 
\end{proof}

These results explicitly establish the connection between the category of Whittaker modules and the category of twisted Harish-Chandra sheaves and prepare us to describe the algorithm in the following section. 


\section{A Kazhdan--Lusztig algorithm}
\label{A Kazhdan-Lusztig algorithm}

This section provides an algorithm for computing composition multiplicities of standard Whittaker modules with regular integral infinitesimal character. These multiplicities are given by Whittaker Kazhdan--Lusztig polynomials which are constructed geometrically using twisted Harish-Chandra sheaves. This algorithm is the main result of this paper, and was inspired by the Kazhdan--Lusztig algorithm for Verma modules in \cite[Ch. 5 \S 2]{localization}. 

To state the theorem containing the algorithm, we return to the combinatorial setting of the introduction. Let $W$ be the Weyl group of a reduced root system $\Sigma$ with simple roots $\Pi \subset \Sigma$, and let $S \subset W$ be the corresponding set of simple reflections. For a subset of simple roots $\Theta \subset \Pi$ with Weyl group $W_\Theta \subset W$, let $\mathcal{H}_\Theta$ be the free $\Z[q,q^{-1}]$-module with basis $\delta_C$, $C \in W_\Theta \backslash W$. For $\alpha \in \Pi$, we define a $\Z[q,q^{-1}]$-module endomorphism by 
\[
T_\alpha (\delta_C) = 
  \begin{cases} 
   0 & \text{if } Cs_\alpha = C; \\
   q \delta_C + \delta_{Cs_\alpha} & \text{if } Cs_\alpha>C; \\
   q^{-1} \delta_C + \delta_{Cs_\alpha} & \text{if } Cs_\alpha < C.
  \end{cases}
\]
The order relation on cosets is the Bruhat order on longest coset representatives. This is a partial order \cite[Ch. 6 \S 1]{localization}. The formula for $T_\alpha$ is inspired by formulas related to the antispherical module for the Hecke algebra appearing in \cite{Soergel97}. We will describe explicitly the relationship between our setting and the setting of \cite{Soergel97} in Section \ref{Whittaker Kazhdan-Lusztig polynomials}. The algorithm is given in the following theorem.

\begin{theorem} 
\label{KLalgorithm}
There exists a unique function $\varphi: W_\Theta \backslash W \rightarrow \mathcal{H}_\Theta$ satisfying the following properties. 
\begin{enumerate}[label=(\roman*)]
\item{For $C \in W_\Theta \backslash W$, 
\[
\varphi(C) = \delta_C + \sum_{D<C}P_{CD} \delta_D, 
\]
where $P_{CD} \in q\Z[q]$.}
\item{For $\alpha \in \Pi$ and $C \in \W$ such that $Cs_\alpha < C$, there exist $c_D \in \Z$ such that 
\[
T_\alpha(\varphi(Cs_\alpha))= \sum_{D\leq C} c_D\varphi(D).
\]}
\end{enumerate}
\end{theorem}

The function $\varphi$ determines a family of polynomials $P_{CD}$ parameterized by pairs of cosets in $\W$. We refer to these polynomials as {\em Whittaker Kazhdan--Lusztig polynomials}, because, as we will see in Section \ref{Composition multiplicities of standard Whittaker modules}, they determine composition multiplicities of standard Whittaker modules. 

First we will prove uniqueness of the function $\varphi:\W \rightarrow \mc{H}_\Theta$ in Theorem \ref{KLalgorithm} using a straightforward combinatorial argument. Next, we prove existence of $\varphi$ by appealing to geometry. Defining $\varphi$ geometrically provides the critical link between the Whittaker Kazhdan--Lusztig polynomials $P_{CD}$ of Theorem \ref{KLalgorithm} and Whittaker modules. This is explained in detail in Section \ref{Composition multiplicities of standard Whittaker modules}. 

We begin by proving uniqueness of $\varphi$ in a slightly stronger form. Denote by $\W _{\leq k}$ the set of cosets $C \in \W$ such that $\ell(w^C)\leq k$. 
\begin{lemma}
\label{uniqueness}
Let $k \in \mathbb{N}$. Then there exists at most one function $\varphi: \W_{\leq k} \longrightarrow \mathcal{H}_\Theta$ such that the following properties are satisfied. 
\begin{enumerate}[label=(\roman*)]
\item{For $C \in \W_{\leq k}$, 
\[
\varphi(C) = \delta_C + \sum_{D<C}P_{CD} \delta_D, 
\]
where $P_{CD} \in q\Z[q]$.}
\item{For $\alpha \in \Pi$ and $C \in \W_{\leq k}$ such that $Cs_\alpha < C$, there exist $c_D \in \Z$ such that 
\[
T_\alpha(\varphi(Cs_\alpha))= \sum_{D\leq C} c_D\varphi(D).
\]}
\end{enumerate}
\end{lemma}
\begin{proof} We proceed by induction in $k$. By \cite[Ch. 6 \S1 Lem. 1.7]{localization}, the unique minimal element in the coset order is $W_\Theta$, so the base case is $k=\ell(w_\Theta)$, where $w_\Theta$ is the longest element in $W_\Theta$. In this case, $\W _{\leq k} = \{ W_\Theta \}$. The only possible function $\varphi:\W \longrightarrow \HH$ which satisfies (i) is $\varphi(W_\Theta) = \delta_{W_\Theta}$, and (ii) is void.  

Assume that for $k>\ell(w_\Theta)$, there exists $\varphi:\W_{\leq k} \longrightarrow \HH$ which satisfies (i) and (ii). Our induction assumption is that $\varphi|_{\W_{\leq k-1}}$ is unique. By \cite[Ch. 6 \S1 Prop. 1.6]{localization}, there is a coset $C \in \W_{\leq k}$ such that $\ell(w^C)=k$. Then by \cite[Ch. 6 \S1 Lem. 1.7]{localization}, there exists $\alpha \in \Pi$ such that $Cs_\alpha<C$. By (ii), 
\[
T_\alpha(\varphi(Cs_\alpha))=\sum_{D\leq C} c_D \varphi(D).
\]
Evaluating at $q=0$ and using (i), we have 
\[
T_\alpha(\varphi(Cs_\alpha))(0) = \sum_{D\leq C} c_D\left(\delta_D + \sum_{E<D}P_{DE}(0)\delta_C \right) = \sum_{D\leq C} c_D \delta_D.
\]
Because $\ell(w^{Cs_\alpha}) = k-1$, the induction assumption implies that the coefficients $c_D$ in this sum are uniquely determined. On the other hand, using the definition of $\varphi$ and $T_\alpha$, we compute
\begin{align*}
T_\alpha(\varphi(Cs_\alpha)) &=T_\alpha(\delta_{Cs_\alpha})+\sum_{D<Cs_\alpha}P_{Cs_\alpha D}T_\alpha (\delta_D) \\
&=q\delta_{Cs_\alpha}+\delta_C+\sum_{D<Cs_\alpha}P_{Cs_\alpha D}T_\alpha(\delta_D).
\end{align*}
Because all cosets $D$ appearing in the sum are less than $Cs_\alpha$ in the coset order, $\ell(w^D)<k-1$ for any such $D$. In particular, $\delta_C$ does not show up in this sum. Evaluating at zero and setting this equal to our first computation, we conclude that $c_C = 1$. Therefore, 
\[
\varphi(C) = T(\varphi(Cs_\alpha)) - \sum_{D<C}c_D\varphi(D).
\]
This shows that the Lemma holds for $\W_{\leq k}$ , and we are done by induction.
\end{proof}
The uniqueness of Theorem \ref{KLalgorithm} follows immediately from Lemma \ref{uniqueness}. Next we establish a parity condition on solutions of Lemma \ref{uniqueness} which will be critical in upcoming computations. 

We define additive involutions $i$ on $\Zq$ and $\iota$ on $\mathcal{H}_\Theta$ by 
\begin{align*}
i(q^m)&= (-1)^m q^m, \text{ for } m \in \Z, \text{ and } \\
\iota(q^m \delta_C) &= (-1)^{m+\ell(w^C)}q^m \delta_C, \text{ for } m \in \Z \text{ and } C \in \W.
\end{align*}
A simple calculation shows that $\iota T_\alpha \iota = -T_\alpha$. 
\begin{lemma}
\label{parity}
Let $k \in \mathbb{N}$. Let $\varphi: \W_{\leq k} \longrightarrow \HH$ be a function satisfying properties (i) and (ii) of Lemma \ref{uniqueness}. Then 
\[
P_{CD}=q^{\ell(w^C)-\ell(w^D)}Q_{CD},
\] 
where $Q_{CD} \in \Z[q^2, q^{-2}]$. 
\end{lemma}
\begin{proof} Define a function $\psi:\W_{\leq k} \rightarrow \HH$ by $\psi(C)=(-1)^{\ell(w^C)}\iota(\varphi(C))$. Then 
\[
\psi(C)=\delta_C + \sum_{D<C} (-1)^{\ell(w^C)-\ell(w^D)}i(P_{CD})\delta_D.
\]
The polynomials $(-1)^{\ell(w^C)-\ell(w^D)}i(P_{CD})$ are in $q \Z[q]$, so $\psi$ satisfies (i). We will show that $\psi$ also satisfies (ii), then use Lemma \ref{uniqueness} to conclude that $\psi = \varphi$. Let $C \in \W_{\leq k}$ and $\alpha \in \Pi$ such that $Cs_\alpha<C$. Then 
{\allowdisplaybreaks
\begin{align*}
T_\alpha(\psi(Cs_\alpha)) &=(-1)^{\ell(w^C)}\left( -T_\alpha(\iota(\varphi(Cs_\alpha)))\right) \\
&=(-1)^{\ell(w^C)}\iota T_\alpha \iota (\iota (\varphi(Cs_\alpha))) \\
&=(-1)^{\ell(w^C)}\iota\left( \sum_{D\leq C}c_D \varphi(D)\right) \\
&=(-1)^{\ell(w^C)} \sum_{D\leq C} c_D \iota(\varphi(D)) \\
&=\sum_{D\leq C} (-1)^{\ell(w^C) - \ell(w^D)} c_D \psi(D).
\end{align*}
}
This shows that $\psi$ satisfies (ii), so Lemma \ref{uniqueness} implies that $\varphi = \psi$; that is, that 
\[
P_{CD} = (-1)^{\ell(w^C)-\ell(w^D)}i(P_{CD}). 
\]
This relationship implies the result. 
\end{proof}
Now we are ready to prove the existence statement of Theorem \ref{KLalgorithm}. Let $\mc{F} \in \MDNe$. For $w \in W$, let $i_w: C(w) \longrightarrow X$ be the canonical immersion of the corresponding Bruhat cell into the flag variety. We note the following facts.
\begin{itemize}
\item{For any $k \in \Z$, $L^{-k}i_w^+(\mc{F})$ is an $\eta$-twisted $N$-equivariant connection on $C(w)$, so it is isomorphic to a direct sum of copies of $\mc{O}_{C(w)}$.  We refer to the number of copies of $\mc{O}_{C(w)}$ that appear in this decomposition as the $\mc{O}$-dimension, and denote it $\text{dim}_\mc{O}(L^{-k}i^+_w(\mc{F}))$.}
\item{Because the dimension of $C(w)$ is $\ell(w)$, for any $k \in \Z$, 
\[
R^{n-\ell(w)-k}i_w^!(\mc{F}) = L^{-k}i_w^+(\mc{F}).
\]
Here $n=\dim X$.}
\end{itemize} 
We define a function $\nu: \MDNe \longrightarrow \HH$ by 
\begin{equation}
\label{nu}
\nu(\mc{F})=\sum_{C \in \W} \sum_{m \in \Z} \text{dim}_\mc{O}(R^m i^!_{w^C}(\mc{F}))q^m \delta_{C}.
\end{equation}
For $C \in \W$, let $\mc{I}_C := \mc{I}(w^C, -\rho, \eta)$ be the standard sheaf in $\mc{M}_{coh}(\mc{D}_X,N,\eta)$ corresponding to the coset $C$ and $\mc{L}_C := \mc{L}(w^C, -\rho, \eta)$ its unique irreducible subsheaf. 
\begin{proposition} 
\label{existence}
Let $\varphi(C)=\nu(\mc{L}_C)$. Then $\varphi$ satisfies conditions (i) and (ii) in Theorem \ref{KLalgorithm}.
\end{proposition}
Checking that $\varphi$ satisfies \ref{KLalgorithm} (i) is straightforward. 
\begin{lemma} Let $\varphi(C) = \nu(\mc{L}_C)$. Then 
\[
\varphi(C)=\delta_C+\sum_{D<C}P_{CD}\delta_D,
\]
where $P_{CD} \in q\Z[q]$. 
\end{lemma}
\begin{proof} We need to show three things:
\begin{enumerate}[label=(\alph*)]
\item{If $D \not \leq C$, $\text{dim}_\mc{O}(R^mi^!_{w^D}(\mc{L}_{C})) = 0$ for all $m \in \Z$,} 
\item $\text{dim}_\mc{O}(R^mi^!_{w^C}(\mc{L}_{C}))=
\begin{cases} 
      1 & \text{ if } m=0 \\
      0 & \text{otherwise} 
   \end{cases}, \text{ and }$
\item{if $D<C$, $\text{dim}_\mc{O}(R^mi^!_{w^D}(\mc{L}_{C}))=0$ for all $m\leq0$.}
\end{enumerate}
Part (a) follows immediately from the fact that $\text{supp}\mc{L}_{C} = \overline{C(w^C)}$ and $D\leq C$ in the coset order if and only if $C(w^D) \subset \overline{C(w^C)}$ \cite[Prop 1.11]{localization}. To see part (b), we first observe that
\[
R^0i^!_{w^C}(\mc{L}_{C}) = R^0i_{w^C}^!(\mc{I}_{C})=R^0i^!_{w^C}(i_{w^C+}(\mc{O}_{C(w^C)})) = \mc{O}_{C(w^C)}.
\]
So dim$_\mc{O}(R^0i^!_{w^C}(\mc{L}_{C}))=1$. Furthermore, for $m \neq 0$,
\[
R^mi_{w^C}^!(\mc{L}_{C})=R^mi_{w^C}^!(\mc{I}_{C}) = R^mi^!_{w^C}(i_{w^C+}(\mc{O}_{C(w^C)})) = 0.
\]
This proves (b). We end by showing (c).  Let $D \in \W$ be a coset so that $D<C$. Because $i_{w^D}$ is an immersion, $i^!_{w^D}$ is a right derived functor, so for any $m<0$, $R^mi^!_{w^D}(\mc{V})=0$ for any $\mc{D}$-module $\mc{V}$ on $X$. Thus all that remains is to show that $R^0i_{w^D}^!(\mc{L}_{C})=0$.  Let $X'=X-\partial C(w^D)$, and let $j_{w^D}:C(w^D) \rightarrow X'$ be the natural closed immersion, and $k_{w^D}:X'\rightarrow X$ the natural open immersion. Then we have a commutative diagram. 
\begin{center}
\begin{tikzcd}
C(w^D) \arrow[swap]{rd}{j_{w^D}}\arrow{rr}{i_{w^D}} 
& \text{ } & X \\
 \text{ }
  & X' \arrow[swap]{ru}{k_{w^D}}  & \text{ } 
\end{tikzcd}
\end{center}
Using the fact that $\dim X= \dim X'$, that $k_{w^D}$ is an open immersion, and Kashiwara's Theorem, we compute
\begin{align*}
R^0j_{w^D+}(R^0i_{w^D}^!(\mc{L}_{C}))&= R^0j_{w^D+}(R^0j_{w^D}^!(R^0k_{w^D}^!(\mc{L}_{C}))) \\
&=R^0j_{w^D+}(R^0j_{w^D}^!(L^0k_{w^D}^+(\mc{L}_{C}))) \\
&=R^0j_{w^D+}(R^0j_{w^D}^!(\mc{L}_{C}|_{X'})) \\
&=R^0\Gamma_{C(w^D)}(\mc{L}_{C}|_{X'}).
\end{align*}
From this calculation we see that $R^0j_{w^D+}(R^0i_{w^D}^!(\mc{L}_{C}))$ is the submodule of $\mc{L}_{C}|_{X'}$ consisting of sections supported on $C(w^D)$. However, because $X'$ is open, $\mc{L}_C|_{X'}$ is irreducible, so this submodule must be zero. We conclude that $R^0i_{w^D}^!(\mc{L}_{D})=0$, which completes the proof of the lemma.
\end{proof}

Our final step in proving Theorem \ref{KLalgorithm} is establishing that $\varphi$ satisfies Theorem \ref{KLalgorithm}(ii). Before we make this argument, we need to introduce a useful family of functors $U_\alpha^k:\mc{M}_{qc}(\mc{D}_X)\rightarrow \mc{M}_{qc}(\mc{D}_X)$ and examine their semisimplicity properties. We dedicate the next page to doing so.  

Fix $\alpha \in \Pi$, and let $p_\alpha: X \longrightarrow X_\alpha$ be projection onto the flag variety of parabolic subalgebras of type $\alpha$. If $P_\alpha \subset G$ is the standard parabolic of type $\alpha$, then $P_\alpha = B \cup Bs_\alpha B$. Let $C(v)$ be the Bruhat cell corresponding to $v \in W$. Then we have the following facts:
\begin{itemize}
\item{The Bruhat cell $C(v)\simeq \C^{\ell(v)}$, so $i_v: C(v) \longrightarrow X$ is an affine morphism.}
\item{The image $p_\alpha(C(v))$ is an affine subvariety of $X_\alpha$.}
\item{The projection $p_\alpha$ is locally trivial, so $p_\alpha^{-1}(p_\alpha(C(v))$ is a smooth, affinely embedded subvariety of $X$.}
\end{itemize}
We conclude that $p_\alpha^{-1}(p_\alpha(C(v)))=C(v)\cup C(vs_\alpha)$. One of these orbits is closed in $p_\alpha^{-1}(p_\alpha(C(v)))$ and the other is open and dense. We have two possible scenarios:
\begin{enumerate}
\item{$\ell(vs_\alpha) = \ell(v)+1$. Then $\text{dim}(C(vs_\alpha))>\text{dim}(C(v))$, and so }
\begin{itemize}
\item{$C(vs_\alpha)$ is open and dense in $p_\alpha^{-1}(p_\alpha(C(v)))$,}
\item{$C(v)$ is closed in $p_\alpha^{-1}(p_\alpha(C(v)))$, and}
\item{$p_\alpha: C(v) \longrightarrow p_\alpha(C(v))$ is an isomorphism.}
\end{itemize}
\item{$\ell(vs_\alpha) = \ell(v)-1$. Then $\text{dim}(C(vs_\alpha))<\text{dim}(C(v))$, and so }
\begin{itemize}
\item{$C(vs_\alpha)$ is closed in $p_\alpha^{-1}(p_\alpha(C(v)))$,}
\item{$C(v)$ is open and dense in $p_\alpha^{-1}(p_\alpha(C(v)))$, and}
\item{$p_\alpha: C(v) \longrightarrow p_\alpha(C(v))$ is a fibration with fibers isomorphic to an affine line.}
\end{itemize}
\end{enumerate} 
We define a family of functors
\index{U-functors} $U_\alpha^k: \mc{M}_{qc}(\mc{D}_X) \longrightarrow \mc{M}_{qc}(\mc{D}_X)$ by 
\[
U_\alpha^k(\mc{F}) = p_\alpha^+(H^kp_{\alpha +}(\mc{F})).
\]
Because the fibers of the projection map $p_\alpha: X \rightarrow X_\alpha$ are one-dimensional, $U_\alpha^k$ can be non-zero only for $k\in \{ -1, 0, 1\}$. These functors are closely related to the $U$-functors discussed in Section \ref{Intertwining functors and U-functors}. (We will make this relationship explicit in the proof of Theorem \ref{semisimplicity}.) Their main utility in our argument comes from their semisimplicity properties. 
\begin{lemma}
\label{semisimplicity}
Let $C \in \W$ and $\alpha \in \Pi$ be such that $Cs_\alpha < C$. Then 
\begin{enumerate}[label=(\roman*)]
\item{$U_\alpha^k(\mc{L}_{Cs_\alpha})=0$ for all $k \neq 0$, and }
\item{$U_\alpha^0(\mc{L}_{Cs_\alpha})$ is a direct sum of $\mc{L}_D$ for $D\leq C$.}
\end{enumerate}
\end{lemma}
\begin{proof} By construction, $U_\alpha^0(\mc{L}_{Cs_\alpha})$ is a holonomic $(\mc{D}_X, N, \eta)$-module supported in $ \overline{C(w^C)}$, so $U_\alpha^0(\mc{L}_{Cs_\alpha})$ has finite length, and its composition factors must be in the set $\{\mc{L}_D|D\in \W \text{ and } D\leq C\}$. Because $p_\alpha$ is a locally trivial fibration with fibers isomorphic to $\mathbb{P}^1$ (in particular, it is a projective morphism of smooth quasi-projective varieties), and $\mc{L}_{Cs_\alpha}$ is a semisimple holonomic $\mc{D}$-module the decomposition theorem \cite[\S 1 Thm. 1.4.1]{Mochizuki} implies that $H^kp_{\alpha +}(\mc{L}_{Cs_\alpha})$ are semisimple. By the local triviality of $p_\alpha$, this in turn implies that $U_\alpha^0(\mc{L}_{Cs_\alpha})$ are semisimple, which completes the proof of (ii).   

To prove (i), we establish the connection between $U_\alpha^0$ and the U-functors of Section \ref{Intertwining functors and U-functors}. Let $Y_\alpha = X \times_{X_\alpha} X$ be the fiber product of $X$ with itself relative to the morphism $p_\alpha$ with projections $q_1$ and $q_2$ onto the factors. By base change (Theorem \ref{basechange}), 
\[
U_\alpha^k(\mc{L}_{Cs_\alpha}) = p_\alpha^+(H^kp_{\alpha+}(\mc{L}_{Cs_\alpha})) = H^kq_{1+}(q_2^+(\mc{L}_{Cs_\alpha})).
\]
Because $\mc{D}_X = \mc{D}_{-\rho}$, we have that the twist $U_\alpha^k(\mc{L}_{Cs_\alpha})(\alpha) = U^k(\mc{L}_{Cs_\alpha})$, where $U^k$ is the functor from Section \ref{Intertwining functors and U-functors}. To complete the proof, we need to show that we are in case (ii) of Theorem \ref{UfunctorsIfunctors}; that is, that $L^{-1}I_{s_\alpha}(\mc{L}_{Cs_\alpha})=0$. Because $Cs_\alpha<C$, we can apply Proposition \ref{Intertwining Functors on Standard HC Sheaves} to the coset $Cs_\alpha$ and conclude that 
\[
LI_{s_\alpha}(\mc{I}(w^Cs_\alpha, \lambda, \eta)) = \mc{I}(w^C, s_\alpha \lambda, \eta).
\] 
In particular, this implies that $L^{-1}I_{s_\alpha}(\mc{I}(w^cs_\alpha, \lambda, \eta))=0$, and because $\mc{L}_{Cs_\alpha}$ is a submodule of $\mc{I}(w^cs_\alpha, \lambda, \eta)$, $L^{-1}I_{s_\alpha}(\mc{L}_{Cs_\alpha})=0$ as well.
\end{proof}
We are working toward showing that $\varphi(C)=\nu(\mc{L}_C)$ satisfies (ii). We will do so by proving that for $\alpha \in \Pi$ and $C \in \W$ such that $Cs_\alpha <C$, $T_\alpha(\varphi(Cs_\alpha))=\nu(U_\alpha^0(\mc{L}_{Cs_\alpha}))$. This relationship is useful because it allows us to use Lemma \ref{semisimplicity} to decompose $\nu(U_\alpha^0(\mc{L}_{Cs_\alpha}))$ and obtain the desired sum in Theorem \ref{KLalgorithm}(ii). Before jumping into the argument, we must establish what happens if we pull back an irreducible module to a Bruhat cell which corresponds to a Weyl group element which is not a longest representative in some coset $C\in \W$. Lemma \ref{whittakercase} will be critical in upcoming computations.
 
\begin{lemma}
\label{whittakercase}
Let $v \in W$ be a Weyl group element such that $v \neq w^C$ is not a longest coset element for any coset $C \in \W$. Let $\mc{F} \in \mc{M}_{coh}(\mc{D}_X,N,\eta)$ be irreducible. Then 
\[
R^ki_v^!(\mc{F})=0
\]
for all $k \in \Z$. 
\end{lemma}
\begin{proof} 
Let $X'=X-\partial C(v)$, and express the canonical immersion $i_v$ as the composition of a closed immersion and an open immersion in the following way. 
\begin{center}
\begin{tikzcd}
C(v) \arrow[rightarrow]{r}{j_v} \arrow[bend right]{rr}[swap]{i_v}
 & X' \arrow[rightarrow]{r}{k_v}
  & X
\end{tikzcd}
\end{center}
Then, if $\mc{F}$ is an irreducible $(D_X, N, \eta)$-module,
\begin{align*}
i_v^!(\mc{F})&=j_v^!k_v^!(\mc{F}) \\
&= i_v^!k_{v+}j_{v+}j_v^!k_v^!(\mc{F}) \\
&=i_c^!k_{v+}R\Gamma_{C(v)}(k_v^!(\mc{F})) \\
&=i_v^!k_{v+}R\Gamma_{C(v)}(\mc{F}|_{X'}). 
\end{align*}
Here we are using Kashiwara's theorem, the fact that $\dim X = \dim X'$, and the fact that $k_v$ is an open immersion. Because $X'$ is open in $X$ and $\mc{F}$ is irreducible, $\mc{F}|_{X'}$ is irreducible as well. For all $k \in \Z$, $R^k \Gamma_{C(v)} \mc{F}|_{X'}$ is a submodule of $\mc{F}|_{X'}$, so either $R^k \Gamma_{C(v)}\mc{F}|_{X'} = 0$, or $R^k \Gamma_{C(v)} \mc{F}|_{X'} = \mc{F}|_{X'}$.  In the first case, the preceding calculation implies that $R^ki_v^!(\mc{F})=0$, and we are done. In the second case, we have $\text{supp}\mc{F}|_{X'} = \text{supp} R^k \Gamma_{C(v)} \mc{F}|_{X'} \subseteq C(v)$. By \cite[Ch. V \S 4 Cor. 4.2]{D-modulesnotes}, $\mc{F}$ is the unique irreducible holonomic $\mc{D}_X$-module that restricts to $\mc{F}|_{X'}$, and $\text{supp}\mc{F} = \overline{\text{supp}\mc{F}|_{X'}} \subseteq \overline{C(v)}$. There are no irreducible objects in $\mc{M}_{coh}(\mc{D}_X, N, \eta)$ with support equal to $\overline{C(v)}$ because $v$ is not a longest coset element, so we must have $\text{supp}\mc{F}\subseteq \partial C(v) = \overline{C(v)} - C(v)$. But this implies that $\text{supp}\mc{F}|_{X'}=\text{supp}R^k\Gamma_{C(v)}\mc{F}|_{X'}=0$, so the second case cannot happen. 
\end{proof}

Let $C \in \W$ and $\alpha \in \Pi$ be such that $Cs_\alpha < C$. The rest of this section is spent proving that $T_\alpha(\varphi(Cs_\alpha))=\nu(U_\alpha^0(\mc{L}_{Cs_\alpha}))$. Our first step in relating these two quantities is to establish the existence of a certain long exact sequence in cohomology which will be useful in relating $\mc{O}$-dimensions of modules which appear in the decomposition of $\nu(U_\alpha^0(\mc{L}_{Cs_\alpha}))$. 

Let $D \in \W$ be a coset such that $D \leq C$, so $\ell(w^D) \leq \ell(w^C)$ and $C(w^D) \subset \overline{C(w^C)}$. By \cite[Ch. 6 \S1 Prop 1.6]{localization}, $w^Cs_\alpha$ is the longest element of $Cs_\alpha$, and $\ell(w^Cs_\alpha) = \ell(w)-1$. By assumption, $C(w^C)$ is open and dense in $p_\alpha^{-1}(p_\alpha(C(w^C)))=C(w^C)\cup C(w^Cs_\alpha)$, so the closure $\overline{p_\alpha^{-1}(p_\alpha(C(w^C)))}=\overline{C(w^C)}.$ Because $C(w^D) \subset \overline{C(w^C)},$ the image $p_\alpha(C(w^D)) \subset p_\alpha(\overline{C(w^C)})$, so    
\[
C(w^D)\cup C(w^Ds_\alpha) = p_\alpha^{-1}(p_\alpha(C(w^D))) \subset \overline{p_\alpha^{-1}(p_\alpha(C(w^C)))}=\overline{C(w^C)}.
\]
We conclude that both $w^Ds_\alpha\leq w^C$ and $w^D \leq w^C$. Because both elements are less than or equal to $w^C$ in the Bruhat order, we can assume without loss of generality that $w^Ds_\alpha \leq w^D$; i.e. $\ell(w^Ds_\alpha) = \ell(w^D) -1$ and $C(w^D)$ is open in $Z_\alpha := p_\alpha^{-1}(p_\alpha(C(w^D)))=C(w^D)\cup C(w^Ds_\alpha)$.

Let $j:Z_\alpha \longrightarrow X$ and $j_D:p_\alpha(C(w^D))\longrightarrow X_\alpha$ be natural inclusions. Let $q_\alpha:Z_\alpha \longrightarrow p_\alpha(C(w^D))$ be the restriction of $p_\alpha$ to $Z_\alpha$. Then we have the following fiber product diagram:  
\begin{center}
\begin{tikzcd}
Z_\alpha \arrow[rightarrow]{r}{j}\arrow[rightarrow]{d}{q_\alpha} 
  & X\arrow[rightarrow]{d}{p_\alpha} \\
p_\alpha(C(w^D)) \arrow[rightarrow]{r}{j_D}  
  &X_\alpha.
\end{tikzcd}
\end{center}
Note that because $p_\alpha$ and $q_\alpha$ are surjective submersions, $p_\alpha^+$ and $q_\alpha^+$ are exact, so they both lift to functors on the respective derived categories $D^b(\mc{M}(\mc{D}_X))$ and $D^b(\mc{M}(\mc{D}_{Z_\alpha}))$. In the calculations below we denote both the functors on the derived category and the functors on modules by the same name, either $p_\alpha^+$ or $q_\alpha^+$. Let $d$ be the codimension of $Z_\alpha$ in $X$. Note that the codimension of $p_\alpha(C(w^D))=p_\alpha(Z_\alpha)$ in $X_\alpha$ is also $d$. Recall that for any immersion $i:Y \rightarrow X$ of smooth algebraic varieties, the extraordinary inverse image and the $\mc{D}$-module inverse image are related by $i^![\text{codim}(Y)]=Li^+$. By this relationship, base change (Theorem \ref{basechange}), and Lemma \ref{semisimplicity}, we compute

\begin{align*}
R^kj^!(U_\alpha^0(\mc{L}_{Cs_\alpha}))&=H^k(j^!p_\alpha^+p_{\alpha+}(\mc{L}_{Cs_\alpha})) \\
&=H^{k+d}(Lj^+(p_\alpha^+p_{\alpha+}(\mc{L}_{Cs_\alpha}))) \\
&=H^{k+d}(q_\alpha^+(Lj_D^+(p_{\alpha+}(\mc{L}_{Cs_\alpha})))) \\
&=H^k(q_\alpha^+j_D^!p_{\alpha+}(\mc{L}_{Cs_\alpha})) \\
&=q_\alpha^+H^k(j_D^!p_{\alpha+}(\mc{L}_{Cs_\alpha})) \\ 
&=q_\alpha^+H^k(q_{\alpha+}j^!(\mc{L}_{Cs_\alpha})). 
\end{align*}

Our next step is to analyze the complex $j^!(\mc{L}_{Cs_\alpha})$. Denote by $i:C(w^D)\longrightarrow Z_\alpha$ and $i':C(w^Ds_\alpha) \longrightarrow Z_\alpha$ the canonical affine immersions. Note that $i$ is an open immersion, and $i'$ is a closed immersion. We have the following commutative diagram. 
\begin{center}
\begin{tikzcd}
C(w^D) \arrow[rightarrow]{rd}{i}\arrow[bend left]{rrd}{i_{w^D}} \\
  
 \text{ }
  & Z_\alpha\arrow{r}{j}  
  & X \\
 C(w^Ds_\alpha) \arrow[rightarrow]{ru}{i'} \arrow[bend right]{rru}[swap]{i_{w^Ds_\alpha}} 
\end{tikzcd}
\end{center}
For any complex $\mc{F}^\cdot \in D^b(\mc{M}(\mc{D}_{Z_\alpha}))$, we have the following distinguished triangle \cite[Ch. IV \S 9]{D-modulesnotes}:
\[
i'_+ i'^!\mc{F}^\cdot  \longrightarrow \mc{F}^\cdot \longrightarrow i_+\mc{F}^\cdot |_{C(w^D)}.
\]
Applying this to $\mc{F}^\cdot = j^!(\mc{L}_{Cs_\alpha})$ and using the facts that  $j^!(\mc{L}_{Cs_\alpha})|_{C(w^D)}=i^+j^!(\mc{L}_{Cs_\alpha})=i^!j^!(\mc{L}_{Cs_\alpha})=i_{w^D}^!(\mc{L}_{Cs_\alpha})$ because $i$ is an open immersion and $i'^!\circ j^! = i_{w^Ds_\alpha}^!$, we obtain the  distinguished triangle 
\[
i'_+i_{w^Ds_\alpha}^!(\mc{L}_{Cs_\alpha}) \longrightarrow j^!(\mc{L}_{Cs_\alpha}) \longrightarrow i_+i_{w^D}^!(\mc{L}_{Cs_\alpha}).
\]
Applying the exact functor $q_{\alpha+}$ we get the following distinguished triangle in $D^b(\mc{M}(\mc{D}_{p_{\alpha}(C(w^D))}))$:
\[
(q_{\alpha}\circ i')_+(i_{w^Ds_\alpha}^!(\mc{L}_{Cs_\alpha})) \longrightarrow q_{\alpha+}j^!(\mc{L}_{Cs_\alpha}) \longrightarrow (q_{\alpha}\circ i)_+(i_{w^D}^!(\mc{L}_{Cs_\alpha})).
\]
Because $p_\alpha(C(w^D))$ is an $N$-orbit in $X_\alpha$ and all $\mc{D}$-modules in the arguments above are $N$-equivariant, the cohomologies of the complexes in this triangle are all direct sums of copies of $\mc{O}_{p_\alpha(C(w^D))}$. From this final distinguished triangle, we obtain a long exact sequence in cohomology:
\begin{align*}
\cdots \rightarrow &H^{k-1} ((q_{\alpha}\circ i)_+(i_{w^D}^!(\mc{L}_{Cs_\alpha})) \rightarrow H^k((q_{\alpha}\circ i')_+(i_{w^Ds_\alpha}^!(\mc{L}_{Cs_\alpha}))) \rightarrow \\
&H^k(q_{\alpha+}(j^!(\mc{L}_{Cs_\alpha})) \rightarrow H^{k} ((q_{\alpha}\circ i)_+(i_{w^D}^!(\mc{L}_{Cs_\alpha})) \rightarrow  \\
&H^{k+1}((q_{\alpha}\circ i')_+(i_{w^Ds_\alpha}^!(\mc{L}_{Cs_\alpha}))) \rightarrow \cdots .
\end{align*}
This is a sequence of $\mc{D}_{p_\alpha(C(w^D))}$-modules which are direct sums of copies of $\mc{O}_{p_\alpha(C(w^D))}$. 

Note that the map 
\[
q_\alpha \circ i': C(w^Ds_\alpha) \longrightarrow p_\alpha(C(w^D))
\]
is an isomorphism, and the map 
\[
q_\alpha \circ i: C(w^D) \longrightarrow p_\alpha(C(w^D))
\]
is a locally trivial projection with one-dimensional fibers. This implies that 
\begin{align}
\text{dim}_\mc{O}H^k((q_\alpha \circ i')_+(i^!_{w^Ds_\alpha}(\mc{L}_{Cs_\alpha})))&=\text{dim}_\mc{O}R^ki^!_{w^Ds_\alpha}(\mc{L}_{Cs_\alpha}), \text{ and } \label{fact3}\\
\text{dim}_\mc{O}H^k((q_\alpha \circ i)_+(i^!_{w^D}(\mc{L}_{Cs_\alpha}))) &= \text{dim}R^{k+1}i^!_{w^D}(\mc{L}_{Cs_\alpha}). \label{fact4}
\end{align}

Now we are ready to prove that $\varphi(C)=\nu(\mc{L}_C)$ satisfies \ref{KLalgorithm} (ii) by induction in the length of $w^C$. The base case is when $w^C=w_\Theta$ and $C=W_\Theta$. In this case, for any $\alpha \in \Pi$, either $Cs_\alpha = C$, or $Cs_\alpha >C$ because $w_\Theta$ is minimal length in the set of longest coset elements, so \ref{KLalgorithm}(ii) is void. 

Fix $k \in \mathbb{N}$. Assume that $\varphi(C):=\nu(\mc{L}_C)$ satisfies \ref{KLalgorithm} (ii) for $C \in \W _{\leq k}$. This is our induction assumption. Under this assumption, we can reformulate the parity condition of Lemma \ref{parity} in the following way. Since $\varphi|_{\W}$ satisfies conditions (i) and (ii) of Lemma \ref{uniqueness} on $\W _{\leq k}$, if $C \in \W _{\leq k}$ and $D \in \W$, then $P_{CD}=q^{\ell(w^C)-\ell(w^D)}Q_{CD}$, for some $Q_{CD} \in \Z[q^2, q^{-2}]$\footnote{Note that we are adopting the convention that for $D \not \leq C$, $P_{CD}=0$, and this statement is trivially true.}. Because 
\[
P_{CD}(q)=\sum_{m \in \Z} \text{dim}_\mc{O}(R^m i^! _{w^D} (\mc{L}_C))q^m, 
\]
by the definition of $\varphi$, we conclude that for any $C \in \W_{\leq k}$ and $D \in \W$, if $m \equiv \ell(w^C)-\ell(w^D)-1$ (mod 2), then $R^mi_{w^D}^!(\mc{L}_C)=0$. We refer to this as the inductive parity condition. 

Let $C \in \W$ be a coset such that $\ell(w^C)=k+1$ and $\alpha \in \Pi$ such that $Cs_\alpha <C$. Let $D \in \W$ be such that $D \leq C$. Then $Cs_\alpha \in \W_{\leq k}$, so we can apply the inductive parity condition to the cosets $Cs_\alpha$ and $D$. This yields
\begin{equation}
\label{fact1}
R^mi^!_{w^D}(\mc{L}_{Cs_\alpha})=0 \text{ for all }m \in \Z \text{ with } m \equiv \ell(w^C) - \ell(w^D) \text{ (mod 2)}.
\end{equation}
Now since we've chosen $D$ arbitrarily, there are two possible relationships between $D$ and $\alpha$. Either $Ds_\alpha=D$ or $Ds_\alpha \neq D$. In the first case, Lemma \ref{whittakercase} implies that for all $m \in \Z$, $R^mi^!_{w^Ds_\alpha}(\mc{L}_{Cs_\alpha})=0$, since $w^Ds_\alpha$ isn't a longest coset representative. In the second case, we can apply the inductive parity condition again to the cosets $Cs_\alpha$ and $Ds_\alpha$ to see that 
\begin{equation}
\label{fact2}
R^mi^!_{w^Ds_\alpha}(\mc{L}_{Cs_\alpha})=0 \text{ for all }m \in \Z \text{ with } m\equiv \ell(w^C) - \ell(w^D) + 1 \text{ (mod 2)}.
\end{equation}

Combining equations (\ref{fact1}) and (\ref{fact2}) with equations (\ref{fact3}) and (\ref{fact4}), we see that for any $D\leq C$ and any integer $m$ such that $m \equiv \ell(w^C)-\ell(w^D)+1$ (mod 2),  
\begin{align*}
H^m((q_\alpha \circ i)_+(i_{w^D}^!(\mc{L}_{Cs_\alpha})))&=0, \text{ and } \\
H^m((q_\alpha \circ i')_+(i^!_{w^Ds_\alpha}(\mc{L}_{Cs_\alpha})))&=0. 
\end{align*}
Using the long exact sequence in cohomology from earlier, we conclude that for any integer $m$ such that $m \equiv \ell(w^C)-\ell(w^D)+1$ (mod 2), 
\begin{equation*}
H^m(q_{\alpha+}j^!(\mc{L}_{Cs_\alpha}))=0.
\end{equation*}
The outcome of the this discussion is that the long exact sequence in cohomology associated to the cosets $C$ and $D$ has the form 
\[
\cdots \rightarrow 0 \rightarrow 0 \rightarrow 0 \rightarrow * \rightarrow * \rightarrow * \rightarrow 0 \rightarrow 0 \rightarrow 0 \rightarrow * \rightarrow * \rightarrow * \rightarrow 0 \rightarrow 0 \rightarrow 0 \rightarrow \cdots ,
\]
where the $*$'s represent possibly non-zero elements. Since $\mc{O}$-dimension sums over short exact sequences, we conclude after another application of equations (\ref{fact3}) and (\ref{fact4}) that for any integer $m$ such that $m \equiv \ell(w^C) + \ell(w^D)+1$ (mod 2),
\begin{equation*}
\text{dim}_\mc{O}H^m(q_{\alpha+}j^!(\mc{L}_{Cs_\alpha})) = \text{dim}_\mc{O}R^mi^!_{w^Ds_\alpha}(\mc{L}_{Cs_\alpha})+\text{dim}_\mc{O}R^{m+1}i^!_{w^D}(\mc{L}_{Cs_\alpha}).
\end{equation*}
By restricting this further to $C(w^D)$ and $C(w^Ds_\alpha)$, we see that for any $m \in \Z$,
\begin{align}
\text{dim}_\mc{O}R^mi^!_{w^D}(U^0_\alpha(\mc{L}_{Cs_\alpha})) &= \text{dim}_\mc{O}R^{m+1}i_{w^D}^!(\mc{L}_{Cs_\alpha})+ \text{dim}_\mc{O}R^mi^!_{w^Ds_\alpha}(\mc{L}_{Cs_\alpha})
\label{fact6}, \text{ and }\\
\text{dim}_\mc{O}R^mi_{w^Ds_\alpha}^!(U^0_\alpha(\mc{L}_{Cs_\alpha}))&=\text{dim}_\mc{O}R^{m}i_{w^D}^!(\mc{L}_{Cs_\alpha})+ \text{dim}_\mc{O}R^{m-1}i^!_{w^Ds_\alpha}(\mc{L}_{Cs_\alpha}). \label{fact7}
\end{align}

In addition, if $D \in \W$ has the property that $Ds_\alpha = D$, we can use Lemma \ref{whittakercase} to further reduce equations (\ref{fact6}) and (\ref{fact7}). Indeed, by Lemma \ref{whittakercase}, if $Ds_\alpha = D$,
\begin{align*}
\text{dim}_\mc{O}R^{m-1}i^!_{w^Ds_\alpha}(\mc{L}_{Cs_\alpha}) &=0, \text{ and }\\
\text{dim}_\mc{O}R^mi^!_{w^Ds_\alpha}(\mc{L}_{Cs_\alpha})&=0
\end{align*}
for all $m \in \Z_+$. By Lemma \ref{semisimplicity}, $U_\alpha^0(\mc{L}_{Cs_\alpha}) = \bigoplus_{D\leq C} m_{CD}\mc{L}_D$ for some $m_{CD} \in \Z_+$, hence Lemma \ref{whittakercase} also implies that 
\[
\text{dim}_\mc{O}R^mi_{w^Ds_\alpha}^!(U^0_\alpha(\mc{L}_{Cs_\alpha}))=0. 
\]
Therefore, we conclude that for all cosets $D\leq C$ such that $Ds_\alpha = D$, 
\begin{equation}
\label{fact8}
\text{dim}_\mc{O}R^mi^!_{w^D}(U^0_\alpha(\mc{L}_{Cs_\alpha}))=0
\end{equation}
for all $m \in \Z$. 

The equations (\ref{fact6}), (\ref{fact7}), and (\ref{fact8}) are what we need to show that $T_\alpha(\varphi(Cs_\alpha))=\nu(U^0_\alpha(\mc{L}_{Cs_\alpha}))$. The computation is as follows.  

{\allowdisplaybreaks
\begin{align*}
\nu(U^0_\alpha(\mc{L}_{Cs_\alpha})) =& \sum_{D \in \W} \sum_{m \in \Z} \text{dim}_\mc{O}(R^mi^!_{w^D}(U_\alpha^0(\mc{L}_{Cs_\alpha})))q^m \delta_D \\
=& \sum_{Ds_\alpha > D} \sum_{m \in \Z} \text{dim}_\mc{O}(R^mi^!_{w^D}(U_\alpha^0(\mc{L}_{Cs_\alpha})))q^m \delta_D \\
&+ \sum_{Ds_\alpha<D} \sum_{m \in \Z} \text{dim}_\mc{O}(R^mi^!_{w^D}(U_\alpha^0(\mc{L}_{Cs_\alpha})))q^m \delta_D \\ 
&+ \sum_{Ds_\alpha = D} \sum_{m \in \Z} \text{dim}_\mc{O}(R^mi^!_{w^D}(U_\alpha^0(\mc{L}_{Cs_\alpha})))q^m \delta_D \\
=& \sum_{Ds_\alpha < D} \sum_{m \in \Z} \text{dim}_\mc{O}(R^mi^!_{w^Ds_\alpha}(U_\alpha^0(\mc{L}_{Cs_\alpha})))q^m \delta_{Ds_\alpha} \\
&+ \sum_{Ds_\alpha<D} \sum_{m \in \Z} \text{dim}_\mc{O}(R^mi^!_{w^D}(U_\alpha^0(\mc{L}_{Cs_\alpha})))q^m \delta_D \\ 
=& \sum_{Ds_\alpha < D} \sum_{m \in \Z} (\text{dim}_\mc{O}R^{m}i_{w^D}^!(\mc{L}_{Cs_\alpha})+ \text{dim}_\mc{O}R^{m-1}i^!_{w^Ds_\alpha}(\mc{L}_{Cs_\alpha}))q^m \delta_{Ds_\alpha} \\
&+\sum_{Ds_\alpha<D} \sum_{m \in \Z}  (\text{dim}_\mc{O}R^{m+1}i_{w^D}^!(\mc{L}_{Cs_\alpha})+ \text{dim}_\mc{O}R^mi^!_{w^Ds_\alpha}(\mc{L}_{Cs_\alpha}))q^m \delta_D \\ 
=&\sum_{Ds_\alpha<D} \sum_{m \in \Z}  (\text{dim}_\mc{O}R^{m+1}i_{w^D}^!(\mc{L}_{Cs_\alpha})+ \text{dim}_\mc{O}R^mi^!_{w^Ds_\alpha}(\mc{L}_{Cs_\alpha}))q^m (\delta_D+q\delta_{Ds_\alpha}) \\
=&\sum_{Ds_\alpha<D} \sum_{m \in \Z}  \text{dim}_\mc{O}R^{m+1}i_{w^D}^!(\mc{L}_{Cs_\alpha})q^{m+1}(q^{-1}\delta_D+\delta_{Ds_\alpha})  \\
&+\sum_{Ds_\alpha>D} \sum_{m \in \Z} \text{dim}_\mc{O}R^mi^!_{w^D}(\mc{L}_{Cs_\alpha}))q^m (\delta_{Ds_\alpha}+q\delta_D) \\
= & T_\alpha(\nu(\mc{L}_{Cs_\alpha})) = T_\alpha(\varphi(Cs_\alpha)).
\end{align*}
}
Therefore, for $C \in \W_{\leq k+1}$ and $\alpha \in \Pi$ such that $Cs_\alpha < C$, 
\[
T_\alpha(\varphi(Cs_\alpha))=\nu(U^0_\alpha(\mc{L}_{Cs_\alpha})) = \nu(\bigoplus_{D\leq C} c_D \mc{L}_D) = \sum_{D \leq C} c_D \nu(\mc{L}_D) = \sum_{D\leq C} c_D \varphi(D),
\]
i.e. Theorem \ref{KLalgorithm} (ii) holds on $\W_{\leq k+1}$. By induction, this completes the proof of Proposition \ref{existence}, which in turn completes the proof of Theorem \ref{KLalgorithm}. 


\subsection{Composition multiplicities of standard Whittaker modules}
\label{Composition multiplicities of standard Whittaker modules}

We are now ready to establish the connection between Whittaker Kazhdan--Lusztig polynomials and multiplicities of irreducible Whittaker modules in standard Whittaker modules. We start with two preliminary lemmas. 

\begin{lemma}
\label{factors through Grothendieck group}
The evaluation $\nu (-1)$ of the map $\nu$ at $-1$ factors through the \linebreak Grothendieck group $K(\mc{M}_{coh}(\mc{D}_X, N, \eta))$ of $\mc{M}_{coh}(\mc{D}_X, N, \eta)$. 
\end{lemma}
\begin{proof} For an object $\mc{F}$ in $\mc{M}_{coh}(\mc{D}_X, N, \eta)$, 
\[
\nu(\mc{F})(-1) = \sum_{C \in \W} \sum_{m \in \Z} (-1)^m \dim_{\mc{O}}(R^mi_{w^C}^!(\mc{F}))\delta_C. 
\]
If $0 \rightarrow \mc{F}_1 \rightarrow \mc{F}_2 \rightarrow \mc{F}_3 \rightarrow 0$ is a short exact sequence in $\mc{M}_{coh}(\mc{D}_X, N, \eta)$, then for each $C \in \W$, we have a long exact sequence 
\[
\cdots \xrightarrow{\partial_{m-1}} R^m i_{w^C}^!(\mc{F}_1) \xrightarrow{f_m} R^m i_{w^C}^!(\mc{F}_2) \xrightarrow{g_m} R^m i_{w^C}^!(\mc{F}_3) \xrightarrow{\partial_m} R^{m+1} i_{w^C}^!(\mc{F}_1) \rightarrow \cdots 
\]
of $N$-equivariant $\eta$-twisted connections on $C(w^C)$. 
For each $m \in \Z$, we have short exact sequences 
\begin{align*}
0 &\rightarrow \ker f_m \rightarrow R^mi_{w^C}^!(\mc{F}_1) \rightarrow \im f_m \rightarrow 0, \\
0 &\rightarrow \ker g_m \rightarrow R^mi_{w^C}^!(\mc{F}_2) \rightarrow \im g_m \rightarrow 0, \text{ and }  \\
0 &\rightarrow \ker \partial_m \rightarrow R^mi_{w^C}^!(\mc{F}_3) \rightarrow \im \partial_m \rightarrow 0 .
\end{align*}
Since $\mc{O}$-dimension sums over short exact sequences and $\ker f_m = \im \partial_{m-1}$, $\ker g_m = \im f_m$, and $\ker \partial_m = \im g_m$, we have 
\begin{align*}
\sum_{m \in \Z}(-1)^m \dim_{\mc{O}}(R^mi_{w^C}^!(\mc{F}_2)) =& \sum_{m \in \Z} (-1)^m \dim_\mc{O}(R^mi_{w^C}^!(\mc{F}_1)) \\
&- \sum_{m \in \Z}(-1)^m \dim_\mc{O} \ker f_m \\
&+ \sum_{m \in \Z}(-1)^m \dim_\mc{O} (R^mi_{w^C}^!(\mc{F}_3))  \\
&- \sum_{m \in \Z}(-1)^m \dim_\mc{O} \ker \partial_m \\
=&  \sum_{m \in \Z} (-1)^m \dim_\mc{O}(R^mi_{w^C}^!(\mc{F}_1)) \\
&+\sum_{m \in \Z}(-1)^m \dim_\mc{O} (R^mi_{w^C}^!(\mc{F}_3)).
\end{align*}
This implies the result. 
\end{proof}

\begin{lemma}
\label{standard modules}
$\nu(\mc{I}_C) = \delta_C$.
\end{lemma}
\begin{proof} By definition, $\mc{I}_C = i_{w^C+}(\mc{O}_{C(w^C)})$. By Kashiwara's theorem (Theorem \ref{kashiwara}), 
\[
R^0i_{w^C}^!(\mc{I}_C) = R^0i_{w^C}^!(i_{w^C+}(\mc{O}_{C(w^C)})) = \mc{O}_{C(w^C)},
\]
and for $m \neq 0$, 
\[
R^mi_{w^C}^!(\mc{I}_C) = R^mi_{w^C}^!(i_{w^C+}(\mc{O}_{C(w^C)})) = 0.
\]
Let $D \neq C$ be another coset in $\W$. Then $i_{w^D}^{-1}(C(w^C))=0$, so by base change (Theorem \ref{basechange}), 
\[
R^mi_{w^D}^!(\mc{I}_C) = R^m i_{w^D}^!(i_{w^C+}(\mc{O}_{C(w^C)})) = 0
\]
for all $m \in \Z$. 
\end{proof}
Let $\chi:\mc{M}_{coh}(\mc{D}_X, N, \eta) \rightarrow K(\mc{M}_{coh}(\mc{D}_X, N, \eta))$ be the natural map of the category $\mc{M}_{coh}(\mc{D}_X, N, \eta)$ into its Grothendieck group $K(\mc{M}_{coh}(\mc{D}_X, N, \eta))$. 
\begin{theorem}
\label{standards in irreducible}
Let $P_{CD}$, $C,D \in \W$ be the polynomials in Theorem \ref{KLalgorithm}. Then 
\[
\chi(\mc{L}_C) = \chi(\mc{I}_C) + \sum_{D<C}P_{CD}(-1)\chi(\mc{I}_D). 
\]
\end{theorem}
\begin{proof}
By definition, $\chi(\mc{L}_C), C \in \W$ form a basis for the Grothendieck group $K(\mc{M}_{coh}(\mc{D}_X, N, \eta))$. Because $\mc{I}_C$ contains $\mc{L}_C$ as a unique irreducible submodule, and the other composition factors of $\mc{I}_C$ are $\mc{L}_D$ for $D<C$, we can see that $\chi(\mc{I}_C), C \in \W$ form another basis for the Grothendieck group. Therefore, there exist $\lambda_{CD} \in \Z$ such that 
\[
\chi(\mc{L}_C) = \sum_{D\leq C} \lambda_{CD} \chi(\mc{I}_D). 
\]
By Lemma \ref{factors through Grothendieck group}, $\nu(-1)$ factors through $K(\mc{M}_{coh}(\mc{D}_X, N, \eta))$ and by Lemma \ref{standard modules}, $\nu(\mc{I}_D) = \delta_D$, so by comparing coefficients and using the definition of $\nu$, we have
\[
\nu(\mc{L}_C) (-1) = \sum_{D \leq C} \lambda_{CD}\nu(\mc{I}_D)(-1) = \sum_{D\leq C} \lambda_{CD} \delta_D.
\]
By construction, $P_{CC} = 1$ for any $C \in \W$, so $\lambda_{CC} = 1$ and $P_{CD}(-1) = \lambda_{CD}$. This proves the theorem. 
\end{proof}
This theorem gives an algorithm for calculating the multiplicities of irreducible Whittaker modules in standard Whittaker modules. Pick a total order compatible with the partial order on $\W$. With respect to this order, the matrix $(\lambda_{CD})_{C,D \in \W}$ is lower triangular and has $1$'s on the diagonal. Here $\lambda_{CD}=P_{CD}(-1)$ as in the proof of Theorem \ref{standards in irreducible}. Let $(\mu_{CD})_{C,D \in \W}$ be the inverse matrix. From Theorem \ref{standards in irreducible}, we have 
\begin{align*}
\chi(\mc{I}_C) &= \sum_{D \in \W} \sum_{E \in \W} \mu_{CE} \lambda_{ED} \chi(\mc{I}_D) \\
&= \sum _{E \in \W} \mu_{CE} \left( \sum_{D \in \W} \lambda_{ED} \chi(\mc{I}_D) \right) \\ 
&= \sum_{E \in \W} \mu_{CE} \chi(\mc{L}_E) \\ 
&= \sum_{E \leq C} \mu_{CE} \chi(\mc{L}_E). 
\end{align*}
By Theorem \ref{global sections of costandards} and Theorem \ref{global sections of irreducibles}, we have established the main result of this paper. 
\begin{corollary}
\label{multiplicity of irreducible in standard}
The multiplicity of the irreducible Whittaker module $L(-w^D \rho, \eta)$ in the standard Whittaker module $M(-w^C \rho, \eta)$ is $\mu_{CD}$. 
\end{corollary}
We can get results analogous to Theorem \ref{standards in irreducible} and Corollary \ref{multiplicity of irreducible in standard} for integral $\lambda \in \mf{h}^*$ by twisting by a equivariant invertible $\mc{O}_X$-module. 

\begin{corollary}
\label{multiplicity integral character}
Let $\lambda \in \mf{h}^*$ be regular, integral, and antidominant. Then the multiplicity of the irreducible Whittaker module $L(w^D(\lambda - \rho), \eta)$ in the standard Whittaker module $M(w^C(\lambda - \rho), \eta)$ is $\mu_{CD}$. 
\end{corollary}

\begin{proof}
From Corollary \ref{multiplicity of irreducible in standard}, we know that in the Grothendieck group of \linebreak $\mc{M}_{coh}(\mc{D}_{-\rho}, N, \eta)$, 
\[
[\mc{I}(w^C, -\rho, \eta)]=[\mc{M}(w^C, -\rho, \eta)]=\sum_{D \in \W} \mu_{CD} [\mc{L}(w^D, -\rho, \eta)].
\]
Moreover, by the projection formula (Proposition \ref{Projection Formula}), we have $\mc{I}(w^C, -\rho, \eta)(\lambda) = \mc{I}(w^C, \lambda - \rho, \eta)$, which in turn implies that $\mc{L}(w^C, -\rho, \eta)=\mc{L}(w^C, \lambda-\rho, \eta)$ since the twist functor $-(\lambda)$ must send irreducible objects in $\mc{M}_{coh}(\mc{D}_{-\rho},N, \eta)$ to irreducible objects in $\mc{M}_{coh}(\mc{D}_{\lambda - \rho}, N, \eta)$ and each standard $\eta$-twisted Harish-Chandra sheaf has a unique irreducible subsheaf. By Theorem \ref{global sections of costandards} this implies the result.
\end{proof}

Establishing the same multiplicity results for standard Whittaker modules of arbitrary infinitesimal character requires further analysis, which we will examine in future work. It is of note that the proof of Theorem \ref{KLalgorithm} immediately implies that the coefficients of the Whittaker Kazhdan--Lusztig polynomials $P_{CD}$ are non-negative integers. 
\begin{corollary}
The coefficients of the polynomials $P_{CD}$ from Theorem \ref{KLalgorithm} are non-negative integers. 
\end{corollary}
\begin{proof}
This follows immediately from Proposition \ref{existence} and the definition of $\nu$. 
\end{proof}


\section{Whittaker Kazhdan--Lusztig polynomials}
\label{Whittaker Kazhdan-Lusztig polynomials}

This section relates the Whittaker Kazhdan--Lusztig polynomials $P_{CD}$ of Theorem \ref{KLalgorithm} to the combinatorics of Kazhdan--Lusztig polynomials appearing in \cite{Soergel97} and \cite[Ch. 5 \S2 \S3]{localization}. We also describe a duality between the Kazhdan--Lusztig algorithm for Whittaker modules established in Section \ref{A Kazhdan-Lusztig algorithm} and the Kazhdan-Lusztig algorithm for generalized Verma modules established in \cite[Ch. 6 \S3 Thm. 3.5]{localization}, following the philosophy of dual Hecke algebra modules laid out in \cite[\S12 \S13]{VoganIV}. To make these associations, we need to introduce the Hecke algebra into our story.


\subsection{The Hecke algebra} 
\label{The Hecke algebra}

Let $(W,S)$ be a Coxeter system with length function $\ell:W \rightarrow \mathbb{N}$.

\begin{definition}
The \emph{Hecke algebra} $\mc{H}=\mc{H}(W,S)$ of the Coxeter system $(W,S)$ is the associative algebra over $\Z[q, q^{-1}]$ with generators $\{H_s\}_{s \in S}$ satisfying the relations
\begin{enumerate}[label=(\roman*)]
\item (quadratic) 
\[
(H_s+q)(H_s-q^{-1})=0 \text{ for all } s \in S, \text{ and }
\] 
\item (braid) for each pair $s,t \in S$, 
\[
H_sH_tH_s \cdots = H_t H_s H_t \cdots  \\
\]
with $m_{st}$ elements on each side of the equality. (Here $m_{st}$ is the order of $st$ in $W$.) 
\end{enumerate}
\end{definition}

All $H_s$ for $s \in S$ are invertible with $H_s^{-1}=H_s+(q-q^{-1})$. For $w \in W$, we choose a reduced expression $rs\cdots t$ of $w$ and define $H_w \in \mc{H}$ by $H_rH_s \cdots H_t$. This element is independent of choice of reduced expression. If $\ell(w)+\ell(v)=\ell(wv)$, then we have $H_wH_v=H_{wv}$. There is exactly one ring homomorphism 
\begin{align*}
d: \mc{H}& \rightarrow \mc{H} \\
H & \mapsto \overline{H}
\end{align*}
such that $\overline{q}=q^{-1}$ and $\overline{H}_w=(H_{w^{-1}})^{-1}$. This is clearly an involution. We say that $H \in \mc{H}$ is \emph{self-dual} if $\overline{H}=H$. For each $s \in S$, the element $C_s:=H_s+q$ is self-dual. Indeed, $\overline{C_s}=(H_s)^{-1}+q^{-1}=H_s+q=C_s$.


\subsection{$\mathcal{H}_\Theta$ is a Hecke algebra module}
\label{H_Theta is a Hecke algebra module}

Now we return to the setting of Section \ref{A Kazhdan-Lusztig algorithm}. Let $W$ be the Weyl group of a reduced root system $\Sigma$ with simple roots $\Pi \subset \Sigma$ and corresponding simple reflections $S \subset W$. Then $(W,S)$ is a Coxeter system. Let $\Theta \subset \Pi$ be a fixed subset of simple roots and let $\mc{H}_\Theta = \bigoplus_{C \in \W} \Z[q,q^{-1}] \delta_C$ be the $\Z[q,q^{-1}]$-module from Theorem \ref{KLalgorithm}. Recall that for each $\alpha \in \Pi$ we defined a $\Z[q,q^{-1}]$-linear endomorphism $T_\alpha$ of $\mc{H}_\Theta$ by 
\[
T_\alpha(\delta_C)=
\begin{cases}
0 & \text{ if }Cs_\alpha=C \\
q\delta_C+\delta_{Cs_\alpha} & \text{ if } Cs_\alpha>C \\
q^{-1}\delta_C + \delta_{Cs_\alpha} & \text{ if } Cs_\alpha<C 
\end{cases}.
\]
Our first observation is that the operators $\{T_\alpha\}_{\alpha \in \Pi}$ give an action of the Hecke algebra of $(W,S)$ on $\mc{H}_\Theta$. Indeed, if we define $S_\alpha:=T_\alpha-q$, then a computation shows that $S_\alpha$ satisfies both the quadratic and braid relations of the Hecke algebra, thus the map $\psi: \mc{H} \rightarrow \End_{\Z[q, q^{-1}]}(\mc{H}_\Theta)$ given by $\psi(H_{s_\alpha})=S_\alpha$ gives $\mc{H}_\Theta$ the structure of a left $\mc{H}$-module. The map $\psi$ sends  the self-dual basis element $C_{s_\alpha} \in \mc{H}$ described in the previous section to the endomorphism $T_\alpha$. 

This extra structure will allow us to relate Theorem \ref{KLalgorithm} to the results in \cite[\S2 \S3]{Soergel97}. Our first step is to establish a relationship between $\mc{H}_\Theta$ and a certain induced right $\mc{H}$-module (the antispherical module for the Hecke algebra) in order to extend the duality in $\mc{H}$ given by the involution $d$ to a duality in $\mc{H}_\Theta$. If $S_\Theta \subset S$ is the subset of simple reflections corresponding to $\Theta \subset \Pi$, then the subalgebra $\mc{H}^\Theta$ of $\mc{H}$ generated by $\{H_{s_\alpha}\}$ for $\alpha \in \Theta$ is isomorphic to the Hecke algebra of the Coxeter system $(W_\Theta, S_\Theta)$. The surjection $\mc{H}^\Theta \twoheadrightarrow \Z[q, q^{-1}]$ sending $H_{s_\alpha} \mapsto -q$ gives $\Z[q,q^{-1}]$ the structure of a $\mc{H}^\Theta$-bimodule, and with this bimodule structure we can form the induced right $\mc{H}$-module
\[
\mc{N}^\Theta:=\Z[q,q^{-1}] \otimes _{\mc{H}^\Theta} \mc{H}.
\]
This is the \emph{antispherical module} of the Hecke algebra $\mc{H}$. Note that in the special case $\Theta = \emptyset$, $\mc{N}^\Theta$ is the Hecke-algebra $\mc{H}$ as a module over itself with the right regular action. The set $\{N_w:=1 \otimes H_w\}$ for minimal coset representatives $w \in C \in \W$ forms a basis for $\mc{N}^\Theta$ as a $\Z[q,q^{-1}]$-module. 

\begin{remark}
\label{spherical module}
By instead using the surjection $\mc{H}^\Theta \twoheadrightarrow \Z[q,q^{-1}]$ given by $H_{s_\alpha} \mapsto q^{-1}$ to form the $\mc{H}^\Theta$-bimodule structure on $\Z[q,q^{-1}]$, it is possible to construct another induced right $\mc{H}$-module $\mc{M}^\Theta:=\Z[q,q^{-1}]\otimes_{\mc{H}^\Theta}\mc{H}$ \cite[\S3]{Soergel97}. This is the {\em spherical module} of the Hecke algebra $\mc{H}$. This module also has the property that $\mc{M}^\emptyset = \mc{H}$. By an analogous argument to the one below, one can show that the Kazhdan--Lusztig combinatorics of generalized Verma modules (as described in \cite[Ch. 6 \S3]{localization}) is given by the spherical $\mc{H}$-module. 
\end{remark}

One can compute \cite{Soergel97} that the action of $C_s$ on $\mc{N}^\Theta$ for $s \in S$ is given by 
\[
N_w C_s = 
\begin{cases}
0 & \text{ if } ws \in C\\
qN_w + N_{ws} & \text{ if } ws>w \text{ and }ws \not \in C \\
q^{-1} N_w+ N_{ws} & \text{ if } ws<w \text{ and } ws \not \in C 
\end{cases}.
\]
Therefore, there is a $\Z[q,q^{-1}]$-module isomorphism
\begin{align*}
\phi: \mc{H}_\Theta &\rightarrow \mc{N}^\Theta \\
\delta_C & \mapsto N_{w_\Theta w^C}
\end{align*}
which intertwines the left $\mc{H}$-action on $\mc{H}_\Theta$ with the right $\mc{H}$-action on $\mc{N}^\Theta$. That is, for $E \in \mc{H}_\Theta$, $\phi(C_{s_\alpha}E)=\phi(E)C_{s_\alpha}$. Here $w_\Theta$ is the longest element in $W_\Theta$. 

Note that in the special case $\Theta = \emptyset$, this provides an $\Z[q,q^{-1}]$-module isomorphism between $\mc{H}_\emptyset$ and the Hecke algebra $\mc{H}$.\footnote{This justifies the notational choice in \cite[Ch. 5 \S2]{localization}, where the $\Z[q,q^{-1}]$-module $\mc{H}_\emptyset$ is referred to as $\mc{H}$.} The benefit of relating $\mc{H}_\Theta$ to this induced module is that it allows us to use the involution $d$ of $\mc{H}$ to construct an involution of the induced module, which we can then use to define self-duality in $\mc{H}_\Theta$. There is a homomorphism of additive groups 
\begin{align*}
\mc{N}^\Theta & \rightarrow \mc{N}^\Theta \\
a \otimes H & \mapsto \overline{a \otimes H}:= \overline{a} \otimes \overline{H}.
\end{align*}
This homomorphism has the property that $\overline{N}_e = N_e $ and 
\begin{equation}
\label{duality distribution}
\overline{NH}=\overline{N} \hspace{1mm} \overline{H}
\end{equation}
for all $N \in \mc{N}^\Theta$ and $H \in \mc{H}$. We say that an element $E \in \mc{H}_\Theta$ is \emph{self-dual} if the corresponding element in $\mc{N}^\Theta$ is fixed under this involution; that is, if $\overline{\phi(E)}=\phi(E)$. Since $\phi(T_\alpha(E))=\phi(E)C_{s_\alpha}$ for any $\alpha \in \Pi$ and $E \in \mc{H}_\Theta$ and $C_{s_\alpha}$ is self-dual in $\mc{H}$, property (\ref{duality distribution}) implies that $T_\alpha$ preserves self-duality. 


\subsection{The recursion relation in Theorem \ref{KLalgorithm} is equivalent to self-duality}
\label{Recursion is equivalent to duality}

The main content of this section is a proof that condition (ii) in Theorem \ref{KLalgorithm} is equivalent to $\varphi(C)$ being self-dual in the sense of the preceding section. 

\begin{theorem}
\label{KLalgorithmequivalence}
Let $\varphi: \W \rightarrow \mc{H}_\Theta$ be a function satisfying 
\begin{equation}
\label{KLcondition}
\varphi(C)=\delta_C+\sum_{D<C}P_{CD}\delta_D \text{ for } P_{CD} \in q\Z[q]
\end{equation}
for all $C \in \W$. Then the following are equivalent. 
\begin{enumerate}[label=(\roman*)]
\item If $\alpha \in \Pi$ and $C \in \W$ are such that $Cs_\alpha<C$, then there exist $m_D \in \Z$ such that 
\[
T_\alpha(\varphi(Cs_\alpha))=\sum_{D \leq C} m_D \varphi(D).
\]
\item All $\varphi(C)$ are self-dual. 
\end{enumerate}
\end{theorem}

\begin{proof}
Assume that (i) holds, and take $C$ and $\alpha$ such that $Cs_\alpha<C$. Using the definition of $T_\alpha$ we compute 
\begin{align*}
T_\alpha(\varphi(Cs_\alpha))&=T_\alpha(\delta_{Cs_\alpha}+\sum_{E<Cs_\alpha}P_{Cs_\alpha E}\delta_E)\\
&= \delta_C + q\delta_{Cs_\alpha} + \sum_{E<Cs_\alpha}P_{Cs_\alpha E}T_\alpha(\delta_E)\\
&= \delta_C + \sum_{D<C} Q_{CD} \delta_C
\end{align*}
for some $Q_{CD} \in \Z[q]$. Therefore, $m_C=1$. Thus, for any $\alpha \in \Pi$ such that $Cs_\alpha<C$,
\begin{equation}
\label{selfdual}
\varphi(C)=T_\alpha(\varphi(Cs_\alpha))-\sum_{D<C}m_D\varphi(D).
\end{equation}
Now we show that all $\varphi(C)$ are self-dual by induction in $\ell(w^C)$. If $C=W_\Theta$, then $\varphi(W_\Theta)=\delta_{W_\Theta}$ is self-dual because $\phi(\delta_{W_\Theta})=1 \otimes H_e$ and $\overline{H}_e=H_e$ in $\mc{H}$. Assume $\varphi(D)$ is self-dual for all $D<C$. Then because $T_\alpha$ preserves self-duality, equation (\ref{selfdual}) implies that $\varphi(C)$ is self-dual. We conclude that (i) implies (ii). 

Now let $\varphi:\W \rightarrow \mc{H}_\Theta$ be a function satisfying equation \ref{KLcondition} and condition (ii). For $C\in \W$, choose $\alpha\in \Pi$ such that $Cs_\alpha <C$. If no such $\alpha$ exists, then (i) is void and we are done. If such an $\alpha$ does exist, we have 
\[
T_\alpha(\varphi(Cs_\alpha))=\delta_C+\sum_{D<C}Q_{CD}\delta_D \
\]
for appropriately chosen $Q_{CD} \in \Z[q]$. Define 
\[
\widetilde{\varphi}(C):=T_\alpha(\varphi(Cs_\alpha))-\sum_{D<C}Q_{CD}(0)\varphi(D).
\]
The function $\widetilde{\varphi}$ satisfies equation (\ref{KLcondition}) and is self-dual by the fact that $T_\alpha$ preserves self-duality. Next we argue that there is a unique function satisfying both equation \ref{KLcondition} and condition (ii), and thus $\widetilde{\varphi}=\varphi$. First, observe that for any $E \in \sum_{C \in \W} q \Z[q] \delta_C$, self-duality implies $E=0$. Indeed, if $E=\sum_{C \in \W} R_C \delta_C$ and we let $C$ be maximal such that $R_C \neq 0$, then $\overline{\phi(E)}=\phi(E)$ implies that $\overline{R}_C=R_C$, which is impossible because $R_C \in q\Z[q]$. Therefore, if $\varphi':\W \rightarrow \mc{H}_\Theta$ and $\varphi:\W \rightarrow \mc{H}_\Theta$ are two functions satisfying equation (\ref{KLcondition}) and (ii), then $\varphi(C)-\varphi'(C) \in \sum_{C \in \W} q \Z[q] \delta_C$ is self-dual, so $\varphi(C)=\varphi'(C)$. 

We conclude that $\widetilde{\varphi}=\varphi$, and by rearranging we obtain
\[
T_\alpha(\varphi(Cs_\alpha))=\sum_{D\leq C}m_D \varphi(D) \text{ for } m_D=\begin{cases} Q_{CD}(0) & \text{ if } D<C \\ 1 & \text{ if } D=C \end{cases}. 
\]
Thus  (ii) implies (i).
\end{proof}

This establishes the relationship between the results in this paper and the results in \cite[\S2 \S3]{Soergel97}. In particular, it establishes that Theorem \ref{KLalgorithm} in this paper is equivalent to part 2 of Theorem 3.1 in \cite{Soergel97}. This allows us to explicitly compare Whittaker Kazhdan--Lusztig polynomials $P_{CD}$ to polynomials that have shown up elsewhere in the  literature under the name ``parabolic Kazhdan--Lusztig polynomials.'' We list these relationships now. 
\begin{remark}
\label{conversions}
\begin{enumerate}
\item The Whittaker Kazhdan--Lusztig polynomials $P_{CD}$ are equal to the polynomials $n_{y,x}$ in \cite{Soergel97} for $x=w_\Theta w^C$ and $y=w_\Theta w^D$. 
\item A normalization of $P_{CD}$ gives the parabolic Kazhdan--Lusztig polynomials in \cite{Deodhar}. The polynomials 
\[
(q^{\ell(w_\Theta w^D)} - q^{\ell(w_\Theta w^C)})P_{CD}
\]
are polynomials in the variable $v:=q^{-2}$, and they are precisely the polynomials $P^I_{(w_\Theta w^D)^{-1}, (w_\Theta w^D)^{-1}}$ in \cite{Deodhar} for $u=v$ and $W_\Theta=W_I$. 
\item In the special case where $\Theta = \emptyset$, the polynomials 
\[
(q^{\ell(v)}-q^{\ell(w)})P_{wv}
\]
are the Kazhdan--Lusztig polynomials as defined in \cite{KL}. 

\end{enumerate}
\end{remark}


\subsection{Duality of Whittaker modules and generalized Verma modules}
\label{Combinatorial duality of Whittaker modules and generalized Verma modules}

We conclude this paper by relating the Whittaker Kazhdan--Lusztig polynomials $P_{CD}$ to the polynomials arising in the Kazhdan--Lusztig algorithm for generalized Verma modules established in \cite[Ch. 6 \S3]{localization}. Generalized Verma modules are a class of parabolically induced highest weight modules for a Lie algebra. For details of their construction, see \cite[Ch. 6]{localization}. The main results of this section are equation (\ref{KLinversionWhittakergVerma}) which relates the algorithm in Theorem \ref{KLalgorithm} to the algorithm in \cite[Ch. 6 Thm. 3.5]{localization}, and Proposition \ref{dualityofKLpolys}, which provides a formula relating Whittaker Kazhdan--Lusztig polynomials to Kazhdan--Lusztig polynomials. By Theorem \ref{KLalgorithmequivalence}, Proposition \ref{dualityofKLpolys} is a special case of \cite[Prop. 3.4]{Soergel97}, but our proof is new, and independent of results in \cite{Soergel97}. Equation (\ref{KLinversionWhittakergVerma}) also recovers the Kazhdan--Lusztig inversion formulas of \cite{KL} as a special case. 

In \cite[Ch. 6 \S3]{localization}, Mili\v{c}i\'{c} establishes a Kazhdan--Lusztig algorithm for generalized Verma modules. We review his results here to establish their relationship with the Whittaker Kazhdan--Lusztig algorithm of this paper. Let $\mc{H}_\Theta=\bigoplus_{C \in \W} \Z[q,q^{-1}] \delta_C$ be the $\Z[q,q^{-1}]$-module from the preceding section. We can realize $\mc{H}_\Theta$ as a $\Z[q,q^{-1}]$-submodule of the $\Z[q,q^{-1}]$-module $\mc{H}_\emptyset=\bigoplus_{w \in W}\Z[q,q^{-1}]\delta_w$ by setting 
\[
\delta_C=\sum_{v \in W_\Theta}q^{\ell(v)}\delta_{vw^C}.
\]
For $\alpha \in \Pi$, let $T_\alpha^\emptyset:\mc{H}_\emptyset \rightarrow \mc{H}_\emptyset$ be the endomorphism defined by 
\[
T_\alpha^\emptyset(\delta_w) = \begin{cases} q\delta_w + \delta_{ws_\alpha} & \text{ if } ws_\alpha>w \\ 
q^{-1} \delta_w + \delta_{ws_\alpha} & \text{ if } ws_\alpha<w \end{cases},
\]
as in Section \ref{H_Theta is a Hecke algebra module}. We introduce $\emptyset$ into the notation here to emphasize that $T_\alpha^\emptyset$ is an endomorphism of $\mc{H}_\emptyset$.  A computation shows that the endomorphism $T_\alpha^\emptyset$ transforms $\delta_C$ in the following way:
\[
T_\alpha^\emptyset(\delta_C)=\begin{cases} (q+q^{-1})\delta_C & \text{ if }Cs_\alpha =C;\\
q \delta_C + \delta_{Cs_\alpha} & \text{ if } Cs_\alpha<C;\\ 
q^{-1} \delta_C + \delta_{Cs_\alpha} & \text{ if } Cs_\alpha>C.
\end{cases}
\]
It follows that $\mc{H}_\Theta$ is stable under $T_\alpha^\emptyset$, so $\mc{H}_\Theta$ is an $\mc{H}$-submodule of $\mc{H}_\emptyset$. In \cite[Ch. 6 \S3]{localization}, Mili\v{c}i\'{c} proves the following Kazhdan-Lusztig algorithm for generalized Verma modules. 
\begin{theorem}
\label{KLalgorithmgvermas}
\cite[Ch. 6 \S3 Thm. 3.5]{localization}
There exists a unique function $\varphi':\W \rightarrow \mc{H}_\Theta$ satisfying the following. 
\begin{enumerate}[label=(\roman*)]
\item For $C \in \W$, 
\[
\varphi'(C)=\delta_C+\sum_{D<C}P'_{CD}\delta_D 
\]
for $P_{CD}' \in q\Z[q]$, and 
\item for $\alpha \in \Pi$ such that $Cs_\alpha<C$, there exist integers $m'_D$ such that 
\[
T^\emptyset_\alpha(\varphi'(Cs_\alpha))=\sum_{D\leq C} m'_D \varphi'(D).
\]
\end{enumerate}
Furthermore, the polynomials $P_{CD}'$ are given by the Kazhdan--Lusztig polynomials for $(W,S)$ by 
\[
P_{CD}'=P_{w^Cw^D}.
\]
\end{theorem}

Since Theorem \ref{KLalgorithm} specializes to the Kazdhan--Lusztig algorithm for Verma modules \cite[Ch. 5 \S2 Thm. 2.1]{localization} when $\Theta = \emptyset$, one can see from Mili\v{c}i\'{c}'s proof of Theorem \ref{KLalgorithmgvermas} that the unique function $\varphi':\W \rightarrow \mc{H}_\Theta$ satisfying Theorem \ref{KLalgorithmgvermas} is the function $\varphi'(D):=\varphi_\emptyset(w^D)$, where $\varphi_\emptyset:W \rightarrow \mc{H}_\emptyset$ is the unique function guaranteed by Theorem \ref{KLalgorithm} in the special case $\Theta=\emptyset$. The Kazhdan--Lusztig polynomials $P_{CD}'$ of Theorem \ref{KLalgorithmgvermas} describe the multiplicities of irreducible highest weight modules in generalized Verma modules \cite[Ch. 6 \S3 Cor. 3.7]{localization}. 

For arbitrary $\Theta \subset \Pi$, the Whittaker Kazhdan--Lusztig polynomials are inverse to the polynomials appearing in Theorem \ref{KLalgorithmgvermas} in the following sense. 
\begin{equation}
\label{KLinversionWhittakergVerma}
\sum_{E \in \W} (-1)^{\ell(w^E) + \ell(w^C)} P'_{Cw_0Ew_0} P_{DE} = \begin{cases} 1 & \text{ if } C=D \\ 0 & \text{ if } C \neq D \end{cases}. 
\end{equation}
This relationship appears as Proposition 3.9 in \cite{Soergel97}, where it is originally attribued to Douglass \cite{Douglass}. If we specialize to $\Theta =\emptyset$, then $\W=W$, equation (\ref{KLinversionWhittakergVerma}) recovers the Kazhdan--Lusztig inversion formulas. 
\begin{equation}
\label{KLinversionformula}
\sum_{u \in W} (-1)^{\ell(u) + \ell(w)} P_{wu} P_{vw_0 \hspace{1mm} uw_0} = \begin{cases} 1 & \text{ if } v=w \\ 0 & \text{ if } v \neq w \end{cases}. 
\end{equation}

We complete this section by describing the relationship between the Whittaker Kazhdan--Lusztig polynomials $P_{CD}$ and the Kazhdan--Lusztig polynomials in \cite{localization}. If $\Theta = \emptyset$, Theorem \ref{KLalgorithm} specializes the algorithm in \cite[Ch. 5 \S 2 Thm. 2.1]{localization}, and the polynomials $P_{wv}$ are the Kazhdan--Lusztig polynomials as defined in \cite{localization}. Note that these polynomials differ in normalization from the Kazhdan--Lusztig polynomials appearing in \cite{KL}; see Remark \ref{conversions}. The following formula relates Whittaker Kazhdan--Lusztig polynomials for general $\Theta$ to Kazhdan--Lusztig polynomials. 
\begin{proposition}
\label{dualityofKLpolys}
For $\Theta \subset \Pi$ arbitrary, 
\[
P_{CD}=\sum_{v \in W_\Theta} (-q)^{\ell(v)} P_{w_\Theta w^C \hspace{1mm} v w_\Theta w^D}.
\]
\end{proposition}
\begin{proof}
Fix an arbitrary $\Theta \subset \Pi$, and pick a total order compatible with the partial order on $\W$. From Theorem \ref{KLalgorithmgvermas} we see that $P_{CD}'=0$ for $D>C$ and $P'_{CD}=1$ if $C=D$, so the matrix $P=(P'_{CD})$ of polynomials with respect to our total order is lower triangular with 1's on the diagonal and coefficients in $\Z[q]$. The inverse matrix $Q=(Q_{CD})$ is also lower triangular with 1's on the diagonal and coefficients in $\Z[q]$. From equation (\ref{KLinversionWhittakergVerma}) we see that the coefficients $Q_{CD}$ of the inverse matrix are related to Whittaker Kazhdan--Lusztig polynomials in the following way:
\begin{equation}
\label{QandWhittaker}
Q_{CD}=(-1)^{\ell(w^C)+\ell(w^D)}P_{Dw_0Cw_0}.
\end{equation}
Then, if $\varphi_\emptyset:W \rightarrow \mc{H}_\emptyset$ is the unique function from Theorem \ref{KLalgorithm} corresponding to the subset $\Theta=\emptyset$, we have 
\begin{align*}
\sum_{D \in \W} Q_{CD}\varphi_\emptyset(w^D)&=\sum_{D \in \W}Q_{CD}\left( \sum_{E \in \W}P_{DE}'\delta_E \right) \\
&= \sum_{E \in \W} \left( \sum_{D \in \W} Q_{CD}P'_{DE} \right) \delta_E \\
&= \delta_C. 
\end{align*}
Here the polynomials $Q_{CD}$ correspond to our arbitrary fixed $\Theta$, and only the function $\varphi_\emptyset$ is specific to the special case $\Theta = \emptyset$. Now, if we specialize further to the case that our fixed $\Theta$ is $\Theta = \emptyset$, the computation above implies 
\begin{equation}
\label{specialsum}
\sum_{v \in W}Q_{wv}\varphi(v) = \delta_w.
\end{equation}
Then, because 
\[
\delta_C=\sum_{v \in W_\Theta} q^{\ell(v)}\delta_{vw^C}, 
\]
we have the following relationship: 
\begin{align*}
\sum_{D \in \W} Q_{CD}\varphi(w^D) & = \sum_{v \in W_\Theta} q^{\ell(v)}\delta_{vw^C} \\
&=\sum_{v \in W_\Theta} q^{\ell(v)} \left( \sum_{u \in W} Q_{vw^C\hspace{1mm}u}\varphi(u) \right) \\
&= \sum_{u \in W} \left( \sum_{v \in W_\Theta} q^{\ell(v)} Q_{vw^C \hspace{1mm}u} \right) \varphi(u).
\end{align*}
Here the second equality follows from equation (\ref{specialsum}). Since $\{\varphi(u): u \in W\}$ form a basis for $\mc{H}_\emptyset$ by Theorem \ref{KLalgorithm}, this implies that 
\[
Q_{CD}=\sum_{v \in W_\Theta} q^{\ell(v)}Q_{vw^C\hspace{1mm}w^D}. 
\]
Thus, since $\ell(vw^C)=\ell(w^C)-\ell(v)$ for $v \in W_\Theta$ by \cite[Ch. 6 \S1 Lem. 1.8]{localization}, an application of equation (\ref{QandWhittaker}) for the special case $\Theta = \emptyset$ results in the following formula: 
\begin{equation}
\label{Qs}
Q_{CD}=(-1)^{\ell(w^C)+\ell(w^D)}\sum_{v \in W_\Theta}(-1)^{\ell(v)} q^{\ell(v)}P_{w^Dw_0 \hspace{1mm} vw^Cw_0}.
\end{equation}
The element $w^Cw_0$ is the shortest element of the coset $Cw_0$, so it is equal to $w_\Theta w^{Cw_0}$ by \cite[Ch. 6 \S1 Thm. 1.4]{localization}. The proposition then follows by combining equation (\ref{Qs}) with equation (\ref{QandWhittaker}).
\end{proof}


\appendix
\section{Geometric preliminaries}
\label{Geometric preliminaries}

In this appendix we record some some fundamental results about functors between categories of modules over twisted sheaves of differential operators which play a critical role in the arguments of Sections \ref{Geometric description of Whittaker modules} and \ref{A Kazhdan-Lusztig algorithm}. For a detailed treatment of this subject, see \cite{HMSWI,D-modules,localization}. 


\subsection{Twisted sheaves of differential operators}
\label{Twisted sheaves of differential operators}

 Let $X$ be a smooth complex algebraic variety of dimension $n$. Denote by $\mc{O}_X$ the structure sheaf of $X$, $\mc{D}_X$ the sheaf of differential operators on $X$, $\mc{T}_x$ the tangent sheaf on $X$, $\Omega_X$ the cotangent sheaf on $X$, and $\omega_X$ the invertible $\mc{O}_X$-module of differential $n$-forms on $X$. Denote by $i_X:\mc{O}_X \rightarrow \mc{D}_X$ the natural inclusion. A \textit{twisted sheaf of differential operators on $X$} is a pair $(\mc{D}, i)$ of a sheaf $\mc{D}$ of associative $\C$-algebras with identity on $X$ and a homomorphism $i:\mc{O}_X \rightarrow \mc{D}$ of sheaves of $\C$-algebras with identity that is locally isomorphic to the pair $(\mc{D}_X, i_X)$.

For $f:Y \rightarrow X$ a morphism of smooth algebraic varieties and $\mc{D}$ a twisted sheaf of differential operators on $X$, we define
\[
\mc{D}_{Y \rightarrow X} =  \mc{O}_Y \otimes_{f^{-1}\mc{O}_X} f^{-1} \mc{D}. 
\]
Then $\mc{D}_{Y \rightarrow X}$ is a left $\mc{O}_Y$-module for left multiplication and a right $f^{-1}\mc{D}$-module for right multiplication on the second factor.  Denote by $\mc{D}^f$ the sheaf of differential $\mc{O}_Y$-module endomorphisms of $\mc{D}_{Y \rightarrow X}$ which are also $f^{-1}\mc{D}$-module endomorphisms. There is a natural morphism of sheaves of algebras $i_f: \mc{O}_Y \rightarrow \mc{D}^f$, and the pair $(\mc{D}^f, i_f)$ is a twisted sheaf of differential operators on $Y$.

Let $\mc{D}$ be a twisted sheaf of differential operators on $X$ and $\mc{L}$ an invertible $\mc{O}_X$-module. The \emph{twist} of $\mc{D}$ by $\mc{L}$ is the sheaf $\mc{D}^\mc{L}$  of differential $\mc{O}_X$-module endomorphisms of $\mc{L} \otimes_{\mc{O}_X} \mc{D}$ that commute with the right $\mc{D}$-action. Because $\mc{L} \otimes _{\mc{O}_X } \mc{D}$ is an $\mc{O}_X$-module for left multiplication, there is a natural homomorphism $i_\mc{L}:\mc{O}_X \rightarrow \mc{D}^\mc{L}$, and $(\mc{D}^\mc{L}, i_\mc{L})$ is a twisted sheaf of differential operators on $X$. If $f:Y \rightarrow X$ is a morphism of smooth algebraic varieties as above, $(\mc{D}^\mc{L})^f = (\mc{D}^f)^{f^*(\mc{L})}$.

 If $X$ is a homogeneous space for a group $G$ with Lie algebra $\mf{g}$, then a \textit{homogeneous twisted sheaf of differential operators} on $X$ is a triple $(\mc{D}, \gamma, \alpha)$, where $\mc{D}$ is a twisted sheaf of differential operators on $X$, $\gamma$ is the algebraic action of $G$ on $X$, and $\alpha:\ug \rightarrow \Gamma(X,\mc{D})$ is a morphism of algebras such that the following three conditions are satisfied: 
\begin{enumerate}[label=(\roman*)]
\item the multiplication in $\mc{D}$ is $G$-equivariant;
\item the differential of the $G$-action on $\mc{D}$ agrees with the action $T \mapsto [\alpha(\xi), T]$ for $\xi \in \mf{g}$ and $T \in \mc{D}$; and 
\item the map $\alpha: \mc{U}(\mf{g}) \rightarrow \Gamma(X, \mc{D})$ is a morphism of $G$-modules. 
\end{enumerate}  
For $x \in X$, denote by $B_x$ the stabilizer of $x$ in $G$ and $\mf{b}_x$ its Lie algebra. For each $B_x$-invariant linear form $\lambda \in \mf{b}_x^*$ one can construct a homogeneous twisted sheaf of differential operators $\mc{D}_{X,\lambda}$ \cite[App. A \S 1]{HMSWI} and all homogeneous twisted sheaves of differential operators on $X$ occur in this way. 

If $\mc{A}$ is a sheaf of $\C$-algebras  on $X$, we denote by $\mc{A}^\circ$ the opposite sheaf%
\index{opposite sheaf} of $\C$-algebras on $X$. Then if $(\mc{D}, i)$ is a twisted sheaf of differential operators on a smooth algebraic variety $X$, $(\mc{D}^\circ, i)$ is also a twisted sheaf of differential operators on $X$. In particular, the pair $(\mc{D}_X^\circ, i_X)$ is a twisted sheaf of differential operators, and it is naturally isomorphic to $(\mc{D}_X^{\omega_X}, i_{\omega_X})$. If $X$ is a homogeneous space and $\delta$ is the $B_x$-invariant linear form which is the differential of the representation of $B_x$ on the top exterior power of the cotangent space at $x$, then $(\mc{D}_{X,\lambda})^\circ$ is naturally isomorphic to $\mc{D}_{X, -\lambda + \delta}$. 


\subsection{Modules over twisted sheaves of differential operators}
\label{Modules over twisted sheaves of differential operators}

Let $\mc{D}$ be a twisted sheaf of differential operators on a smooth complex algebraic variety $X$. For a category $\mc{M}(\mc{D})$ of $\mc{D}$-modules, we denote by $\mc{M}_{qc}(\mc{D})$ (resp. $\mc{M}_{coh}(\mc{D})$) the corresponding category of quasicoherent (resp. coherent) $\mc{D}$-modules. We can view left $\mc{D}$-modules as right $\mc{D}^\circ$-modules and vice-versa. In other words, the category $\mc{M}^L_{qc}(\mc{D})$ of quasicoherent left $\mc{D}$-modules on $X$ is isomorphic to the category $\mc{M}_{qc}^R(\mc{D}^\circ)$ of quasicoherent right $\mc{D}^\circ$-modules on $X$. This relationship allows us to freely use right or left modules depending on the particular situation, and because of this, we frequently drop the exponents `L' and `R' from our notation. 

For a coherent $\mc{D}$-module $\mc{V}$, we can define the \textit{characteristic variety} Ch$\mc{V}$ of $\mc{V}$ in the same way as the non-twisted case \cite[Ch. III \S 3]{D-modulesnotes}. Because this construction is local, the results in the non-twisted case carry over to our setting. In particular, we have the following structure:
\begin{enumerate}[label=(\roman*)]
\item Ch$\mc{V}$ is a conical subvariety of the cotangent bungle $T^*(X)$. 
\item dim$(\text{Ch}\mc{V})\geq \text{dim}(X)$. 
\end{enumerate}
If dim$(\text{Ch}\mc{V})= \text{dim}(X)$, we say that $\mc{V}$ is a \textit{holonomic} $\mc{D}$-module%
\index{holonomic!$\mc{D}$-module}. Holonomic $\mc{D}$-modules form a thick subcategory $\mc{M}_{hol}(\mc{D})$ of $\mc{M}_{coh}(\mc{D})$. If $\mc{V}$ in $\mc{M}_{coh}(\mc{D})$ is coherent as an $\mc{O}_X$-module, we call $\mc{V}$ a \textit{connection}. Connections are locally free as $\mc{O}_X$-modules and their characteristic variety is the zero section of $T^*(X)$, so they are holonomic.

For an invertible $\mc{O}_X$-module $\mc{L}$ and a twisted sheaf $\mc{D}$ of differential operators on $X$, we define the \emph{twist functor} from $\mc{M}_{qc}^L(\mc{D})$ into $\mc{M}_{qc}^L(\mc{D}^\mc{L})$ by 
\[
\mc{V}\mapsto (\mc{L} \otimes_{\mc{O}_X} \mc{D}) \otimes_\mc{D} \mc{V}
\]
for $\mc{V} \in \mc{M}_{qc}^L(\mc{D})$. The twist functor is an equivalence of categories. 

For an abelian category $\mc{C}$, we use the notation $D(\mc{C})$ and $D^b(\mc{C})$ to refer to the derived category and bounded derived category of $\mc{C}$, respectively. We identify $\mc{C}$ with its image in $D(\mc{C})$ (resp. $D^b(\mc{C})$) under the natural embedding.

For a morphism $f: Y \rightarrow X$ of smooth algebraic varieties and a twisted sheaf $\mc{D}$ of differential operators on $X$, we define the \emph{inverse image functor} $f^+:\mc{M}_{qc}^L(\mc{D}) \rightarrow \mc{M}_{qc}^L(\mc{D}^f)$ by 
\[
f^+(\mc{V})=\mc{D}_{Y \rightarrow X} \otimes_{f^{-1}\mc{D}}f^{-1}\mc{V}
\]
for $\mc{V} \in \mc{M}_{qc}^L(\mc{D})$. In general $f^+$ is right exact with left derived functor $Lf^+$. If $f$ is an open immersion, then $f^+$ is exact and $f^+(\mc{V}) = \mc{V}|_Y$. If $f$ is a submersion, then $f^+$ is exact. We define the \emph{extraordinary inverse image functor} $f^!:D^b(\mc{M}_{qc}^L(\mc{D}))\rightarrow D^b(\mc{M}_{qc}^L(\mc{D}^f))$ by 
\[
f^!=Lf^+\circ [\text{dim}Y-\text{dim}X]. 
\] 
If $f$ is an immersion then $f^!$ is the right derived functor of the left exact functor $L^{\dim Y - \dim X} f^+:\mc{M}_{qc}^L(\mc{D})\rightarrow \mc{M}_{qc}^L(\mc{D}^f)$. In this setting, we refer to the functor $L^{\dim Y - \dim X} f^+$ as $f^!$, and for $\mc{V} \in \mc{M}_{qc}(\mc{D})$, we refer to the $k^{th}$-cohomology modules $H^kf^!(\mc{V})$ as $R^kf^!(\mc{V})$. 

We define the \emph{direct image functor} $f_+:D^b(\mc{M}_{qc}^R(\mc{D}^f)) \rightarrow D^b(\mc{M}_{qc}^R(\mc{D}))$ by
\[
f_+(\mc{W}^\cdot) = Rf_\bullet(\mc{W}^\cdot \otimes ^L_{\mc{D}^f} \mc{D}_{Y \rightarrow X}), 
\]
for $\mc{W}^\cdot \in D^b(\mc{M}^R(\mc{D}^f))$. Here $Rf_\bullet$ is the right derived functor of the sheaf-theoretic direct image functor $f_\bullet$. If $f$ is an immersion, $f_+$ is the right derived functor of the left exact functor $H^0 \circ f_+ \circ D: \mc{M}_{qc}^R(\mc{D}^f) \rightarrow \mc{M}_{qc}^R(\mc{D})$, where $D$ is the natural embedding of $\mc{M}_{qc}^R(\mc{D}^f)$ into the derived category $D(\mc{M}_{qc}^R(\mc{D}^f))$. In this setting, we refer to $H^0 \circ f_+ \circ D$ by $f_+$. If $f$ is an open immersion, then $f_+=Rf_\bullet$ is the sheaf-theoretic direct image. If $f$ is affine, then $f_+$ is exact. 

The relationship between the twist functor and the direct image functor is the following. 
\begin{proposition}
\label{Projection Formula}
(Projection Formula) Let $f:Y \rightarrow X$ be a morphism of smooth complex algebraic varieties, $\mc{D}$ a twisted sheaf of differential operators on $X$, and $\mc{L}$ be an invertible $\mc{O}_X$-module. Then the following diagram commutes. 
\begin{center}
\begin{tikzcd}
D(\mc{M}(\mc{D}^f)) \arrow[rightarrow]{r}{f_+}\arrow[rightarrow,swap]{d}{f^*(\mc{L}) \otimes_{\mc{O}_Y} -} 
  & D(\mc{M}(\mc{D}))\arrow[rightarrow]{d}{\mc{L}\otimes_{\mc{O}_X}-} \\
D(\mc{M}((\mc{D}^\mc{L})^f)) \arrow[rightarrow]{r}{f_+}  
  &D(\mc{M}(\mc{D}^\mc{L}))
\end{tikzcd}
\end{center}
\end{proposition}

For a module $\mc{V} \in \mc{M}^R_{qc}(\mc{D})$, and a smooth subvariety $Y \subset X$, denote by $\Gamma_Y(\mc{V})$ the $\mc{D}$-module of local sections of $\mc{V}$ supported in $Y$. The functor $\Gamma_Y:\mc{M}^R_{qc}(\mc{D})\rightarrow \mc{M}^R_{qc}(\mc{D})$ is a left-exact functor, and we denote by $R\Gamma_Y:D^b(\mc{M}_{qc}^R(\mc{D}))\rightarrow D^b(\mc{M}_{qc}^R(\mc{D}))$ its right derived functor. The following equivalence of categories is very useful in computations.
\begin{theorem} (Kashiwara)
\label{kashiwara}
If $Y$ is a closed smooth subvariety of a smooth algebraic variety $X$, $i:Y \rightarrow X$ the natural immersion, and $\mc{D}$ a twisted sheaf of differential operators on $X$, then the functor 
\[
i_+:\mc{M}_{qc}^R(\mc{D}^i) \rightarrow \mc{M}^R_{qc}(\mc{D})
\]
establishes an equivalence of categories between $\mc{M}^R_{qc}(\mc{D}^i)$ and the full subcategory $\mc{M}^R_{qc,Y}(\mc{D})$ of $\mc{M}_{qc}(\mc{D})$ consisting of modules supported in $Y$. The quasiinverse of $i_+$ is $i^!$. In particular, if $\mc{V}$ is a quasicoherent $\mc{D}^i$-module, then $i^!(i_+(\mc{V}))=\mc{V}$, and if $\mc{U}$ is a quasicoherent $\mc{D}$-module, then $i_+(i^!(\mc{U}))=\Gamma_Y(\mc{U})$. 
\end{theorem}

Let $i:Y \rightarrow X$ be the immersion of a closed subvariety. If $\mc{J}_Y$ is the ideal of $\mc{O}_X$ consisting of germs vanishing on $Y$, we can define an increasing filtration of $\mc{D}_{Y \rightarrow X}$ by (left $\mc{D}^i$, right $i^{-1}\mc{O}_X$)-modules by 
\[
F_p\mc{D}_{Y \rightarrow X}  =\{T \in \mc{D}_{Y \rightarrow X}|T \varphi = 0 \text{ for } \varphi \in (\mc{J}_Y)^{p+1}\},
\]
for $p \in \Z_+$. We call this filtration the filtration by \textit{normal degree}.  By Kashiwara's theorem, it induces a natural $\mc{O}_X$-module filtration on $\mc{D}$-modules supported on $Y$. Namely, if $\mc{W} \in \mc{M}^R_{qc}(\mc{D}^i)$, 
\[
F_pi_+(\mc{W})=i_\bullet(\mc{W} \otimes_{\mc{D}^i} F_p\mc{D}_{Y \rightarrow X}). 
\]
The associated graded module has the form
\begin{equation}
\label{grading by normal degree}
Gr i_+(\mc{W})=i_\bullet(\mc{W} \otimes_{\mc{O}_Y} S(\mc{N}_{X|Y})),
\end{equation}
where $\mc{N}_{X|Y} = i^*(\mc{T}_X)/\mc{T}_Y$ denotes the normal sheaf of $Y$, and $S(\mc{N}_{X|Y})$ is the corresponding sheaf of symmetric algebras \cite[App. A \S 3.3]{HMSWI}. 

The interaction between $\mc{D}$-module functors and fiber products is captured by base change. 
\begin{theorem}
\label{basechange}
(Base Change Formula) Let $f: X \rightarrow Z$ and $g: Y \rightarrow Z$ be morphisms of smooth complex algebraic varieties such that the fiber product $X \times_Z Y$ is a smooth algebraic variety, and let $\mc{D}$ be a twisted sheaf of differential operators on $Z$. Then the commutative diagram 
\begin{center}
\begin{tikzcd}
X \times_Z Y \arrow[rightarrow]{r}{q}\arrow[rightarrow,swap]{d}{p} 
  & Y\arrow[rightarrow]{d}{g} \\
X \arrow[rightarrow]{r}{f}  
  &Z
\end{tikzcd}
\end{center}
determines an isomorphism 
\[
g^! \circ f_+ = q_+ \circ p^!
\]
of functors from $D^b(\mc{M}(\mc{D}^f))$ to $D^b(\mc{M}(\mc{D}^g))$. 
\end{theorem}


\subsection{Beilinson--Bernstein localization}
\label{Beilinson-Bernstein localization}

A key ingredient in this story is the localization theory of Beilinson and Bernstein, which we briefly review here. Full details can be found in \cite{BB,localization}. For the remainder of this appendix, let $\mf{g}$ be a complex reductive Lie algebra, $\mf{h}$ the abstract Cartan subalgebra of $\mf{g}$ \cite[\S 2]{D-modules}, and $X$ the flag variety of $\mf{g}$. Fix $\lambda \in \mf{h}^*$, and let $\theta$ be the Weyl group orbit of $\lambda$ in $\mf{h}^*$.  In \cite{BB}, Beilinson and Bernstein construct a twisted sheaf of differential operators $\mc{D}_\lambda$ on $X$ for each $\lambda \in \mf{h}^*$. (In the notation of Section \ref{Twisted sheaves of differential operators}, $\mc{D}_\lambda = \mc{D}_{X, \lambda + \rho}$.) They show that for any $\mu$ in the Weyl group orbit $\theta$ of $\lambda$, the global sections $\Gamma(X, \mc{D}_\mu)$ of $\mc{D}_\mu$ are equal to $\mc{U}_\theta$, which is the quotient of $\ug$ by the ideal in $\zg$ corresponding to $\theta$ under the Harish-Chandra homomorphism. This implies that the global sections functor $\Gamma$ maps quasicoherent $\mc{D}_\lambda$-modules into $\ug$-modules with infinitesimal character $\chi_\lambda$; that is, there is a left exact functor 
\[
\Gamma: \mc{M}_{qc}(\mc{D}_\lambda) \rightarrow \mc{M}(\mc{U}_\theta).
\]
Beilinson and Bernstein define a \emph{localization functor} 
\[
\Delta_\lambda: \mc{M}(\mc{U}_\theta) \rightarrow \mc{M}_{qc}(\mc{D}_\lambda)
\]
by $\Delta_\lambda(V) = \mc{D}_\lambda \otimes_{\mc{U}_\theta}V$ for $V \in \mc{M}(\mc{U}_\theta)$. The localization functor is right exact and is a left adjoint to $\Gamma$. In \cite{BB} it is shown that for antidominant regular $\lambda \in \mf{h}^*$, $\Delta_\lambda$ is an equivalence of categories, and its quasi-inverse is $\Gamma$.


\subsection{Translation functors}
\label{Translation functors}

Fix $\lambda \in \mf{h}^*$, and let $\mc{D}_\lambda$ be the corresponding homogeneous twisted sheaf of differential operators. Any $\mu$ in the weight lattice $P(\Sigma) = \{\lambda \in \mf{h}^* | \alpha^\vee(\lambda) \in \Z \text{ for all } \alpha \in \Sigma \}$  naturally determines a $G=\Int{\mf{g}}$-equivariant invertible $\mc{O}_X$-module $\mc{O}(\mu)$ on $X$. Twisting by $\mc{O}(\mu)$ defines a functor 
\[
-(\mu): \mc{M}(\mc{D}_\lambda) \rightarrow \mc{M}(\mc{D}_{\lambda + \mu})
\]
by $\mc{V}(\mu) = \mc{O}(\mu) \otimes_{\mc{O}_X} \mc{V}$ for $\mc{V} \in \mc{M}(\mc{D}_\lambda)$. We call this functor the \textit{geometric translation functor}. It is evidently an equivalence of categories, and it also induces an equivalence of categories on $\mc{M}_{qc}(\mc{D}_\lambda)$ (resp. $\mc{M}_{coh}(\mc{D}_\lambda)$) with $\mc{M}_{qc}(\mc{D}_{\lambda+ \mu})$ (resp. $\mc{M}_{coh}(\mc{D}_{\lambda+\mu})$).

\bibliographystyle{alpha}
\bibliography{RomanovWhittakerModules}
\end{document}